 \documentclass[preprint,11pt]{siamart1116}
 \include{siamplain.bst}

 \pdfoutput=1
 
 \usepackage[sort&compress,numbers]{natbib}
 
\usepackage{color}
 \usepackage{amssymb,amsmath}

\usepackage{calc}
 \usepackage[margin=1.in]{geometry}

\usepackage{graphicx}
\usepackage{multirow}

\usepackage{fancyhdr}
\usepackage{time}
 \newcommand{\xv}{{\bf x}}
\newcommand{\kv}{{\bf k}}

\newcommand{\jv}{{\bf j}}

\newcommand{\phiv}{\boldsymbol{\phi}}
\newcommand{\psiv}{\boldsymbol{\psi}}

\newcommand{\Ej}{\mathbf{E}_{\jv}}
\newcommand{\psij}{\mathbf{\psiv}_{\jv}}
\newcommand{\Psij}{\mathbf{\Psi}_{\jv}}

\newcommand{\phij}{\mathbf{\phiv}_{\jv}}
\newcommand{\Phij}{\mathbf{\Phi}_{\jv}}

\newcommand{\dt}{\Delta t}

\newcommand{\Ejn}{\Ej^n}
\newcommand{\psijn}{\psij^n}
\newcommand{\phijn}{\phij^n}

\newcommand{\coeff}[1]{C_{#1}}

\newcommand{\mystrutS}{\rule[-.15in]{0in}{.375in}}
\newcommand{\mystrutD}{\rule[-.275in]{0in}{.6125in}}

\newcommand{\LapFour}{L_{4h}}
\newcommand{\LapTwo}{L_{2h}}

\newtheorem{myproposition}{Proposition}

 \usepackage{color}

  \numberwithin{myproposition}{section}

\title{High-Order Accurate FDTD Schemes for Dispersive Maxwell's Equations in Second-Order Form Using Recursive Convolutions%
\thanks{Preprint.
\funding{This research was supported by a U.S. Presidential Early Career Award for Scientists and Engineers. This work was also supported, in part, through the NSF Research Training Groups program (DMS-1344962) and was partially performed under DOE contracts from the ASCR Applied Math Program. This work was partially funded by the DARPA Defense Sciences Office.}}}

\author{Michael J. Jenkinson%
			\thanks{Department of Mathematical Sciences, Rensselaer Polytechnic Institute, Troy, NY USA (\email{jenkim2@rpi.edu}, \email{banksj3@rpi.edu})}
			\and
 		  Jeffrey W. Banks\footnotemark[2] %
			}

\begin{document}

\maketitle

\begin{abstract}
We propose a novel finite-difference time-domain (FDTD) scheme for the solution of the Maxwell's equations in which linear dispersive effects are present. The method uses high-order accurate approximations in space and
time for the dispersive Maxwell's equations written as a second-order vector wave equation with a time-history convolution term. The modified equation approach is combined with the recursive convolution (RC) method to develop high-order approximations accurate to any desired order in space and time.
High-order-accurate centered approximations of the physical Maxwell interface conditions are derived for the dispersive setting in order to fully restore accuracy at discontinuous material interfaces. Second- and fourth-order accurate versions of the scheme are presented and implemented in two spatial dimensions for the case of the Drude linear dispersion model. The stability of these schemes is analyzed. Finally, our approach is also amenable to curvilinear numerical grids if used with appropriate generalized Laplace operator.
\end{abstract}



\section{Introduction}

In many materials, such as metal or biological tissue, lossy and dispersive effects resulting from the electronic response of the material have a considerable effect on electromagnetic wave propagation and must be included in the model. These materials are referred to as ``dispersive" media. Simple models of dispersive optical media often use a classical linear Newtonian description to describe the response of charge carriers to an external electric field. Three common examples are the Debye, Lorentz, and Drude models, all of which can be generalized using the Pad\'e approximant method \cite{prokopeva_2011,prokopeva_2011_2}.

Accurate time-domain simulation of dispersive optical effects is important for various applications in optical computing, imaging, and sensing, where transient electromagnetic waves play an important role in the system of interest. Furthermore, these applications often involve multiple materials that meet along a common interface, which results in a situation with discontinuous material parameters and the possibility of important interface phenomenon. For example, the propagating surface waves (``surface plasmon polaritons") at interfaces between dielectrics and metals, made possible by the dispersive nature of the metal, are of considerable interest in the field of plasmonics, and dispersion plays an important role in their excitation and attenuation \cite{maier_2007,pitarke_2007}. It is therefore desirable to derive high-order accurate time-domain solvers to simulate dispersive electromagnetic wave propagation and have the capability to treat material interfaces.

Perhaps the most common time-domain scheme is the ubiquitous second-order accurate Yee scheme \cite{yee_1966,taflove_2005}. This finite-difference time-domain (FDTD) solver for computational electromagnetics  solves the first-order form of Maxwell's equations on a grid which is staggered in time and space. The scheme is attractive for many reasons including its efficiency and conservation of a discrete energy. In addition, linear-dispersive-material physics can be included into a traditional Yee-type schemes in at least two ways; the auxiliary differential equation (ADE) method, and recursive convolution (RC). The ADE approach couples Maxwell's equations to the ordinary differential oscillator equation derived from the classical description discussed above~\cite{kashiwa_1990,kashiwa_1990_2,kashiwa_1990_3,joseph_1991,gandhi_1993,chun_2013}. This equation is then approximated using similar discretization rules as used to treat the governing PDEs. 
On the other hand, the RC approach directly incorporates the solution to the ordinary differential oscillator equation into Maxwell's equations in the form of a time-delay convolution of the electric field, which is subsequently approximated  using a discrete sum through numerical quadrature~\cite{luebbers_1990,luebbers_1991,luebbers_1992,kelley_1996,chun_2013}.
Rather than storing the full time-history of the electric field, this term may be updated recursively at each time step due to the exponential kernel of the convolution. Higher-order accurate methods using the ADE approach have also been developed \cite{young_1995,young_1996,young_1997,prokopidis_2004,prokopidis_2006}, while implementations of the RC approach are limited to second-order accuracy in time due to the choice of convolution discretization. Various other approaches exist for the computational simulation of dispersive optical media in the time domain. An FDTD approach for both dispersive and nonlinear electromagnetic effects was developed using the Z-transform \cite{sullivan_1992}, the implementation of which has many similarities with the RC and ADE methods. 



The presence of material interfaces complicates numerical simulation since, in general, solutions are not smooth where physical parameters are discontinuous. Unless this irregularity is properly treated in the numerical method, difference approximations can suffer from errors associated with inaccuracies in the interface equations, which usually limits the overall accuracy to first-order at best. This is a well-known problem in computational wave equation literature \cite{ditkowski_2001,kreiss_2006}, and is considerably more difficult to avoid when the interface does not align with the numerical grid, as is often the case with traditional staggered cartesian grids. Successful attempts to reduce this error have been made in both the non-dispersive \cite{mohammadi_2005} and dispersive \cite{popovic_2003,deinega_2007,liu_2012} setting by smoothing the material parameters in the numerical cells adjacent to curved interfaces. In the event that high-order accuracy near material interfaces is required, or smoothing of the material parameters cannot be tolerated for other reasons, physical jump conditions can be incorporated directly into the discretization scheme as numerical compatibility conditions. This is often referred to as the derivative matching or matched interface and boundary (MIB) method \cite{zhao_2004,henshaw_2006,zhang_2016}. Note that this general approach may also be used at (curved) exterior boundaries, e.g. perfect conductors, to retain full accuracy. For example in \cite{henshaw_2006}, Henshaw presented a fourth-order accurate scheme which solved Maxwell's equations as a second-order vector wave equation on curvilinear and overlapping grids. Most recently, the MIB approach has been adapted for several linear optical dispersion models at curved interfaces using a hybrid formulation of dispersive Maxwell's equations \cite{nguyen_2014,nguyen_2014_2,nguyen_2015,nguyen_2016}.

 
In addition to FDTD schemes, there exists a wide range of finite element methods for the electromagnetic wave problems \cite{nedelec_1980,jin_1993,rodrigue_2001}. Various finite-element time-domain (FETD) and discontinuous Galerkin (DGTD) schemes also exist for the solution of the dispersive Maxwell's equations in both first- \cite{lu_2004,gedney_2012} and second-order form \cite{jiao_2001,li_2006,banks_2009}. These typically implement the dispersive material effects using either the RC or ADE approach, where again ADE method is used to achieve high-order accuracy than the second-order accurate RC discretization \cite{gedney_2012}. By choosing the appropriate numerical flux between elements, these methods may be designed to avoid the aforementioned loss of accuracy at material discontinuities \cite{lu_2004,li_2008,zhang_2016}.

 In this paper, we propose a numerical scheme for the solution of the time-dependent Maxwell's equations with linear dispersion which can achieve arbitrary desired order of accuracy in time and space. Following the approach in \cite{henshaw_2006} for the non-dispersive case, the dispersive Maxwell's equations are expressed as a vector wave equation with a time-history convolution term. A modified-equation time stepping approach, is used to obtain a high-order accurate approximation of the equation through a temporal Taylor series. 
 Consecutive terms in the expansion require additional time-history convolution terms, which are derived from the original temporal convolution representing the formal solution to the ordinary differential oscillator equation. The integral terms are approximated using quadrature rules at the appropriate order of accuracy in time, and subsequently transformed into a recursive convolution formulation. In order to preserve accuracy of the scheme at the interface, the governing PDEs are used to derive high order centered approximations of the physical Maxwell interface conditions. Note that the proposed scheme is amenable to discretization on the curvilinear overlapping grids and high order accurate symmetric discretizations of the generalized Laplace operator, as in for example~\cite{henshaw_2006}.
 
In section \ref{section:governing_equations}, we review the most common linear dispersion models and the way in which they couple with the macroscopic form of Maxwell's equations. A vector wave equation formulation of the dispersive Maxwell's equations is then derived. Section \ref{section:scheme} describes our proposed numerical scheme for the general case, and explicitly presents the second- and fourth-order accurate versions for the Drude model. An analysis of the numerical stability of these two schemes for the Cauchy problem is also presented. In particular, we discuss a subtle aspect of the stability for the second-order code which admits weak exponential growth of the numerical solution. On the other hand, the fourth-order accurate version admits no such spurious growth, and so it is superior from both and accuracy and stability perspective. In Section~\ref{sec:interface}, the discrete treatment of interface conditions is discussed including explicit formulations for the afore mentioned Drude model and the second- and fourth-order accurate schemes. Finally in Section~\ref{section:six}, numerical results are presented which confirm the theoretical predictions as well as illustrate the spurious growth in the second-order accurate scheme and lack thereof in the fourth-order accurate scheme.


\section{Governing Equations}  \label{section:governing_equations}

This work considers time-dependent electromagnetic wave propagation without external forcing as governed by Maxwell's equations\begin{subequations}
\begin{align}
  \nabla \times \mathbf{E} &  = - \partial_t \mathbf{B}, \label{eqn:macro1}\\
  \nabla \cdot \mathbf{D} & = 0,  \label{eqn:macro2} \\
  \nabla \times \mathbf{H} & =  \partial_t \mathbf{D},\label{eqn:macro3} \\
  \nabla \cdot \mathbf{B} & = 0. \label{eqn:macro4}
\end{align}
\end{subequations}
Here, $t$ is time, $\xv$ is the independent (vector) spatial coordinate, $\mathbf{E} = \mathbf{E}(\xv,t)$ is the electric field, $\mathbf{D} = \mathbf{D}(\xv,t) $ is the displacement field, $\mathbf{B} = \mathbf{B}(\xv,t)$ is the magnetic field, and $\mathbf{H} = \mathbf{H} (\xv,t)$ is the magnetizing field. In the scope of this article, we consider materials which are linear and non-dispersive with respect to magnetization\footnote{Note that this assumption is not critical to the basic approach and schemes could be extended to magnetically dispersive materials as well.} and so
\begin{align}
  \mathbf{B} = \mu \mathbf{H},  \label{eqn:HtoB}
\end{align}
where $\mu  = \mu(\xv) > 0$ is the permeability. Here we focus on the situation of electromagnetic wave propagation in a domain consisting of distinct materials which are separated by material interfaces, and as such, the permeability is assumed to be piecewise constant.

Consider linear, isotropic, and (temporally) dispersive media in which the electric permittivity depends continuously on the temporal frequency, $\omega$, of the electromagnetic fields. Such materials may be modeled in the frequency (Fourier) domain via the linear constitutive relation 
\begin{align}
  \widehat{\mathbf{D}}(\xv,\omega) = \widehat{\epsilon}( \omega) \widehat{\mathbf{E}}(\xv,\omega) =  \left( \epsilon_0 \epsilon_r + \widehat{\chi}(\omega)\right)  \widehat{\mathbf{E}} (\xv, \omega). \label{eqn:frequencydomain}
\end{align}
Here, $\widehat{\epsilon}(\omega)$ is the total frequency-dependent permittivity, $\epsilon_0$ is the permittivity of free space, $\epsilon_r$ is the high-frequency relative permittivity constant of the medium, and $ \widehat{\chi}(\omega) \equiv \widehat{\epsilon}(\omega)  - \epsilon_0 \epsilon_r $ is the electric susceptibility. The latter is defined out of notational convenience for the following discussion. Both $\widehat{\epsilon}(\omega)$ and $\widehat{\chi}(\omega)$ are typically obtained from a classical linear description of electron motion in the material. Note also that it is common to find a definition of the polarizability $\widehat{\mathbf{P}} \equiv \widehat{\chi}(\omega)  \widehat{\mathbf{E}} (\omega) $, but we do not make use of this quantity and make no further mention of it here. In the time-domain, equation \eqref{eqn:frequencydomain} can be written
\begin{align}
  \mathbf{D} (\xv,t) & = \int_{0}^{\infty} \epsilon (  \tau) \mathbf{E} (\xv,t - \tau) d \tau  \label{eqn:epsilon}  \\
  & =  \epsilon_0 \epsilon_r \mathbf{E} (\xv,t) + \int_{0}^{\infty} \chi (  \tau) \mathbf{E} (\xv,t - \tau) d \tau.  \nonumber  
\end{align}
Here, $\epsilon ( \tau)$ and $\chi ( \tau )$ are the permittivity and the susceptibility kernels respectively, and are obtained via the inverse Fourier transform from $\widehat{\epsilon}(\omega)$ and $\widehat{\chi}(\omega)$. Equation \eqref{eqn:epsilon} may be interpreted as a description of the non-instantaneous delayed response of the material to the applied electric field. A summary of some common linear dispersion models and their associated susceptibilities is found in table \ref{table:dispersionmodels}. Note that while the convolution bounds for the traditional Fourier transform are $\tau \in \mathbb{R}$, the bounds of the convolution in \eqref{eqn:epsilon} are necessarily restricted to past values of the electric field, ruling out any non-causal effects. This is reflected in the Fourier transform by the appearance of the Heaviside function $\Theta(\tau)$ in the susceptibility kernel $\chi(\tau)$.

\begin{table}[h]\begin{center}\begin{tabular}{|c||c|c|}
\hline
Dispersion Model & Susceptibility $\widehat{\chi}(\omega)$ & Susceptibility Kernel  $\chi(\tau)$    \\
\hline\hline
 \mystrutS Debye &   $\frac{ \gamma \epsilon_0 ( \epsilon_s -  \epsilon_r ) }{\gamma + i \omega } $	& $ \epsilon_0 ( \epsilon_s - \epsilon_r ) \gamma e^{- \gamma \tau}	\Theta( \tau )  $  \\
 \hline
 \mystrutS Lorentz & $ \frac{ \epsilon_0 \omega_p^2 }{ ( \omega_0^2 - \omega^2) + i \gamma \omega} $	& $ \epsilon_0 \omega_p^2 \frac{ \sin \left( \sqrt{ \omega_p^2 - \gamma^2 / 4} \tau \right) }{ \sqrt{ \omega_p^2 - 4 \gamma^2 / 4} }  e^{- \gamma \tau / 2} \Theta(\tau) $	  \\
\hline
 \mystrutS Drude & $ - \frac{ \epsilon_0 \omega_p^2}{ \omega^2 - i \gamma \omega} $ & 	 $\frac{ \epsilon_0 \omega_p^2}{\gamma} \left( 1 - e^{ - \gamma \tau} \right) \Theta(\tau)	$   \\
 \hline
 \mystrutS Generalized & 
   \begin{tabular}{c}
     \footnotesize $ \begin{aligned} & \epsilon_0 c_1 \bigg( \frac{ e^{i c_2} }{ c_3 - ( \omega + i c_4 )} \\ & \qquad  \quad +   \frac{ e^{- i c_2} }{  c_3 + ( \omega + i  c_4 )} \bigg) \end{aligned} $
   \end{tabular} & 
   $ \epsilon_0 c_1 e^{- c_4 \tau} \sin \left( c_3 \tau - c_2 \right) \Theta(\tau) $  \\
 \hline\end{tabular}
\caption{Classical linear dispersion models \cite{prokopeva_2011,prokopeva_2011_2}. Here, $\delta(\tau)$ is the Dirac delta function, $\Theta(\tau)$ is the Heaviside function, $\epsilon_0$ is the permittivity of free space, $\epsilon_s$ is the static permittivity, $\epsilon_r$ is the high frequency permittivity, $\omega_p$ is the electron plasma frequency, and $\gamma$ is the damping coefficient. Note that in order to derive the Lorentz susceptibility kernel, one requires $\gamma < \omega_p / 2$, which is a physically realistic assumption for common Lorentz media. In the generalized model, Pad\'{e} Approximants are used to fit the constant parameters $c_1, c_2, c_3$ and $c_4$ to experimental data. Typically, several of these terms are added together to capture each of the optical resonances of the medium.}
\label{table:dispersionmodels}\end{center}\end{table}

Equations \eqref{eqn:macro1} through \eqref{eqn:macro4} along with \eqref{eqn:epsilon}, fully describe the propagation of electromagnetic waves in linear, isotropic, and (temporally) dispersive media. This formulation can be simplified to a description of the electric field $\mathbf{E} (\xv,t)$ in isolation as follows. First, working in the frequency domain, \eqref{eqn:macro2} can be combined with \eqref{eqn:frequencydomain} to yield
\begin{align}
  \nabla \cdot \widehat{\mathbf{D}}(\xv, \omega)&  = 0, \nonumber \\
  \nabla \cdot \widehat{\epsilon}(\omega) \widehat{\mathbf{E}}(\xv, \omega) &  = 0, \nonumber \\
  \widehat{\epsilon}(\omega) \nabla \cdot \widehat{\mathbf{E}}(\xv, \omega) &  = 0, \nonumber
\end{align}
from whence we obtain the constraint
\begin{align}
  \nabla \cdot \mathbf{E} (\xv, t) = 0. \label{eqn:gauss}
\end{align}
Now taking the time derivative of \eqref{eqn:macro3}, and making use of relations \eqref{eqn:HtoB} and \eqref{eqn:gauss} gives
\begin{align}
\begin{array}{ll} 
  \mu \partial_t^2 \mathbf{D} & = \nabla \times \partial_t \mathbf{B}, \\
  \mu \partial_t^2 \mathbf{D} & = -\nabla \times \nabla \times \mathbf{E}, \\
  \mu \partial_t^2 \mathbf{D} & = \Delta \mathbf{E} - \nabla ( \nabla \cdot \mathbf{E} ),  \\
  \mu \partial_t^2 \mathbf{D} & = \Delta \mathbf{E} . 
  \end{array}
  \label{eqn:vectorid}
\end{align}
The quantity $\partial_t^2 \mathbf{D}$ can be expressed in terms of the electric field $\mathbf{E}$ by using \eqref{eqn:epsilon}. For convenience, we work in the frequency domain to obtain
\begin{align}
  - \omega^2 \widehat{\mathbf{D}}(\xv, \omega) = 
  - \omega^2 \ \epsilon_0 \epsilon_r \widehat{\mathbf{E}}(\xv, \omega) - \omega^2 \widehat{\chi}(\omega) \widehat{\mathbf{E}}(\xv,t). \label{eqn:Dttfourier}
\end{align}
Now by defining
\begin{align}
  \widehat{\eta}(\omega) \equiv  \omega^2 \widehat{\chi}(\omega) , \label{eqn:etadef}
\end{align}
Equation \eqref{eqn:Dttfourier} can be expressed in the time domain as 
\begin{align}
  \frac{ \partial^2 \mathbf{D}}{ \partial t^2}  = \epsilon_0 \epsilon_r \partial_t^2  \mathbf{E} - \int_0^{\infty} \eta(\tau) \mathbf{E}(t - \tau) d \tau, \label{eqn:Dtt}
\end{align}
via the inverse Fourier transform. The respective functions $\widehat{\eta}(\omega)$ and $\eta(\tau)$ for each of the dispersion models shown in Table~\ref{table:dispersionmodels} are given in Table \ref{table:etas}. Finally, equations \eqref{eqn:vectorid} and \eqref{eqn:Dtt} give the dispersive and dissipative wave equation
\begin{align}
 \frac{ \partial^2 \mathbf{E}}{ \partial t^2}  = \frac{1}{\mu \epsilon_0 \epsilon_r} \Delta \mathbf{E} + \frac{1}{ \epsilon_0 \epsilon_r} \int_0^{\infty} \eta(\tau) \mathbf{E}(t - \tau) d \tau , \label{eqn:waveE1}
\end{align}
or alternatively, 
\begin{align}
 \frac{ \partial^2 \mathbf{E} }{ \partial t^2}  = \frac{1}{\mu \epsilon_0 \epsilon_r} \Delta \mathbf{E} + \frac{1}{ \epsilon_0 \epsilon_r} \eta *  \mathbf{E} =  \frac{1}{\mu \epsilon_0 \epsilon_r}  \left( \Delta + \mu  \eta \  * \right) \mathbf{E}. \label{eqn:waveE2}
\end{align}
Finally, a complete initial boundary value problem can be formulated for the domain $\xv \in \Omega$ with boundary $\xv \in \partial \Omega$ as
\begin{subequations}
\begin{align}
& \partial_t^2 \mathbf{E}(\xv,t)  =  \frac{1}{\mu \epsilon_0 \epsilon_r}  \left( \Delta + \mu  \eta \  * \right) \mathbf{E}(\xv,t) , \label{eqn:governingE} \\
 & \mathbf{E}(\xv,0) = \mathbf{f}_1(\xv), \qquad \partial_t \mathbf{E}(\xv,0) = \mathbf{f}_2(\xv), \qquad \eta *  \mathbf{E}(\xv,0) = \mathbf{f}_3(\xv),  \\
 & \mathcal{B}\left[\mathbf{E}(\xv,t)\right] = \mathbf{g}(\xv,t), \qquad \xv \in \delta \Omega  \label{eqn:governingE3} 
\end{align}
\label{eqns:governing}
\end{subequations}where $\mathcal{B}[\cdot]$ is a generic boundary operator that can be used to specify perfect conductance, far-field conditions, or given fields for example.

\begin{table}[h]
\begin{center}
\begin{tabular}{|c||c|c|}
\hline
\begin{tabular}{c} Dispersion \\ Model \end{tabular} &   $\widehat{\eta}(\omega)$ &  $\eta(\tau)$  \\
\hline\hline
 \mystrutS Debye  	&    $ -  \frac{ \gamma \epsilon_0 ( \epsilon_s -  \epsilon_r ) \omega^2 }{\gamma + i \omega } $  & \footnotesize$ \begin{aligned} & \epsilon_0 \gamma ( \epsilon_s -  \epsilon_r ) \big[  \delta'(\tau) + \gamma \delta(\tau) \\
&  \qquad \qquad \qquad - \gamma^2 e^{ - \gamma \tau} \Theta(\tau)  \big] \end{aligned} $ \\
 \hline
 \mystrutS Lorentz 	&  $ -   \frac{ \epsilon_0 \omega_p^2  \omega^2}{ ( \omega_0^2 - \omega^2) + i \gamma \omega} $	 &  \footnotesize     \begin{tabular}{c} $
  \epsilon_0 \omega_p^2 \big[ \delta(\tau) + \frac{ e^{- \gamma \tau / 2}}{ 4 a}
  \big(  - 4 \gamma a \cos( a \tau) $ \smallskip \\
  	$ + ( \gamma^2 - 4 a^2 ) \sin( a \tau) \big)  \Theta(\tau) \big]  $  
	\end{tabular} \\
\hline
 \mystrutS Drude & $  - \frac{ \epsilon_0 \omega_p^2 \omega }{ \omega - i \gamma } $ &  \footnotesize $ \epsilon_0 \omega_p^2 \left( - \delta(\tau)  + \gamma e^{ - \gamma \tau}  \Theta(\tau)  \right) $ \\
 \hline
 \mystrutD Generalized &   \footnotesize $ \begin{aligned} &  - \epsilon_0    c_1 \bigg( \frac{ \omega^2 e^{i c_2} }{ c_3 - ( \omega + i c_4 )} \\ &  \hspace{1.5cm} +   \frac{ \omega^2 e^{- i c_2} }{  c_3 + ( \omega + i  c_4 )} \bigg) \end{aligned} $ &  \footnotesize
   \begin{tabular}{c}
      $ \epsilon_0 c_1 \big[ \left( c_3 \cos ( c_2) + c_4 \sin( c_2) \right) \delta(\tau)$ \smallskip \\ 
      $ + e^{- c_4 \tau} \big( - c_3 c_4 \cos( c_3 \tau - c_2 )  $ \smallskip \\  
      $ + ( c_4^2 - c_3^2) \sin( c_3 \tau - c_2 ) \big) \Theta(\tau) \big]  $
   \end{tabular} 
 \\ \hline
\end{tabular}
\caption{ Auxiliary functions $\widehat{\eta}(\omega)$ and $\eta(\tau)$, defined in \eqref{eqn:etadef}, for the classical linear dispersion models from table \ref{table:dispersionmodels} where all of the relevant constant parameters are defined. In the Lorentz model, $a \equiv  \sqrt{ \omega_p^2 -  \gamma^2 / 4}$. }
\label{table:etas}
\end{center}
\end{table}

As discussed earlier, our primary interest is in situations with piecewise-constant materials. In this situation, the interface between two materials plays a key role, and its treatment is one focus of the current manuscript. Maxwell's equations \eqref{eqn:macro1} through \eqref{eqn:macro4} may be used to derive jump conditions at interfaces where material parameters are discontinuous. Consider an interface $\mathcal{I}$ between two material domains. Integrating each of \eqref{eqn:macro1} through \eqref{eqn:macro4} over an appropriate control volume spanning the interface gives the following physical interface jump conditions:
\begin{subequations}
\begin{align}
\left[ \mathbf{n} \times \mathbf{E} \right]_{\mathcal{I}} & = 0,  \label{eqn:jumps1} \\
\left[ \mathbf{n} \cdot \mathbf{D} \right]_{\mathcal{I}} & = 0,   \label{eqn:jumps2} \\
\left[ \mathbf{n} \times \mathbf{H} \right]_{\mathcal{I}} & = 0, \label{eqn:jumps3} \\
\left[ \mathbf{n} \cdot \mathbf{B} \right]_{\mathcal{I}} = \left[ \mathbf{n} \cdot \mu \mathbf{H} \right]_{\mathcal{I}}  & = 0. \label{eqn:jumps4}
\end{align}
\label{eqns:int}\end{subequations}
Here, $[ f ]_{\mathcal{I}}$ denotes the jump in the function $f$ across the interface, and $n$ is the normal unit vector to the interface.

\section{Summary of Numerical Scheme}  
\label{section:scheme}
Consider now the interior discretization of the governing dispersive wave equation \eqref{eqn:governingE}, that is to say, independent of physical boundaries or interfaces. The basic approach will follow the description in~\cite{henshaw_2006}, with appropriate modifications as required to treat the dissipative and dispersive time convolutions. In~\cite{henshaw_2006}, Henshaw considered the second-order (in derivative) formulation of the non-dispersive Maxwell's equations for dielectrics ({\it i.e.} \eqref{eqn:macro1} through \eqref{eqn:macro4}, with \eqref{eqn:HtoB} and $D  = \epsilon_0 \epsilon_r E$). In the case of piecewise constant material parameters $\epsilon_r$ and $\mu$ this reduces to the wave equation with speed $c = 1/\sqrt{ \epsilon_0 \epsilon_r \mu}$:
\begin{align}
& \partial_t^2 u(\xv,t)  = c^2 \Delta u(\xv,t),
\end{align} 
where $u$ is used to indicate a generic component of the solution. The time-stepping approach in~\cite{henshaw_2006} used a so-called modified-equation, or Taylor, time stepper based on the temporal Taylor expansion
\begin{align}
 \label{eqn:expansion1}  u(\xv,t + \Delta t) - 2 u(\xv,t) & + u(\xv,t - \Delta t)  \\
 \nonumber & = 2 \left( \frac{ \Delta t^2 }{2!}  \partial_t^2 u +  \frac{ \Delta t^4}{4!} \partial_t^4 u  +   \frac{ \Delta t^6}{6!} \partial_t^6 u + \dots \right).
\end{align}
Subsequently, the governing PDE is used to replace temporal with spatial derivatives and arrive at a discrete approximation of arbitrary spatial and temporal order-of-accuracy using only three time levels. 

Motivated by this approach, we apply the modified-equation time-stepping approach to the governing equation \eqref{eqn:governingE}. Since we are considering piecewise constant media, the Fourier transform and the spatial independence of the kernel $ \eta $ imply that the operators $\partial_t^2$, $\Delta$ and $\eta \ * $ all commute. Therefore, \eqref{eqn:governingE} implies that for $m \in \mathbb{N}$,
\begin{align}
\partial_t^{2m} \mathbf{E} = c^{2m} \left( \Delta + \mu \eta \ * \right)^m \mathbf{E}. \label{eqn:higherderivs}
\end{align}
Thus, the modified-equation time stepper for the electric field can be written as
\begin{align}
\label{eqn:expansion2} \mathbf{E}(\xv,t + \Delta t) & - 2 \mathbf{E}(\xv,t)  + \mathbf{E}(\xv,t - \Delta t)   \\ 
& = 2 \bigg( \frac{ \Delta t^2 }{2!} c^2 ( \Delta +   \mu \eta \ *  ) \mathbf{E}  +  \frac{ \Delta t^4}{4!} c^4 ( \Delta +  \mu \eta \ *  )^2 \mathbf{E} \nonumber \\
&  \hspace{4cm} +   \frac{ \Delta t^6}{6!} c^6  ( \Delta + \mu \eta \ * ) ^3 \mathbf{E} + \dots \bigg) .   \nonumber 
\end{align}
The operator binomial $( \Delta +  \mu \eta \ * )$ can be directly expanded to obtain its higher powers:
\begin{align}
 &  ( \Delta +  \mu \eta \ * )^2 \mathbf{E} = \Delta^2 \mathbf{E} + 2 \mu  \Delta   \eta * \mathbf{E}  +   \mu^2 \eta * \eta * \mathbf{E}, \label{eqn:operatorpowers0} \\
 &  ( \Delta +  \mu \eta \ * )^3 \mathbf{E} = \Delta^3 \mathbf{E} + 3 \mu  \Delta^2   \eta * \mathbf{E} +  \mu^2 \Delta \eta * \eta * \mathbf{E} +  \mu^3 \eta * \eta * \eta * \mathbf{E}  . \label{eqn:operatorpowers}
\end{align}
The convolution operators in \eqref{eqn:operatorpowers0} and \eqref{eqn:operatorpowers} can be obtained by computing powers of $\widehat{\eta}(\omega)$ in the frequency domain and then 
ing the Fourier transform. According to the expressions for $\eta$ in Table~\ref{table:etas}, each of these operators may be rewritten as the sum of the electric field $\mathbf{E}$, up to a constant, and a convolution term, which we define as an auxiliary integral expression. For example, for the Drude model,
\begin{align}
\label{eqn:order1drude} \eta * \mathbf{E}(\xv, t)  & = \epsilon_0 \omega_p^2 \int_0^{\infty} \left(  -  \delta(\tau) + \gamma e^{- \gamma \tau} \right) \mathbf{E}(\xv, t - \tau) d \tau \\
\nonumber & = - \epsilon_0 \omega_p^2 \mathbf{E}(\xv,t) + \epsilon_0 \omega_p^2 \gamma \int_0^{\infty}  e^{- \gamma \tau}  \mathbf{E}(\xv, t - \tau) d \tau . 
\end{align}

Fully discrete schemes can now be derived using appropriately accurate quadrature rules, which can be implemented in code using recursive convolution, along with centered difference discretizations of the powers of the Laplacian, as required to obtain the desired order of accuracy in space. The necessary orders of accuracy for the convolution and spatial discretizations follow the prescription given in~\cite{henshaw_2006}, and relates to the power of $\dt$ multiplying the term in \eqref{eqn:expansion2} as well as overall accuracy desire. In particular, the leading terms with $\dt^2$ require full accuracy, the next terms with $\dt^4$ can be two orders less accurate, and so on. Examples of second- and fourth-order accurate spatial and temporal discretizations for the Drude model are given below to illustrate the mechanics of the process. In each case, the stability of the fully discrete scheme applied to the Cauchy problem is analyzed. The second-order accurate scheme is found to be stable for finite-time integration given a time-step restriction, although it does admit bounded exponential growth. On the other hand, the fourth-order accurate scheme admits no exponential growth provided a time-step restriction is satisfied, and is therefore stable for both finite time, and infinite time integration.

\subsection{Second-Order Accurate Scheme for the Drude Model} 
\label{section:secondorderdrude}

Consider now a second-order accurate discretization of the governing equation \eqref{eqn:governingE} with $\widehat{\eta}(\omega)$ and $\eta(\tau)$ given by the Drude model in Table \ref{table:etas}. For convenience, the auxiliary quantity $\psiv$ is introduced as
\begin{align}
  \psiv(\xv,t) =  \int_0^{\infty}  e^{- \gamma \tau}  \mathbf{E}(\xv, t - \tau) d \tau .  \label{eqn:psidef}
\end{align}
Now let $\Ejn  \approx \mathbf{E}( \xv_{\jv}, n \Delta t)$ and $\psijn  \approx \psiv( \xv_{\jv}, n \Delta t)$ denote approximations to $\mathbf{E}$ and $\psiv$, where $\jv \in \mathbb{Z}^{\mathcal{D}}$ is a multi-index so that $\xv_{\jv}$ indicates a grid point in multiple space dimensions. In the usual way, a second-order accurate discretization can be derived using a second-order accurate truncation of  the expansion in \eqref{eqn:expansion2} along with the standard second-order accurate difference approximation of the Laplacian (denoted $\Delta_{2h}$) to  give
\begin{align}
  \mathbf{E}_{\jv}^{n+1} - 2 \mathbf{E}_{\jv}^n + \mathbf{E}_{\jv}^{n-1}  =   \Delta t^2 \left(  c^2 \Delta_{2h}\Ejn -  \frac{ \omega_p^2 }{\epsilon_r} \Ejn +   \frac{ \omega_p^2 \gamma }{\epsilon_r} \psijn \right), \label{eqn:E2nd}
\end{align}
where $c^2 = 1/(\epsilon_0 \epsilon_r \mu)$. For completeness, the discrete Laplacian is given by
\[
  \Delta_{2h}=\sum_{d=1}^{\mathcal{D}}D_{+,x_d}D_{-,x_d},
\]
where $\mathcal{D}$ is the number of spatial dimensions, and $D_{\pm,x_d}$ are the usual forward and backward divided difference operators in the $x_d$ direction given by
\[
  D_{+,x_d}u_{\jv}=\frac{u_{\jv+e_d}-u_{\jv}}{h_d} \qquad  \text{and} \qquad D_{-,x_d}u_{\jv}=\frac{u_{\jv}-u_{\jv-e_d}}{h_d}. 
\]
Here, $e_d$ is the vector of all zeros except a 1 in the $d^{\hbox{th}}$ entry, and $h_d$ is the grid spacing in the $x_d$ direction. For future reference, the undivided difference operators $\delta_{\pm,x_d}$ are defined similarly, as for example
\[
  \delta_{+,x_d}u_{\jv}={u_{\jv+e_d}-u_{\jv}}.
\]

Equation \eqref{eqn:E2nd} defines a temporal update for the electric field $\mathbf{E}$, and a corresponding update equation for the auxiliary quantity $\psiv$ must also be developed. In order to do so, the convolution term defining $\psij^{n+1}$ is evaluated to second-order accuracy, here using the composite trapezoidal quadrature as
\begin{align}
  \psij^{n+1} & = \frac{\Delta t}{2} \sum_{m = 0}^{\infty} \left(  \Ej^{n +1 - m} e^{- \gamma m \Delta t}+ \Ej^{n - m} e^{- \gamma (m + 1) \Delta t} \right).\label{eqn:ETrap} 
\end{align}
As written, \eqref{eqn:ETrap} would require an infinite time history for the field $\mathbf{E}$. However, due to the form of the kernel, Equation \eqref{eqn:ETrap} can be reorganized as follows
\begin{align}
  \psij^{n + 1} & =  \frac{\Delta t}{2} \Ej^{n+1} + \Delta t \sum_{m = 0}^{\infty} e^{- \gamma (m  + 1 ) \Delta t} \Ej^{n - m }  \\
    & =  \frac{\Delta t}{2} \Ej^{n + 1} + \Delta t  e^{ - \gamma \Delta t}  \Ej^n +  \Delta t \sum_{j = 1}^{\infty} e^{- \gamma ( m + 1 ) \Delta t} \Ej^{n - m }  \nonumber \\
    & =  \frac{\Delta t}{2} \Ej^{n + 1} + \Delta t e^{ - \gamma \Delta t} \Ej^n  +  \Delta t e^{-\gamma\Delta t} \sum_{m = 0}^{\infty} e^{- \gamma ( m + 1 ) \Delta t} \Ej^{n - m - 1}  \nonumber \\
    & = \frac{\Delta t}{2} \Ej^{n + 1}  + \frac{\Delta t}{2} e^{ - \gamma \Delta t}  \Ej^n  + e^{- \gamma \Delta t} \psij^n.  \nonumber
\end{align}
This formulation is often referred to as recursive convolution, and gives a time update of $\psiv$ based on data from just one prior time level. The entire second-order accurate solution update is then given as
\begin{subequations}
\begin{align}
  \mathbf{E}_{\jv}^{n+1}  & =   2 \mathbf{E}_{\jv}^n - \mathbf{E}_{\jv}^{n-1} \Delta t^2 \left(  c^2 \Delta_{2h}\Ejn -  \frac{ \omega_p^2 }{\epsilon_r} \Ejn +   \frac{ \omega_p^2 \gamma }{\epsilon_r} \psijn \right), \label{eqn:secondorderexpansion}\\
  \psij^{n + 1} & = \frac{\Delta t}{2} \Ej^{n + 1}  + \frac{\Delta t}{2} e^{ - \gamma \Delta t}  \Ej^n  + e^{- \gamma \Delta t} \psij^n.\label{eqn:recursionpsi} 
\end{align}
\label{eqns:RC2}
\end{subequations}

\subsection{Stability of the Second-Order Accurate Scheme for the Drude Media} \label{section:2ndorderstab}
One important property of a numerical scheme is that of stability for the Cauchy problem, i.e. the governing PDE on an infinite domain. The key idea of numerical stability is that the discretization not admit solutions with unbounded exponential growth, and therefore that the norm of the solution can be bounded by the initial (and boundary) conditions. However, an important subtlety in the present context is that numerical stability does not imply a lack of exponential growth. In fact as we shall see, the second-order accurate recursive convolution algorithm of \eqref{eqns:RC2} admits solutions that can grow slowly in time when $\gamma\ne0$ and $\omega\ne0$. Nonetheless, the exponential growth rate is finite (and approaches zero with the time-step), and the algorithm is formally stable provided the time-step satisfies certain constraints. Importantly, this result is primarily dependent on the recursive convolution approximation, and not the precise treatment of the spatial derivative operator. As a result, if no exponential growth at all can be tolerated in a particular application, one can simply replace the second-order accurate convolution by the fourth-order accurate convolution described in Section~\ref{section:fourthorderdrude} below, with corresponding changes to the quantitative stability bounds.

In the discussion of stability to follow, the results center around the exact solutions of the discretization represented by \eqref{eqns:RC2}, which is now presented. The equations are rescaled to obtain
\begin{subequations}
\begin{align}
  \Ej^{n + 1} & = 2 \Ej^n - \Ej^{n - 1}  +   \sum_{d=1}^{\mathcal{D}}\lambda_{d}^2\delta_{+,x_d}\delta_{-,x_d}\Ej^n - \Omega^2 \Ej^n +  \Omega^2 \Psij^n   , \label{eqn:2ndorderscheme_EScaled}  \\
  \Psij^{n + 1} & = \frac{ \Gamma }{2} \Ej^{n + 1} + \frac{ \Gamma }{2} e^{- \Gamma} \Ej^n + e^{ - \Gamma} \Psij^n   \label{eqn:2ndorderscheme_psiScaled}
 \end{align}
 \label{eqns:2ndorderscheme_Scaled}
 \end{subequations}
where $\Psij^n \equiv \gamma \psij^n$, and the following dimensionless parameters have been identified:
\begin{align}
 \lambda_{d} \equiv c\frac{\Delta t}{h_d} , \qquad \Omega \equiv  \sqrt{ \frac{ \Delta t^2  \omega_p^2  }{\epsilon_r} } , \qquad \hbox{and} \qquad \Gamma \equiv \Delta t \gamma.
\label{eqn:stabscaling}
\end{align}
Equations \eqref{eqn:2ndorderscheme_EScaled} and \eqref{eqn:2ndorderscheme_psiScaled} are linear constant coefficient difference equations, and so can be solved  by seeking separable solutions of the form
\begin{align}
  \Ej^n = \mathbf{C}_E A^n e^{i \, \kv\cdot\xv_{\jv}} ,  \qquad \Psi_j^n = \mathbf{C}_{\Psi}A^n e^{i \, \kv\cdot\xv_{\jv}} ,
  \label{eqn:2ndAnzatz}
\end{align}
where $\kv$ is a real-valued, vector wave number and $A$ is a complex-valued amplification factor\footnote{This approach to determining the solution to the difference equation is essentially a discrete Fourier transformation in space and a discrete Laplace transformation in time.}. Substitution of \eqref{eqn:2ndAnzatz} into \eqref{eqns:2ndorderscheme_Scaled} yields the linear system
\begin{align}
  \begin{bmatrix}
     - A^2 + 2 A - 1 - \sum_{d=1}^{\mathcal{D}}4 \lambda_{d}^2  \sin^2 \left( \frac{ \xi_d}{2} \right) A -   \Omega^2 A  & \Omega^2A \\  
    \frac{ \Gamma }{2} \left( A + e^{-\Gamma} \right) & -A +  e^{ - \Gamma}
  \end{bmatrix}
  \begin{bmatrix} 
    \mathbf{C}_E \\ \mathbf{C}_{\Psi}
  \end{bmatrix}  
  = 
  \begin{bmatrix}
    0\\0
  \end{bmatrix}
  \label{eqn:ampmatrix}
\end{align}
where $\xi_d=k_dh_d$ is the grid wave number (i.e. $\xi_d$ takes discrete values) in the $x_d$ direction. However it is convenient to allow $\xi_d$ to vary continuously over a $2\pi$ range; here we choose $\xi_d\in[-\pi,\pi]$. Nontrivial solutions of \eqref{eqn:ampmatrix} exist when the determinant of the matrix is zero which gives
\begin{align}
A^3  & + \coeff{2} A^2+\coeff{1} A+\coeff{0}=0,
\label{eqn:charpol}
\end{align}
where the coefficients in the polynomial are
\begin{align*}
  \coeff{2} & = 4\sum_{d=1}^{\mathcal{D}}4 \lambda_{d}^2  \sin^2 \left( \frac{ \xi_d}{2}\right)  + \Omega^2 - e^{- \Gamma} - 2 - \frac{1}{2} \Omega^2 \Gamma,\\
  \coeff{1} & = - 4\sum_{d=1}^{\mathcal{D}}4 \lambda_{d}^2  \sin^2 \left( \frac{ \xi_d}{2} \right) e^{- \Gamma} - \Omega^2 e^{- \Gamma} + 2 e^{- \Gamma} + 1 - \frac{1}{2} \Omega^2 \Gamma e^{- \Gamma},\\
  \coeff{0} & = - e^{ - \Gamma}.
\end{align*}
Note that because each vector $\mathbf{C}_E$ and $\mathbf{C}_{\Psi}$ has three components, the full determinant condition is in fact given by
\begin{align*}
( A^3  & + \coeff{2} A^2+\coeff{1} A+\coeff{0})^3  =0,
\end{align*}
and the roots corresponding to each component pair of $\mathbf{C}_E$ and $\mathbf{C}_{\Psi}$ occur thrice and are each determined by \eqref{eqn:charpol}. Equation \eqref{eqn:charpol} is a cubic polynomial defining the amplification factor $A$ which can be solved exactly to give the three roots
\begin{subequations}
\begin{align}
  A_0 & = \zeta_{-}-\frac{\coeff{2}}{3}, \label{eqn:badRoot}\\
  A_{\pm} & = -\frac{\zeta_{-}}{2}-\frac{\coeff{2}}{3}\pm\frac{i\zeta_{+}}{2}\sqrt {3},\label{eqn:goodRoots}
\end{align}
\end{subequations}
where the following definitions have been used
\begin{align*}
  \zeta_{\pm} = & \  \frac{\eta}{6}\pm6\frac{3\coeff{1}-\coeff{2}^2}{9\eta}\\
  \eta  = & \ \bigg(   36\coeff{1}\coeff{2}-108\coeff{0}-8{\coeff{2}}^{3} \\ 
  & +12\sqrt {12\coeff{0}{\coeff{2}}^{3}-3{\coeff{1}}^{2}{
\coeff{2}}^{2}-54\coeff{2}\coeff{1}\coeff{0}+12{\coeff{1}}^{3}+81{\coeff{0}}^{2}} \bigg)^{1/3} .
\end{align*}
Note that the principle branches of the square and cube root functions are used. This choice, as well as the choice of notation in \eqref{eqn:badRoot} and \eqref{eqn:goodRoots}, is made to specifically identify the $A_0$ root of \eqref{eqn:charpol}, which illustrates particularly interesting and subtle behavior for the present numerical scheme as outlined in the following propositions. Note also that $A_0$ as defined in \eqref{eqn:badRoot} is always real valued for $\Gamma>0$, $\Lambda>0$, and $\Omega\in \mathbb{R}$.

\begin{myproposition} The second-order accurate recursive convolution algorithm of \eqref{eqns:RC2} admits exponential growth provided $\gamma\ne0$ and $\omega\ne0$. 
\label{prop:RC2Growth}
\end{myproposition}
\begin{proof}

 To show the existence of exponentially growing modes it suffices to simplify Equation \eqref{eqn:charpol} by setting $\xi_d=0$ and looking only at the constant spatial modes\footnote{In practice the boundary conditions may or may not support an exponentiating constant mode; for example periodic conditions would support it while zero Dirichlet conditions would not. In practice, the dominant growing mode is the lowest frequency mode which also satisfies the boundary conditions.}. Surfaces of the magnitude of the three roots $A_0$ and $A_{\pm}$ are plotted in Figure~\ref{fig:RC2_exp}. In particular, the left-most plot shows clearly that $|A_0|>1$ except when $\Omega=0$ or $\Gamma=0$ (corresponding to $\omega=0$ or $\gamma=0$ in physical parameters). As a result, the constant mode is seen to admit exponentially growing solutions provided $\gamma\ne0$ and $\omega\ne0$.
\begin{figure}[hbt]
  \begin{center}
 \includegraphics[width=.325\textwidth]{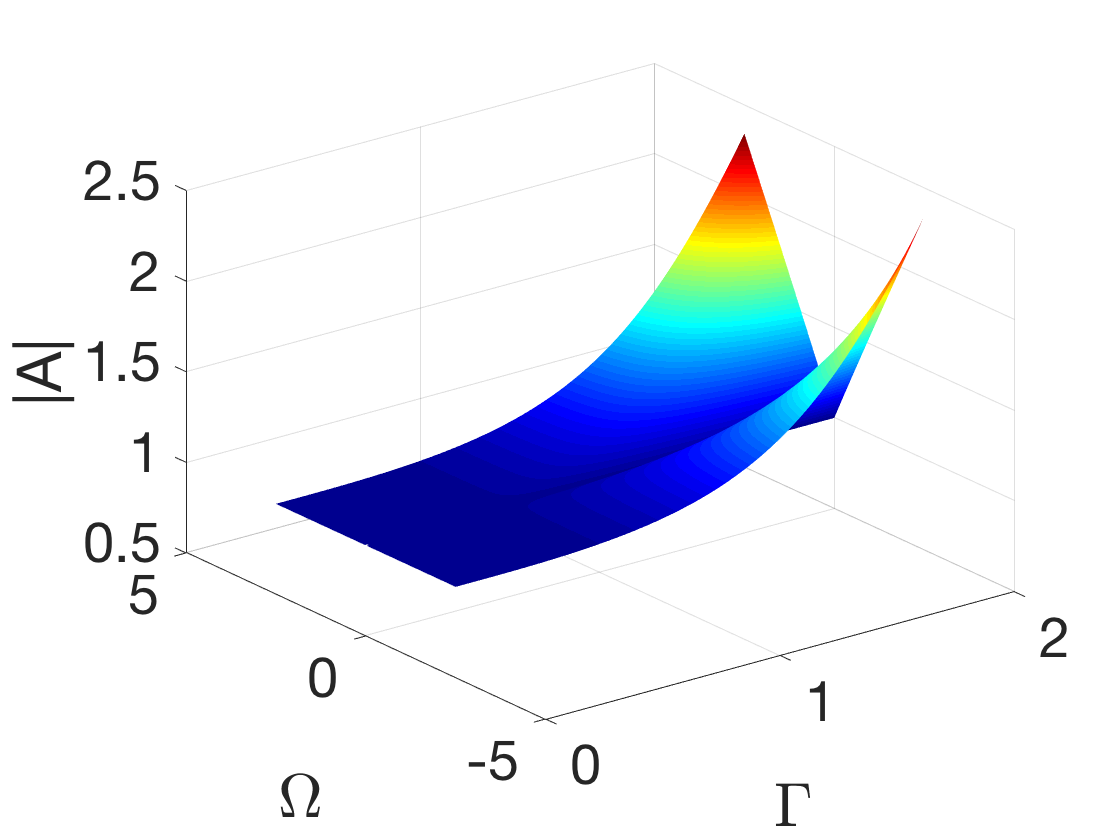} \hfill
  \includegraphics[width=.325\textwidth]{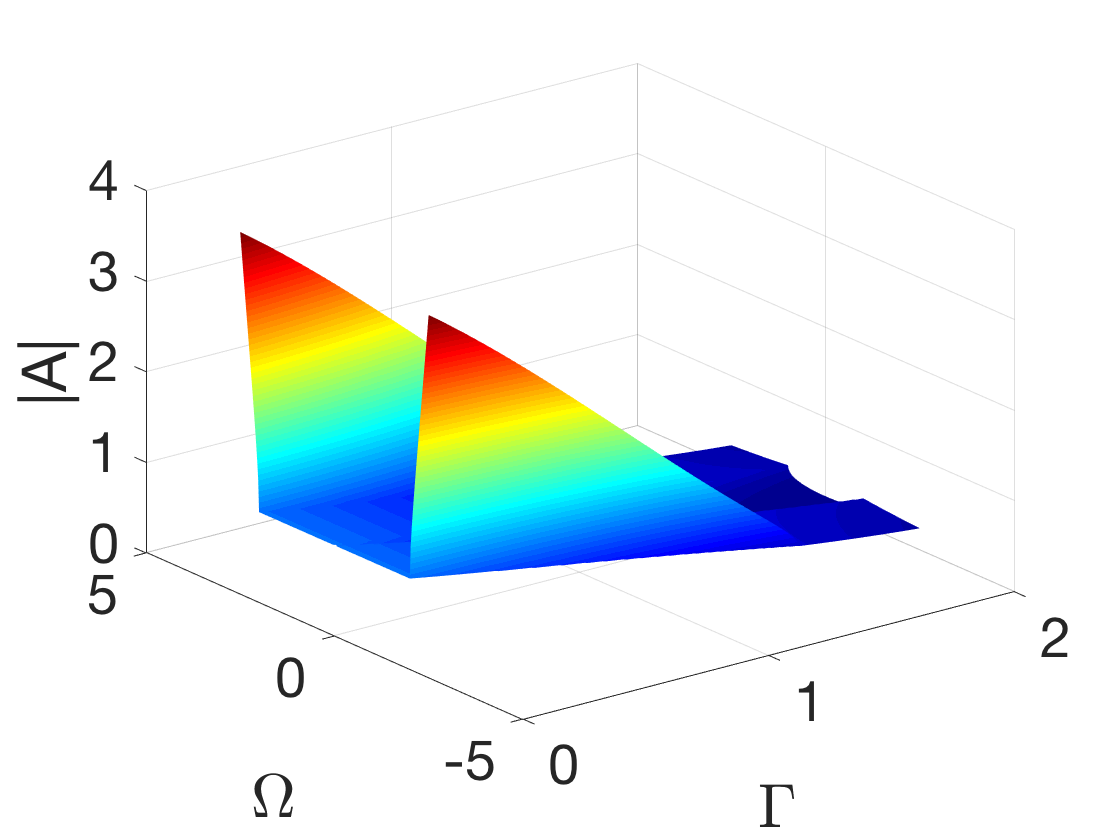} \hfill
  \includegraphics[width=.325\textwidth]{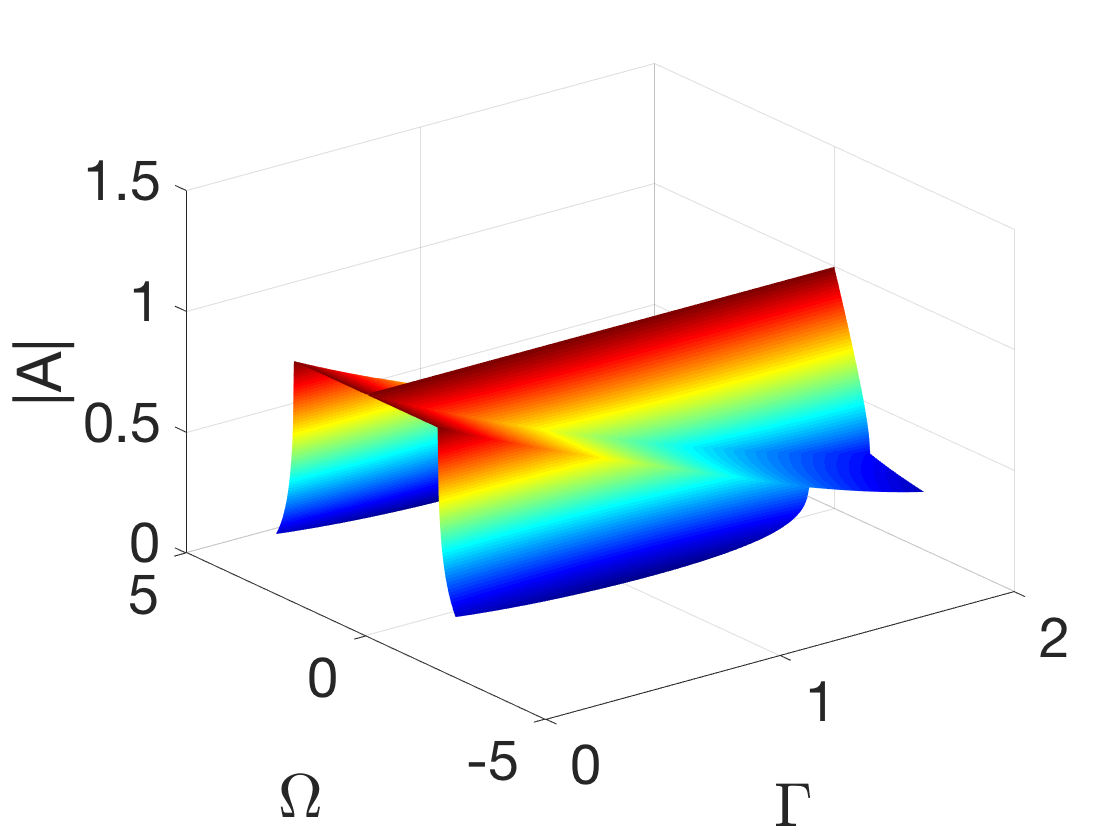}
  \caption{Surfaces of the magnitude of the amplification factor $A$ for the recursive convolution algorithm of \eqref{eqns:RC2} for the constant mode with $\xi_d=0$. From left to right are $|A_0|$, $|A_{+}|$ and $|A_{-}|$, and the left-most root shows that $|A_0|>0$ for $\gamma\ne0$ and $\omega\ne0$. As  result, the constant mode grows exponentially in time provided $\gamma\ne0$ and $\omega\ne0$.}
  \label{fig:RC2_exp}
  \end{center}
\end{figure}
\end{proof}

The fact that the second-order accurate recursive convolution algorithm admits exponential growth may be troubling, but it does not in itself imply that the scheme is unstable. In order to quantify this statement, one needs a precise definition of stability and here we adopt the following classical notion of stability~\cite{thomas1999}. 
\begin{definition}
  A numerical scheme defining a grid function $\mathbf{u}_{\jv}^n$, which is computed from its initial condition $\mathbf{u}_{\jv}^0$, is said to be stable with respect to the discrete norm $||\cdot||_h$ if there exists positive constants $\beta$ and $\kappa$ (independent of the grid spacing $h$ and time step $\Delta t$) such that
  \[
    || \mathbf{u}_{\jv}^n||_h \le \kappa e^{\beta t} || \mathbf{u}_{\jv}^0||_h
  \]
  for $0\le t \le T$,  $T$ is a given final time, and $h$ and $\Delta t$ are taken sufficiently small.
  \label{def:stab}
\end{definition}
This definition of stability, along with numerical consistency, is sufficient to guarantee convergence of a discrete approximation to the true solution for $t<T$ as the grid is refined. Importantly in the current discussion, Definition~\ref{def:stab} allows bounded exponential growth. Careful investigation of the roots of \eqref{eqn:charpol} leads to the following stability result.

\begin{myproposition} Given physical parameters $c$, $\omega_p^2/\epsilon_r$, and $\gamma$ and according to Definition \ref{def:stab}, a sufficient condition for stability of the second-order accurate recursive convolution algorithm of \eqref{eqns:RC2} is 
\begin{align}
    \Lambda + \frac{\Omega^2}{4} \equiv   c\Delta t\sqrt{\sum_{d=1}^{\mathcal{D}}\frac{1}{h_d^2} } + \frac{\omega_p^2 \Delta t^2}{ 4 \epsilon_r} < 1, 
    \label{eq:RC2CFL}
\end{align}
Recall that $\mathcal{D}$ is the spatial dimension, and $h_d$ is the grid spacing in the $d$ coordinate direction, and so \eqref{eqn:charpol} is a typical CFL-like constraint on the time step.
\label{prop:RC2Stab}
\end{myproposition}
\begin{proof}
In investigating Proposition~\ref{prop:RC2Stab} the two sets of roots of \eqref{eqn:charpol}, $A_0$ in \eqref{eqn:badRoot} and $A_{\pm}$ in \eqref{eqn:goodRoots}, are discussed independently. First consider the pair $A_{\pm}$, which are functions of the discrete wave numbers $\xi_d$. After maximizing over those wave numbers (the maxima occurring with $\xi_d=\pm\pi$ or $\xi_d=0$), the unity iso-surface of the maximum of $|A_{\pm}|$ in the space of dimensionless parameters $\Omega$, $\Gamma$, and $\Lambda$ defined by the set
\begin{align}
\mathcal{S}_1 = \bigg\{ ( \Omega, \Gamma, \Lambda ) \ : \  \underset{ \pm}{ \max } \big|A_{\pm}(\Omega, \Gamma, \Lambda) \big| = 1 \bigg\} ,  \label{eqn:isoset}
\end{align}
is plotted in Figure~\ref{fig:RC2_Apm_iso}. 
\begin{figure}
  \begin{center}
  \includegraphics[width=.325\textwidth]{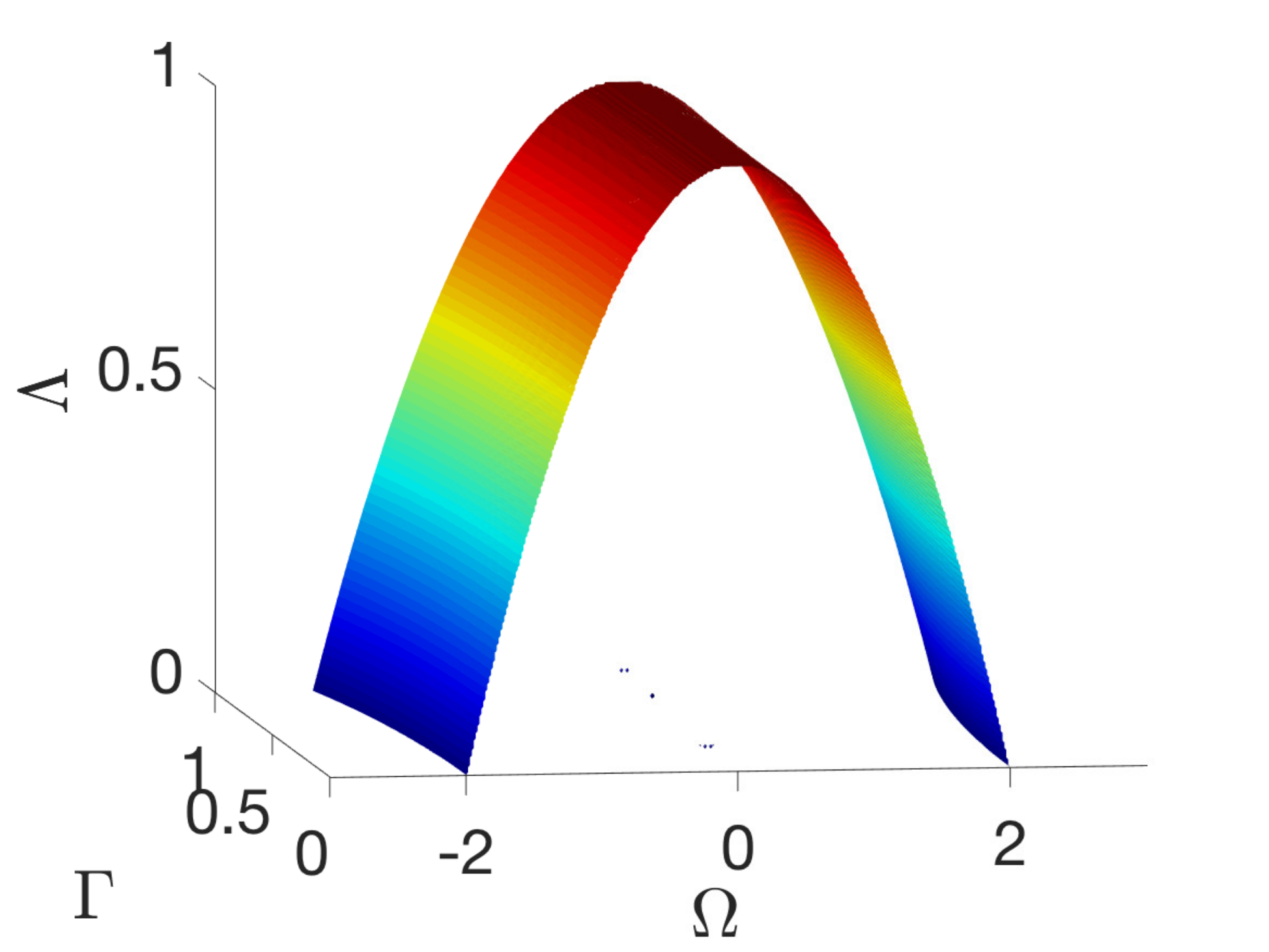} \hfill
  \includegraphics[width=.325\textwidth]{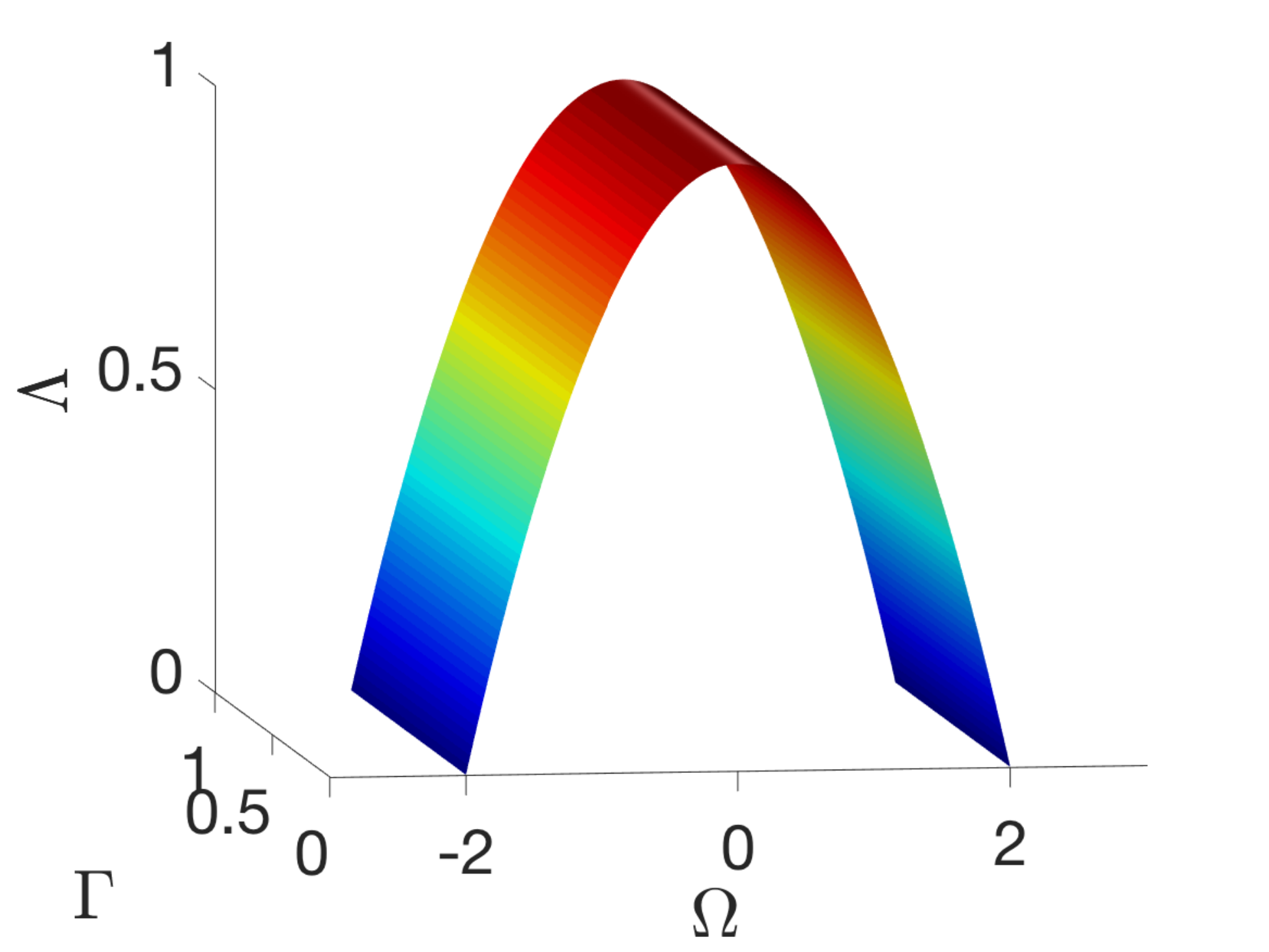} \hfill
  \includegraphics[width=.325\textwidth]{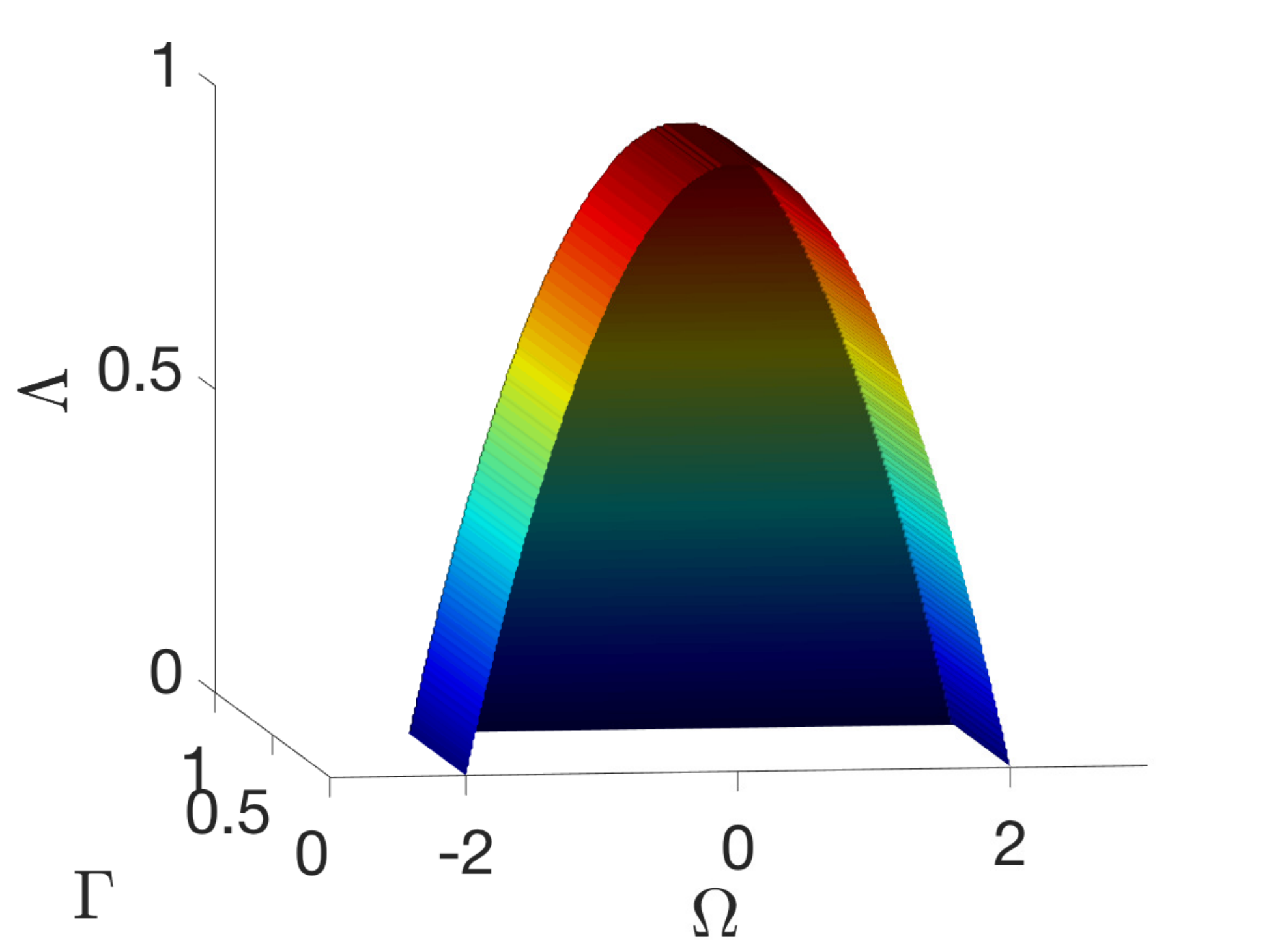}
  \caption{At left is the unity iso-surface $\mathcal{S}_1$, defined in \eqref{eqn:isoset}, of the maximum of $|A_{\pm}|$ in the space of dimensionless parameters $\Omega$, $\Gamma$, and $\Lambda$. Below this unity iso-surface, the maximum $\max |A_{\pm}|\le1$ and so those roots correspond to modes in the approximation that do not grow in time.  In center is a simplified isosurface $\mathcal{S}_2$, defined in \eqref{eqn:isoset2} by $\Lambda+\Omega^2/4=1$. At right is the complete bound \eqref{eq:RC2_time_step}, which is used in the time-step determination.}
  \label{fig:RC2_Apm_iso}
  \end{center}
\end{figure}
A sufficient condition for solutions of the second-order algorithm \eqref{eqns:RC2} to be stable is that they do not admit exponential growth in time. For the solutions corresponding to the amplification factors $A_{\pm}$, we therefore seek $\Delta t$ such that $ \underset{ \pm}{ \max } \big|A_{\pm}(\Omega, \Gamma, \Lambda) \big|  < 1$, which forbids the possibility of such growth. This region is bounded above by the isosurface $\mathcal{S}_1$ in Figure~\ref{fig:RC2_Apm_iso}, and a necessary and sufficient condition for stability is therefore that the parameters $(\Omega, \Gamma, \Lambda)$ lie within it. To simplify this condition, note that the region bounded by $\mathcal{S}_1$ contains the isosurface
\begin{align}
\mathcal{S}_2 =  \bigg\{ ( \Omega, \Gamma, \Lambda ) \ : \  \Lambda+\Omega^2/4=1 \bigg\}. \label{eqn:isoset2}
\end{align}
The set of $(\Omega, \Gamma, \Lambda)$ bounded above by $\mathcal{S}_2$ is therefore stable, and the bound $  \Lambda+\Omega^2/4 \leq 1$ provides a simplified and more easily manipulated sufficient stability criterion. To complete the proof of Proposition~\ref{prop:RC2Stab}  for the solutions corresponding to amplification factors $A_{\pm}$, we recall that $\Omega\sim\dt$, and because by assumption $\Lambda \le \Lambda_0$, we are interested in the limit $\dt\to 0$. As a result, we are free to pick an $\Omega_0>0$ sufficiently small such that $\max(|A_{\pm}|)<1$ provided $|\Omega|<\Omega_0$ and $\Lambda<1-\Omega_0^2$. The latter constant is easily identified as $\Lambda_0=1-\Omega_0^2$, which guarantees that $(\Omega, \Gamma, \Lambda) \in \mathcal{S}_2 \subset \mathcal{S}_1$ and completes the consideration of stability for the solutions corresponding to the amplification factors $A_{\pm}$. 

Moving on to $A_0$, we are faced with the more subtle situation where $|A_0|>1$ as discussed in Lemma~\ref{prop:RC2Growth}. As before, $\dt\to 0$ which implies $\Gamma\sim\dt$ and $\Omega\sim\dt$. We now expand the root $A_0$ about $\Gamma = 0$ to illustrate its size. Note that some care must be taken due to the fact that when $\Omega = 0$ and $\Gamma > 0$, the root $A_0 = 1$, while there exists a branch point of $A_0$ at the point $(\Gamma, \Omega) = (0,0)$. Taking a Taylor expansion of $A_0$ about $\Gamma=0$ for $ 0 <  | \Omega | < 2$ and using the fact that $\Gamma>0$ yields 
\[
  A_0=1+\frac{\Gamma^3}{12}+O(\Gamma^4).
\]
This expansion implies that the spurious exponential growth takes the form $A_0^n \sim e^{n \gamma^3 \dt^3}$, or equivalently $A_0^n \sim e^{T \dt^2}$ where the final time $T=n\dt$. That is to say that the exponential growth rate is $O(\dt^3)$ in time-step, or $O(\dt^2)$ in physical time. Integrating to a final time as given in Definition~\ref{def:stab} then yields bounded bounded exponential growth tending to zero as $\dt\to 0$. This, along with the fact that $|A_{\pm}|<1$ provided $\Lambda<1$ and $\dt$ is sufficiently small establishes finite-time stability as described by Definition~\ref{def:stab}.
\end{proof}


We now provide a short discussion relating to the choice of time step for physically reasonably values of $\Lambda$, $\Omega$, and $\Gamma$. As discussed in Proposition~\ref{prop:RC2Growth}, the largest amplification factor $|A|$ is always greater than unity and so numerical solutions would exhibit exponential growth. However, the growth associated with $A_0$ as $\dt\to 0$ is rather weak, with $|A_0|=1+O(\dt^3)$, in comparison to the more rapid growth characterized by a classical CFL violation when $\Lambda+\Omega^2/4>1$ for which $|A_{\pm}|=1+O(1)$. The latter bound is plotted in the center panel of Figure~\ref{fig:RC2_Apm_iso}, and is slightly more restrictive but algebraically simpler version of the true iso-surface shown in the leftmost panel of Figure~\ref{fig:RC2_Apm_iso}. In selecting a time step $\dt$, it is therefore practically useful to distinguish between these two types of exponentiation, i.e. the slow growth associated with the $A_0$ and fast growth associated with $A_{\pm}$. To that end, Figure \ref{fig:second_order_stab} shows contours of the maximum amplification factor for the spatially constant mode as a function of $\Gamma$ and $\Omega$ (recall that the constant mode is associated with the maximum $A_0$). In this plot, the transition from slow to fast growth is clearly apparent near $\Omega\approx 2$.
%
%
\begin{figure}
\begin{center}
\includegraphics[width=.5\textwidth]{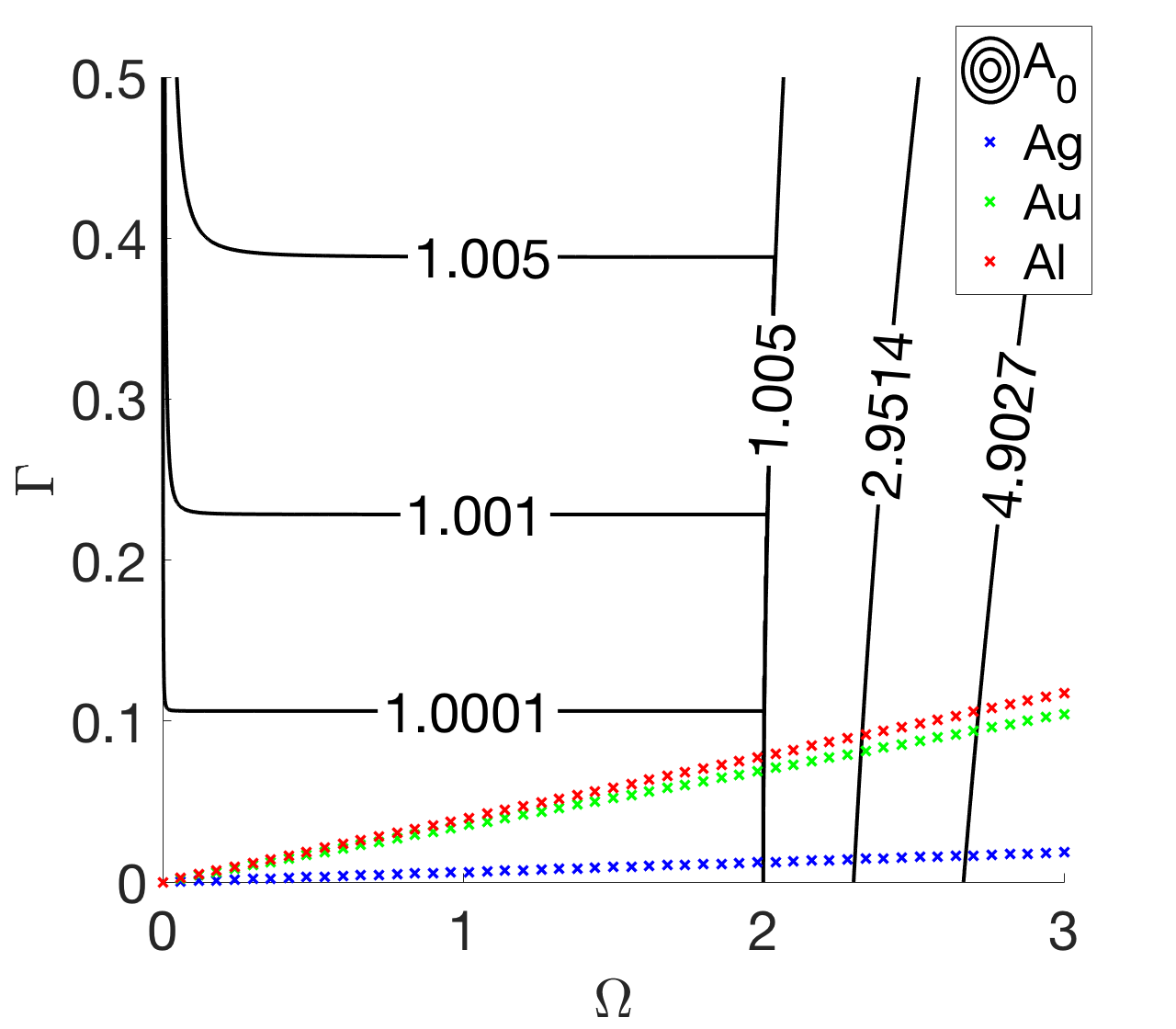}
\caption{Contours of the amplification factor $A_0$ for the second-order scheme and the spatially constant mode. The straight lines correspond to several common Drude models for metals (from top to bottom, Aluminum \cite{blaber_2009}, Silver \cite{yang_2015}, and Gold \cite{olmon_2012} as time step $\dt$ is refined.} 
\label{fig:second_order_stab} 
\end{center}
\end{figure}
In addition, reference lines showing the trajectory followed as $\dt\to 0$ for physical parameters corresponding to gold, silver, and aluminum are also included. These lines indicate that for physically realistic metals, the transition from fast to slow exponential growth happens around $\Gamma\approx .05$. Nonetheless it is prudent to exclude arbitrarily large $A_0$ in the event that an extremely large value of $\gamma$ is used and so we (somewhat arbitrarily) impose $\Gamma<0.5$. Taken together the bounds give a region or practical stability, and so the time-step for the computations presented in Section~\ref{section:six} is chosen to satisfy 
\begin{equation}
  \Lambda+\Omega^2/4\le1, \qquad \hbox{and} \qquad \Gamma \leq 0.5.
  \label{eq:RC2_time_step}
\end{equation}
A surface illustrating this bound is shown in the far right of Figure~\ref{fig:RC2_Apm_iso}. We summarize the appropriate choice of $\Delta t$ for a stable second-order algorithm with the following proposition.

\begin{myproposition} 
Given physical parameters $c$, $\omega_p^2/\epsilon_r$, and $\gamma$, define $\Delta t_m$ to be the smallest root of the quadratic polynomial equation
\begin{align}
   \frac{\omega_p^2 \Delta t^2 }{\epsilon_r} + c  \Delta t  \sqrt{\sum_{d=0}^{\mathcal{D}}\frac{1}{h_d^2} } - 1 = 0. 
\end{align}
A practical sufficient condition for stability of the second-order accurate recursive convolution algorithm of \eqref{eqns:RC2} is given by
\begin{align}
\Delta t = \min \left\{  \Delta t_m ,  \frac{0.5}{\gamma} \right \} .
\end{align}
\end{myproposition}
The proof follows from \eqref{eq:RC2_time_step} with the substitutions $\Lambda = c\Delta t\sqrt{\sum_{d=0}^{\mathcal{D}}\frac{1}{h_d^2} } $, $\Omega = \Delta t \omega_p / \sqrt{\epsilon_r}$, and $\Gamma = \Delta t \gamma$ given in \eqref{eqn:stabscaling}. 

\subsection{Fourth-Order Accurate Scheme for the Drude Model} 
\label{section:fourthorderdrude}
Consider again the initial boundary value problem \eqref{eqn:governingE} with $\widehat{\eta}(\omega)$ and $\eta(\tau)$ given by the Drude model in Table \ref{table:dispersionmodels}. Previously in Section~\ref{section:secondorderdrude}, discussion of the convolution $\eta * E$ led to the introduction of the auxiliary quantity $\psiv$ as defined in Equation \eqref{eqn:psidef}. Similarly, discussion of the fourth-order accurate discretization will naturally lead to the convolution $\eta  * \eta* E$, can be found by a straightforward computation in the frequency domain
\begin{align}
 \widehat{\eta}(\omega)^2 &  =    \frac{ \epsilon_0^2 \omega_p^4  \omega^2}{ (\omega + i \gamma)^2 } \\
    & = \epsilon_0^2 \omega_p^4 \left( 1 + \frac{2 \gamma} {i \omega - \gamma} + \frac{\gamma^2}{(i \omega - \gamma )^2} \right). \nonumber
\end{align}
Transformation back to the time domain gives
\begin{align}
 \eta * \eta (\tau) = \epsilon_0^2 \omega_p^4 \left[   \delta(\tau) -  2   \gamma  e^{- \gamma \tau}  \Theta(\tau) +    \gamma^2 \ \tau  e^{- \gamma \tau}  \Theta(\tau) \right] ,
\end{align}
and therefore,
\begin{align}
 \label{eqn:etaconvo2} \eta * \eta * \mathbf{E}(\xv,t) & =  \epsilon_0^2 \omega_p^4 \bigg[  \mathbf{E}(\xv,t) -  2   \gamma  \int_0^{\infty} e^{- \gamma \tau} \mathbf{E}(\xv, t - \tau) d \tau  \\ 
 & \hspace{4cm}  +  \gamma^2  \int_0^{\infty}  \tau \ e^{- \gamma \tau} \mathbf{E}(\xv, t - \tau) d \tau \bigg]   \nonumber \\
 & =    \epsilon_0^2 \omega_p^4  \left[   \mathbf{E}(\xv,t) -  2   \gamma \psiv(\xv,t)  +  \gamma^2  \phiv(\xv,t) \right],   \nonumber
\end{align}
where the additional auxiliary field $\phiv$ is defined as
\begin{align}
  \phiv(\xv,t) =  \int_0^{\infty}  \tau \  e^{- \gamma \tau}  \mathbf{E}(\xv, t - \tau) d \tau .
\end{align}
Using this notation the time convolutions needed for the fourth-order scheme can be expressed as 
\begin{subequations}
\begin{align}
 ( \Delta +   \mu \eta \ *  ) \mathbf{E}(\xv,t)   = &  \Delta \mathbf{E}(x,t)  + \mu \epsilon_0  \omega_p^2 \left[   -  \mathbf{E}(\xv,t) +   \gamma \psiv(\xv,t) \right] , \label{eqn:4thorderpart}\\
\label{eqn:2ndorderpart} ( \Delta +   \mu \eta \ *  )^2 \mathbf{E}(\xv,t) =  & \ \Delta^2 \mathbf{E}(\xv,t)  +  2 \mu \epsilon_0 \omega_p^2 \left[ -  \Delta \mathbf{E}(\xv,t) +  \gamma \Delta \psiv(\xv,t) \right]    \\
 	& \quad + \mu^2 \epsilon_0^2 \omega_p^2 \left[   \mathbf{E}(\xv,t) -  2    \gamma \psiv(\xv,t)  +  \gamma^2  \phiv(\xv,t) \right].  \nonumber
\end{align}
\end{subequations}

Derivation of a fourth-order accurate approximation to \eqref{eqn:governingE} now follows the approach outlined in~\cite{henshaw_2006}. In particular, the Taylor expansion \eqref{eqn:expansion2} is truncated after the first two terms, and numerical approximations to the differential and integral convolution terms are employed. As before $\Ejn  \approx \mathbf{E}( \xv_{\jv}, n \Delta t)$, $\psijn  \approx \psiv( \xv_{\jv}, n \Delta t)$, and additionally $\phijn  \approx \phiv( \xv_{\jv}, n \Delta t)$ is introduced as an approximation to $\phiv$. The discrete equation for the electric field can now be written as
\begin{align}
  \Ej^{n + 1}  -\ 2 \Ej^n + \Ej^{n - 1}  = & \    \Delta t^2 \left[ c^2 \Delta_{4h} \Ej^n  - \frac{\omega_p^2}{\epsilon_r}  \Ej^n +  \frac{\omega_p^2 \gamma}{\epsilon_r}  \psij^n \right]  \label{eqn:E4th} \\
		&  +  \frac{ \Delta t^4}{12}  \bigg[ c^4 \Delta^2_{2h} \Ej^n  -  \frac{2 c^2 \omega_p^2}{\epsilon_r}  \Delta_{2h} \Ej^n +  \frac{2 c^2 \omega_p^2 \gamma}{\epsilon_r}  \Delta_{2h} \psij^n  \nonumber \\
		&  \hspace{2cm} +  \frac{\omega_p^4}{\epsilon_r^2}  \Ej^n -  \frac{2 \omega_p^4 \gamma}{\epsilon_r^2}  \psij^n  + \frac{\omega_p^4 \gamma^2 }{\epsilon_r^2}   \phij^n \bigg],  \nonumber 
\end{align}
where $\Delta_{2 h}$ is, as before, the second-order accurate centered Laplacian, and $\Delta_{4 h}$ denotes the standard fourth-order accurate centered-difference approximation of the Laplacian as in
\[
  \Delta_{4h}=\sum_{d=1}^{\mathcal{D}}D_{+,x_d}D_{-,x_d}\left(I-\frac{h_d^2}{12}D_{+,x_d}D_{-,x_d}\right).
\]
Note that the fourth-order accurate approximation $\Delta_{4h}$ is used in the 1st term of \eqref{eqn:E4th}, while the second-order approximation $\Delta_{2h}$ is used in the 2nd term. The fact that the second-order accurate $\Delta_{2 h}$ in the 2nd term in \eqref{eqn:E4th} is sufficient for overall fourth-order accuracy relates to the term's location in the Taylor expansion, as discussed in~\cite{henshaw_2006,fornberg96}. In a similar manner, accuracy of the update equation \eqref{eqn:E4th} demands that the integral defining $\psij^n$ should be fourth-order accurate, while the integral defining $\phij^n$ needs only to be second-order accurate.

To obtain fourth-order accuracy for $\psij^n$, cubic polynomials are defined to interpolate the integrand between time levels, and the result is then integrated in time.  In particular, using the notation $v^n = v(n \Delta t)$ the infinite time convolution is decomposed as
\begin{align}
  \int_{0}^{\infty} v(\tau) d\tau & = 
    \int_{0}^{\Delta t} v(\tau) d\tau
    +\int_{\Delta t}^{\infty} v(\tau) d\tau. \label{eqn:TC}
\end{align}
The second term in \eqref{eqn:TC} is straightforward to treat using centered cubic interpolation over each time interval 
\begin{align}
  \int_{\Delta t}^{\infty} v(\tau) d\tau & = \sum_{m = 1}^{\infty} \int_{m \Delta t}^{(m + 1) \Delta t} v( \tau) d \tau  \\
   & = \frac{\Delta t}{24} \sum_{ m = 1}^{\infty} ( - v_{m - 1} + 13 v_m + 13 v_{m + 1} - v_{m + 2} ) + \mathcal{O}( \Delta t^4) \nonumber \\
  & \approx \Delta t \sum_{j = 3}^{\infty} v_j + \frac{25}{24} \Delta t \ v_2 + \frac{1}{2} \Delta t \ v_1 - \frac{1}{24} \Delta t \ v_0. \nonumber 
\end{align}
The integral over the last time interval $[0,\Delta t]$ uses a one-sided interpolant to yield
\begin{align}
\int_0^{\Delta t} v(\tau) d\tau \approx \frac{\Delta t}{24} \left( v_3 - 5 v_2 + 19 v_1 + 9 v_0 \right).
\end{align}
Substitution into \eqref{eqn:TC} then gives the following integral approximation
\begin{align} \label{eqn:psidisc}
\psij^n  
& = \Delta t \bigg( \sum_{ m = 4}^{\infty} e^{- m \Delta t \gamma}  \Ej^{n - m} + \frac{25}{24} e^{- 3 \Delta t \gamma} \Ej^{n - 3} \\
& \hspace{1.5cm}  + \frac{5}{6} \Ej^{n - 2}  \ e^{-  2 \Delta t \gamma} + \frac{31}{24} \ e^{- \Delta t \gamma} \Ej^{n - 1} + \frac{1}{3} \Ej^n \bigg) , \nonumber
\end{align}
which leads to the five level recursion relation
\begin{align}
\label{eqn:psiupdate4} \psij^{n + 1} = & \   e^{- \gamma \Delta t} \psij^n  + e^{- \gamma \Delta t} \Delta t \bigg( - \frac{1}{24} e^{- 3 \Delta t \gamma} \Ej^{n - 3}   + \frac{5}{24} e^{- 2 \Delta t \gamma} \ \Ej^{n - 2} \\
& \hspace{3.5cm} - \frac{11}{24} e^{- \Delta t \gamma} \Ej^{n - 1} + \frac{23}{24} \Ej^n \bigg) + \frac{1}{3} \Delta t \Ej^{n + 1} . \nonumber
\end{align}

The construction of a second-order accurate approximation of $\phij^n$ follows the approach used in Section~\ref{section:secondorderdrude}, albeit with a different convolution kernel, to obtain
\begin{align}
\label{eqn:phi} \phij^n & \approx   \frac{\Delta t}{2} \sum_{m = 0}^{\infty} \left(  m \Delta t \Ej^{n - m} e^{- \gamma j \Delta t}+ (m + 1) \Delta t \Ej^{n - m - 1} e^{- \gamma (m + 1) \Delta t} \right)  \\
 & \approx   \Delta t^2  \sum_{m = 0}^{\infty} (m + 1) e^{- \gamma (m + 1 ) \Delta t} \Ej^{n - m - 1} . \nonumber
\end{align}
This leads to a five level recursion relation 
\begin{align}
    \label{eqn:phiupdate4} \phij^{n + 1} & =  \Delta t^2 \left(  \sum_{m = 0}^{\infty} m e^{- \gamma (m + 1 ) \Delta t} \Ej^{n - m }  +   \sum_{m = 0}^{\infty} e^{- \gamma (m + 1 ) \Delta t} \Ej^{n - m } \right) \\
    & =    e^{- \Delta t \gamma} \Delta t^2  \sum_{m = 0}^{\infty} (m + 1) e^{- \gamma (m + 1 ) \Delta t} \Ej^{n - m - 1 }  +  \Delta t^2  \sum_{m = 0}^{\infty} e^{- \gamma (m + 1 ) \Delta t} \Ej^{n - m }   \nonumber \\
& = e^{- \Delta t \gamma} \phij^n + \Delta t^2 \left( e^{- 4 \Delta t \gamma} \Ej^{n - 3}  + e^{- 3 \Delta t \gamma} \Ej^{n - 2}  + e^{- 2 \Delta t \gamma} \Ej^{n - 1}  + \Ej^n \ e^{- \Delta t \gamma}  \right) \nonumber \\
& \hspace{5cm} + \Delta t^2 e^{- \Delta t \gamma} \sum_{ m = 4}^{\infty}    e^{- m \Delta t \gamma} \Ej^{n - m} \nonumber \\
& = e^{- \Delta t \gamma} \phij^n + \Delta t \ e^{- \Delta t \gamma} \psij^n \nonumber \\
& \quad + \Delta t^2 \left( - \frac{1}{24} e^{- 4 \Delta t \gamma}  \Ej^{n - 3} +  \frac{1}{6} e^{- 3 \Delta t \gamma}  \Ej^{n - 2} 
-  \frac{7}{24} e^{- 2 \Delta t \gamma}  \Ej^{n - 1} 
 +  \frac{2}{3} e^{-  \Delta t \gamma}  \Ej^n  \right). \nonumber
\end{align}
Here \eqref{eqn:psidisc} has been used in the last equality.

Equations \eqref{eqn:E4th}, \eqref{eqn:psidisc}, and \eqref{eqn:phiupdate4} give a complete five time level update from old times to the new time $t_{n+1}$. For clarity and completeness, the fully fourth-order accurate discretization is given as
\begin{subequations}
 \begin{align}
 \Ej^{n + 1}  = &  \ 2 \Ej^n + \Ej^{n - 1} +  \Delta t^2 \left[ c^2 \Delta_{4h} \Ej^n  - \frac{\omega_p^2}{\epsilon_r}  \Ej^n +  \frac{\omega_p^2 \gamma}{\epsilon_r}  \psij^n \right]  \label{eqn:4thscheme1} \\
		&  \hspace{3cm} +  \frac{ \Delta t^4}{12}  \bigg[ c^4 \Delta^2_{2h} \Ej^n  -  \frac{2 c^2 \omega_p^2}{\epsilon_r}  \Delta_{2h} \Ej^n +  \frac{2 c^2 \omega_p^2 \gamma}{\epsilon_r}  \Delta_{2h} \psij^n  \nonumber \\ 
		& \hspace{3cm} +  \frac{\omega_p^4}{\epsilon_r^2}  \Ej^n -  \frac{2 \omega_p^4 \gamma}{\epsilon_r^2}  \psij^n  + \frac{\omega_p^4 \gamma^2 }{\epsilon_r^2}   \phij^n \bigg], \nonumber  \\
	 \psij^{n + 1} = & \   e^{- \gamma \Delta t} \psij^n + \frac{1}{3} \Delta t \Ej^{n + 1}  \label{eqn:4thscheme2}    \\
& + e^{- \gamma \Delta t} \Delta t \left( - \frac{1}{24} e^{- 3 \gamma \Delta t} \Ej^{n - 3} + \frac{5}{24} e^{- 2 \gamma \Delta t} \ \Ej^{n - 2} - \frac{11}{24} e^{- \gamma \Delta t} \Ej^{n - 1} + \frac{23}{24} \Ej^n \right)  , \nonumber   \\
	 \phij^{n + 1}  = & \ e^{- \gamma \Delta t}  \phij^n + \Delta t \ e^{- \gamma \Delta t} \psij^n \label{eqn:4thscheme3} \\
&   + \Delta t^2 \left( - \frac{1}{24} e^{- 4 \gamma \Delta t}  \Ej^{n - 3} +  \frac{1}{6} e^{- 3 \gamma \Delta t}  \Ej^{n - 2} 
-  \frac{7}{24} e^{- 2 \gamma \Delta t}  \Ej^{n - 1} 
 +  \frac{2}{3} e^{-  \gamma \Delta t}  \Ej^n  \right). \nonumber
 \end{align}
 \label{eqns:4th}
 \end{subequations}

\subsection{Stability of the Fourth-Order Accurate Scheme for Drude Media} 
\label{section:4thorderstab}
Numerical stability of the discrete system given in \eqref{eqns:4th} is investigated following a similar approach to that of Section~\ref{section:2ndorderstab}. However, for the present case of the fourth-order accurate discretization, the result is less subtle than the corresponding result for the second-order accurate scheme. In particular, it is found that the system \eqref{eqns:4th} does not admit exponential growth for a range of dimensionless parameters which depend on the grid spacing and time step. This region, or a similar more restrictive but algebraically simpler region, can then be used in the determination of the time step. Thus the result here for the fourth-order scheme falls into the more intuitive notion of stability for wave equations which entirely disallows exponential growth. 

Similar to Section~\ref{section:2ndorderstab}, the exact solution to the discrete system \eqref{eqns:4th} employs a rescaling of the discrete unknowns with $\Psij^n \equiv \gamma \psij^n$ and $\Phij^n \equiv \gamma^2 \phij^n$. The relevant non-dimensional parameters are likewise identical to the prior case, and are given in Equation \eqref{eqn:stabscaling}. The discrete equations \eqref{eqns:4th} are now rescaled to obtain
\begin{subequations}
\begin{align}
  \Ej^{n + 1}  = &  \ 2 \Ej^n + \Ej^{n - 1} + \LapFour\Ej^n  - \Omega^2 \Ej^n +  \Omega^2 \Psij^n +  \frac{ 1}{12}  \big( \LapTwo^2 \Ej^n  - 2\Omega^2\LapTwo \Ej^n    \label{eqn:4thscheme1-2}   \\
		 & \hspace{4cm} + 2 \Omega^2\LapTwo \Psij^n +  \Omega^4 \Ej^n - 2 \Omega^4 \Psij^n  + \Omega^4 \Phij^n \big), \nonumber \\
	 \Psij^{n + 1} = & \   e^{- \Gamma } \Psij^n + \frac{1}{3} \Gamma  \Ej^{n + 1} \label{eqn:4thscheme2-2} \\
& + e^{- \Gamma } \Gamma \left( - \frac{1}{24} e^{- 3 \Gamma } \Ej^{n - 3} + \frac{5}{24} e^{- 2 \Gamma  } \Ej^{n - 2} - \frac{11}{24} e^{- \Gamma } \Ej^{n - 1} + \frac{23}{24} \Ej^n \right)  , \nonumber  \\
	 \Phij^{n + 1}  = & \ e^{- \Gamma } \Phij^n + \Gamma \ e^{- \Gamma } \Psij^n \label{eqn:4thscheme3-2} \\
&   + \Gamma^2 \left( - \frac{1}{24} e^{- 4 \Gamma }  \Ej^{n - 3} +  \frac{1}{6} e^{- 3 \Gamma }  \Ej^{n - 2} 
-  \frac{7}{24} e^{- 2 \Gamma }  \Ej^{n - 1}   +  \frac{2}{3} e^{-  \Gamma }  \Ej^n  \right), \nonumber 
\end{align}
\label{eqns:scaled4th}
\end{subequations}
where the following definitions of undivided operators have been used for simplicity:
\begin{align*}
  \LapTwo & = \sum_{d=1}^{\mathcal{D}}\lambda_d^2 \delta_{+,x_d}\delta_{-,x_d} \\
  \LapFour & = \sum_{d=1}^{\mathcal{D}}\lambda_d^2 \delta_{+,x_d}\delta_{-,x_d}\left(I-\frac{1}{12}\delta_{+,x_d}\delta_{-,x_d}\right).
\end{align*}
Separable solutions to the linear difference equations \eqref{eqns:scaled4th} are sought using the ansatz
\begin{align}
\Ej^n = \mathbf{C}_E A^n e^{i \, \kv\cdot\xv_{\jv} } ,  \qquad \Psij^n = \mathbf{C}_{\Psi}A^n e^{i \, \kv\cdot\xv_{\jv} } ,  \qquad \Phij^n = \mathbf{C}_{\Phi} A^n e^{i \, \kv\cdot\xv_{\jv} } ,
\end{align}
which yields the linear system
\begin{align}
\mathcal{M}(A)  \begin{bmatrix}
    \mathbf{C}_E \\ \mathbf{C}_{\Psi}\\ \mathbf{C}_{\Phi}
 \end{bmatrix} \equiv  \begin{bmatrix}
   M_{EE} & M_{E\Psi} & M_{E\Phi} \\
   M_{\Psi E} & M_{\Psi\Psi} & 0\\
   M_{\Phi E} & M_{\Phi\Psi} & M_{\Phi \Phi} \\
 \end{bmatrix}
 \begin{bmatrix}
    \mathbf{C}_E \\ \mathbf{C}_{\Psi}\\ \mathbf{C}_{\Phi}
 \end{bmatrix}
 =
 \begin{bmatrix}
   0 \\ 0 \\ 0
 \end{bmatrix}
 \label{eqns:Fourier4th}
\end{align}
where 
\begin{align*}
  M_{EE} & = -A^2+2A-1 + A\rho_{4h}-A\Omega^2 + A\frac{1}{12}\left(\rho_{2h}^2 - 2\rho_{2h}\Omega^2+\Omega^4\right)\\
  M_{E\Psi} & = A\Omega^2+\frac{1}{12}A\left( 2 \Omega^2 \rho_{2h}- 2 \Omega^4\right)\\
  M_{E\Phi} & = \frac{1}{12}A\Omega^4\\
  M_{\Psi E} & = \frac{1}{3}  A^4\Gamma + e^{- \Gamma } \Gamma \left( - \frac{1}{24} e^{- 3 \Gamma } + \frac{5}{24} A  e^{- 2 \Gamma  } - \frac{11}{24}  A^2 e^{- \Gamma } + \frac{23}{24} A^3  \right)\\
  M_{\Psi\Psi} & = -A^4+A^3e^{- \Gamma } \\
  M_{\Phi E} & = \Gamma^2 \left( - \frac{1}{24} e^{- 4 \Gamma } + \frac{1}{6}  A e^{- 3 \Gamma } -  \frac{7}{24}  A^2 e^{- 2 \Gamma } + \frac{2}{3} A^3  e^{-  \Gamma }   \right)\\
  M_{\Phi\Psi} & = A^3 \Gamma \ e^{- \Gamma } \\
  M_{\Phi \Phi} & = -A^4+ A^3e^{- \Gamma }, 
\end{align*}
and where the scaled symbols of the difference operators are given as
\begin{align*}
  \rho_{4h} & = \sum_{d=1}^{\mathcal{D}}\left[\lambda_d^2\left( -\frac{7}{3}+\frac{8}{3} \cos( \xi_d ) -\frac{1}{3} \cos^2(\xi_d)\right)\right]\\
  \rho_{2h} & = \sum_{d=1}^{\mathcal{D}}\left[\lambda_d^2\left( -4\sin^2(\xi_d) \right)\right].
\end{align*}

The solvability condition for nontrivial solutions of \eqref{eqns:Fourier4th} gives a 10th order polynomial in A defined by
\begin{align}
\text{Det} \ \mathcal{M(A)} = 0, \label{eqn:detcond}
\end{align}
which has four trivial roots A = 0. Again, since $\mathbf{C}_E, \mathbf{C}_{\Psi}$, and $\mathbf{C}_{\Phi}$ have three components, each of these ten roots is actually of multiplicity three (or more depending on their repetition) in the full problem. For $\lambda_d$, $\Omega$, and $\Gamma$ defined in equation \eqref{eqn:stabscaling}, the roots of this polynomial can be
evaluated numerically for a given set of wave numbers $\xi_d\in[-\pi,\pi]$. Given these preliminaries, we are now in a position to discuss the stability of the fourth-order accurate recursive convolution algorithm of \eqref{eqns:4th}. Of primary importance is the existence of a time-step restriction determined by the sufficient condition $ \max |A| \leq 1$ that ensures the scheme exhibits no exponential growth and therefore long-time stability, as summarized in the following proposition.
\begin{myproposition}
Let $A$ be the solutions to the solvability condition \eqref{eqn:detcond} for solutions of the fourth-order accurate recursive convolution algorithm of \eqref{eqns:4th}. A sufficient condition for $\max |A|\le1$ is given by
\begin{subequations}
\begin{align}
  \frac{4}{5}\Omega^4+16\Lambda^2-\frac{72}{5}\Omega^2-64\Lambda+48\ge0, & \qquad \hbox{and} \label{eq:RC4_time_step} \\
  |\Omega|<2 & \qquad \hbox{and} \\
  \Gamma\le0.68. \label{eq:RC4_GammaBound}& 
\end{align}
\label{eqns:RC4Bound}
\end{subequations}
Note that if the bound \eqref{eqns:RC4Bound} is satisfied then $|A|<1$, and therefore the scheme is clearly stable under Definition~\ref{def:stab} even in the infinite time limit when the final time $T\to\infty$. Note also that \eqref{eq:RC4_time_step} is a simple quadratic equation in the time step $\dt$, making the entire bound \eqref{eqns:RC4Bound} amenable to simple computation of the time step.
\end{myproposition}
\begin{proof}
The solvability condition for \eqref{eqns:Fourier4th} for nontrivial solutions gives a 10th order polynomial in $A$, with four trivial roots $A=0$. Given $\lambda_d$, $\Omega$, and $\Gamma$, as in Equation \eqref{eqn:stabscaling}, the roots of this polynomial can be evaluated numerically for a given set of wave numbers $\xi_d\in[-\pi,\pi]$. Maximizing the size of this amplification factor over the range of wave numbers reveals that the maximum occurs when either $\xi_d=0$, or $\xi_d=\pm\pi$. When $\xi_d=0$ the stability of the algorithm depends only on the dispersive and dissipative terms (i.e. $A$ is pure real similar to the case $A_0$ for the second-order scheme above), and the roots do not depend on $\lambda_d$. When $\xi_d=\pm\pi$, the dependence on the individual $\lambda_d$ is simplified to dependence on $\Lambda$. Therefore in the space of relevant parameters $\Lambda$, $\Omega$, and $\Gamma$, the unity iso-surface of the maximum $|A|$, defined by
\begin{align}
\mathcal{S}_3 = \bigg\{ ( \Omega, \Gamma, \Lambda ) \ : \  \underset{A}{ \max } \big|A(\Omega, \Gamma, \Lambda) \big| = 1 \bigg\} ,  \label{eqn:isoset3}
\end{align}
can be computed, and is displayed on the left of Figure~\ref{fig:RC4_iso}.
\begin{figure}
  \begin{center}
  \includegraphics[width=.325\textwidth]{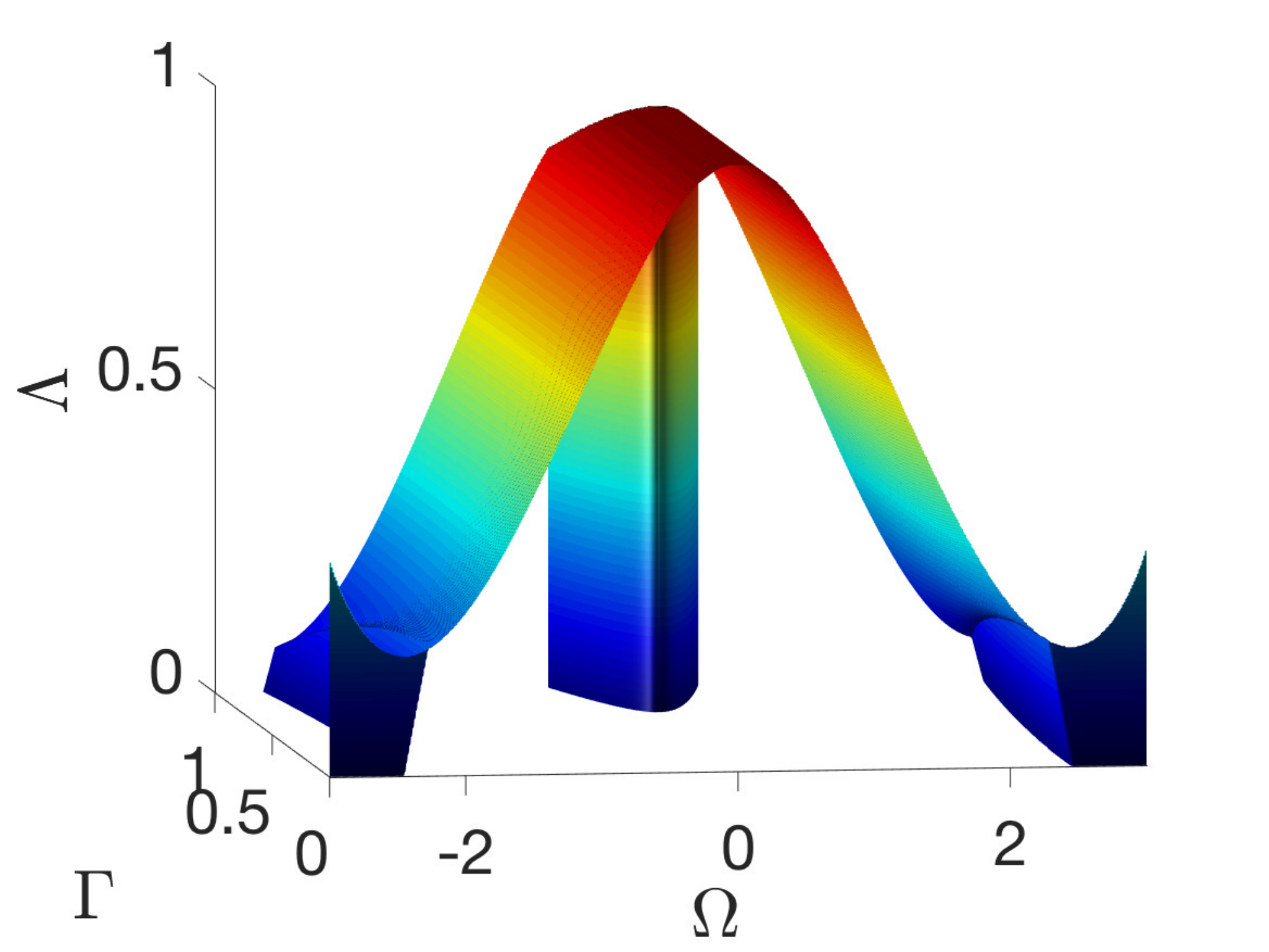} \hfill
  \includegraphics[width=.325\textwidth]{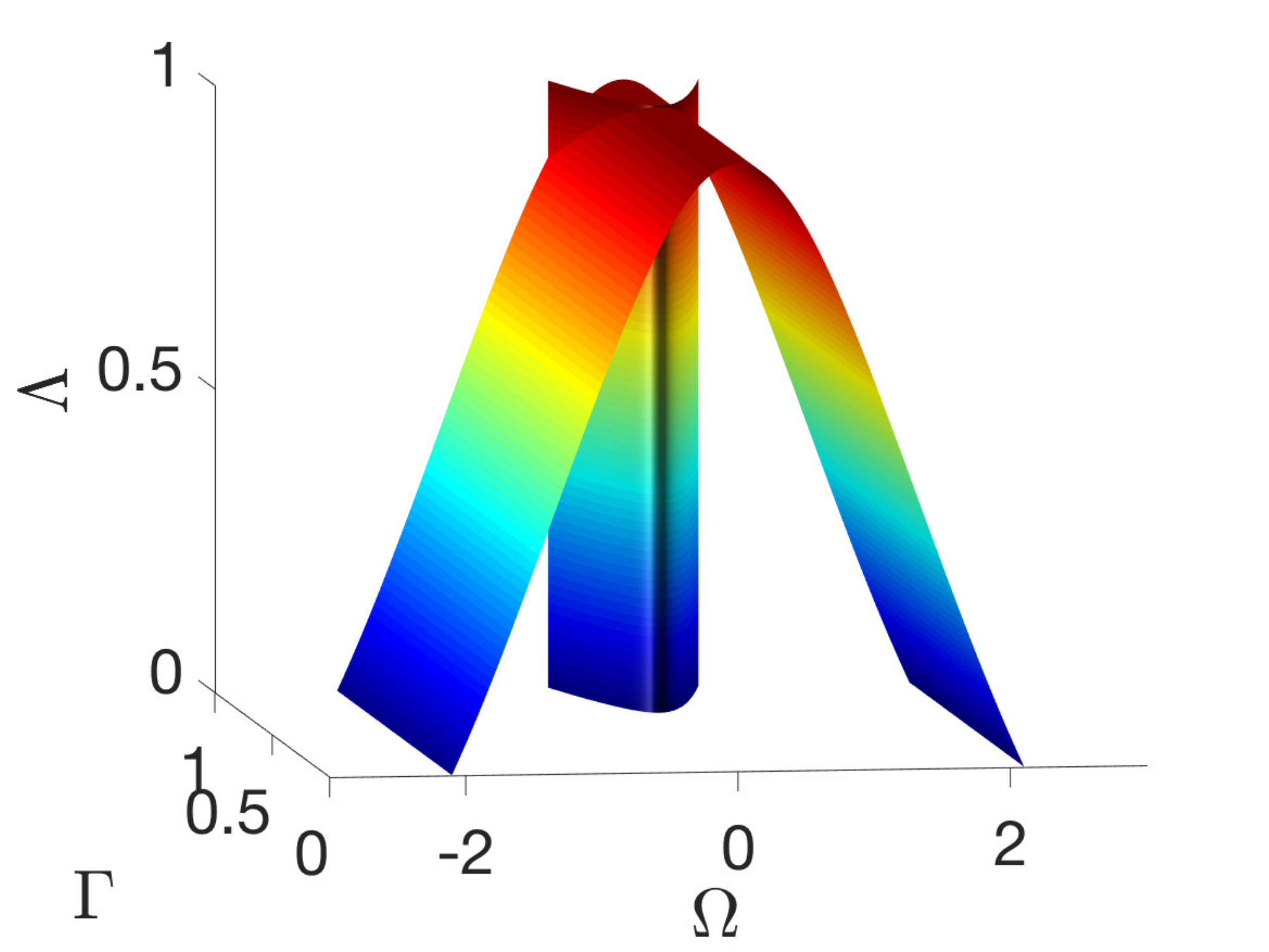} \hfill
  \includegraphics[width=.325\textwidth]{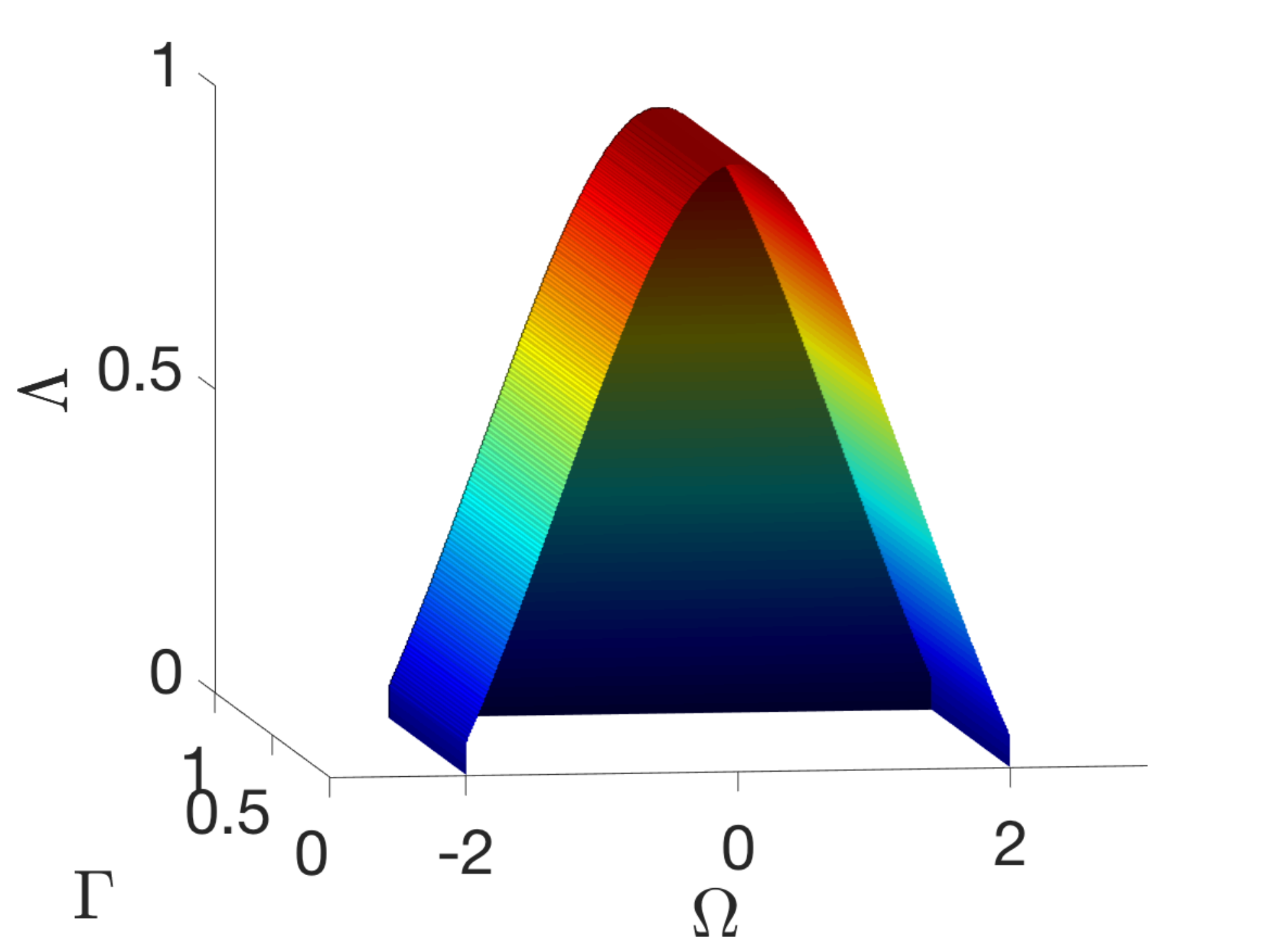}
  \caption{At left is the unity iso-surface of the maximum of $|A|$ in the space of dimensionless parameters $\Omega$, $\Gamma$, and $\Lambda$. In the middle are simplified surfaces defined by \eqref{eq:RC4_time_step} in blue and \eqref{eq:RC4GamBound} in red. At right is the entire bounding surface defined by \eqref{eqns:RC4Bound} which is used in determining the time step size.}
  \label{fig:RC4_iso}
  \end{center}
\end{figure}
The curve defining the ``roof'' of this iso-surface at $\Gamma=0$ is easily found as the solution to $\Omega^4+16\Lambda^2-12\Omega^2-64\Lambda+48=0$. It is tempting to simply extend this curve uniformly in the $\Gamma$ direction, but a close inspection shows that slight modification is needed since the actual iso-surface becomes slightly more restrictive as $\Gamma\to.68$. One such modification is given in \eqref{eq:RC4_time_step}, which is shown in the center and left panels of Figure~\ref{fig:RC4_iso}. Looking again at the iso-surface defining $|A|=1$, any simplified bound must also be limited in both the $\Omega$ and $\Gamma$ directions. For the former, the branch cut is avoided by simply restricting $|\Omega|<2$ (again not a tight bound). For the latter, the curve defining the limit for increasing $\Gamma$ is found to solve the equation
\begin{align}
  \frac{2}{3}\Gamma\Omega^2+\Omega^2+4\Gamma-12
  +e^{-\Gamma}\left(\Gamma^2\Omega^2+\Gamma\left[-\frac{5}{4}\Omega^2+\frac{15}{2}\right]-2\Omega^2+24\right)   \label{eq:RC4GamBound}
  \\
   +e^{-2\Gamma}\left(\frac{17}{6}\Gamma\Omega^2+\Omega^2-17\Gamma-12\right) 
  +e^{-3\Gamma}\Gamma\left(-\frac{4}{3}\Omega^2+8\right) \nonumber \\
  +e^{-4\Gamma}\Gamma\left(\frac{1}{2}\Omega^2-3\right) 
  -\frac{1}{12}e^{-5\Gamma}\Gamma\left(\Omega^2-6\right)=0, \nonumber
\end{align}
the most restrictive of which lies along $\Omega=0$ where $\Gamma<.6889953407$. Again this is reduced slightly in the bound given in \eqref{eq:RC4_GammaBound}. The complete simplified bound \eqref{eqns:RC4Bound} is illustrated in the right panel of Figure~\ref{fig:RC4_iso}.
\end{proof}

We finally summarize the appropriate choice of $\Delta$ for a stable fourth-order algorithm with the following proposition.

\begin{myproposition} 
Given physical parameters $c$, $\omega_p^2/\epsilon_r$, and $\gamma$, define $\Delta t_m$ to be the smallest root of the fourth order polynomial equation
\begin{align}
\left(  \frac{4 \omega_p^4  }{5 \epsilon_r^2} \right) \Delta t^4 + \left( 16 c \sqrt{\sum_{d=0}^{\mathcal{D}}\frac{1}{h_d^2} } - \frac{72 \omega_p^2}{5 \epsilon_r} \right) \Delta t^2 - \left( \frac{64 c}{\Delta x} \right) \Delta t + 48 = 0 .
\end{align}
A practical sufficient condition for stability of the second-order accurate recursive convolution algorithm of \eqref{eqns:RC2} is given by
\begin{align}
\Delta t = \min \left\{ \Delta t_m,  \frac{0.68}{\gamma} \right \} .
\end{align}
\end{myproposition}
\noindent The proof follows from \eqref{eq:RC4GamBound} along with the choice $\Gamma \leq .68$ with the substitutions $\Lambda = c\Delta t\sqrt{\sum_{d=0}^{\mathcal{D}}\frac{1}{h_d^2} } $, $\Omega = \Delta t \omega_p / \sqrt{\epsilon_r}$, and $\Gamma = \Delta t \gamma$ given in \eqref{eqn:stabscaling}.

\section{Numerical Compatibility Conditions at Material Interfaces} 
\label{sec:interface}
The discrete treatment of the interface conditions \eqref{eqns:int} separating disparate materials is now considered, recalling that in, the present discussion, each material is assumed to have constant material properties $\epsilon_r,   \mu_r,  \omega_p,$ and $\gamma$. Following the approach used by Henshaw in \cite{henshaw_2006}, which considered the case of non-dispersive dielectric media, ghost cells are introduced in order to simplify the application of the interface conditions. The discretization will be applied up to, and including, the interface, and so a sufficient number of ghost cells are required to support the spatial discretization stencil. In particular, for an order $p$ discretization (with $p$ assumed even), $p/2$ ghost cells will be needed. Thus for higher-order schemes, additional equations must be derived to define the data in the ghost cells. The additional equations will be called ``compatibility conditions,'' because the equations are compatible with both the interior and interface conditions. One important note is that the recursive convolution formulas, e.g. \eqref{eqn:recursionpsi} for the second-order scheme or \eqref{eqn:4thscheme2} and \eqref{eqn:4thscheme3} for the fourth-order scheme, are applied in the ghost cells to define the auxiliary quantities $\psiv$ and possibly $\phiv$ from prior time levels. In addition, since the governing equations \eqref{eqns:governing} are expressed in their second-order formulation, the magnetics do not directly appear in the mathematical formulation. Taken together this implies that it is sufficient to derive compatibility conditions for $\mathbf{E}$ (or equivalently $\mathbf{D}$).

Equations \eqref{eqns:int} along with the governing PDE \eqref{eqns:governing} provide a complete description of the problem at hand. However, as discussed above, additional compatible interface conditions will be derived at a discrete level in order to specify the normal and tangential components of the fields in the ghost cells. These additional conditions will take the form of derivative jump conditions of increasing order, the primal conditions \eqref{eqns:int} being undifferentiated, as required by the numerical scheme. First derivative jump conditions on the normal components of the fields are derived by directly integrating equation \eqref{eqn:gauss} across the interface. For the tangential components of the fields, equation \eqref{eqn:jumps3} is differentiated in time, and the definition of the time derivative from governing equation \eqref{eqn:macro1} is used. These two first-derivative jump conditions are
\begin{subequations}
\begin{align}
  \left[ \nabla \cdot \mathbf{E} \right]_{\mathcal{I}} = 0, \label{eqn:firstDerivativeJumps_normal} \\
  \left[ \mu^{-1} \mathbf{n} \times \nabla \times \mathbf{E} \right]_{\mathcal{I}} = 0.\label{eqn:firstDerivativeJumps_tangential} 
\end{align}
\label{eqns:firstDerivativeJumps}
\end{subequations}
Second-derivative jump conditions are derived similarly by differentiating \eqref{eqn:jumps2} and \eqref{eqn:jumps1} twice with respect to time, followed by substitution of the time derivatives as defined in the governing equations \eqref{eqn:waveE2} and \eqref{eqn:vectorid}, respectively, to yield\begin{subequations}
\begin{align}
  \left[ \mu^{-1}  \mathbf{n} \cdot \Delta \mathbf{E} \right]_{\mathcal{I}} = 0, \label{eqn:secondDerivativeJumps_normal}\\
  \left[ ( \epsilon_0 \epsilon_r \mu )^{-1}  \mathbf{n} \times  ( \Delta + \mu \eta \ * ) \mathbf{E} \right]_{\mathcal{I}} = 0. \label{eqn:secondDerivativeJumps_tangential} 
\end{align}
\label{eqns:secondDerivativeJumps}
\end{subequations}
Equations \eqref{eqn:firstDerivativeJumps_normal} and \eqref{eqn:secondDerivativeJumps_normal} represent two conditions that, upon discretization by standard second-order accurate centered differences, define the normal component of $\mathbf{E}$ in two ghost cells, one on either side of an interface, as appropriate for a second-order scheme. Similarly, Equations\eqref{eqn:firstDerivativeJumps_tangential} and \eqref{eqn:secondDerivativeJumps_tangential} can be used to define the tangential components in the same two ghost cells.

Additional derivative conditions are derived by taking additional time derivatives of \eqref{eqns:firstDerivativeJumps} and \eqref{eqns:secondDerivativeJumps}, and making use of \eqref{eqn:higherderivs} to give a full set of interface conditions for $q = 1, 2, 3, \dots$:
\begin{subequations}
\begin{align}
  \left[ \epsilon_0  \epsilon_r (\epsilon_0  \epsilon_r \mu)^{-q} \nabla \cdot \Delta \left( \Delta +  \mu \eta \ * \right)^{q - 1} \mathbf{E} \right]_{\mathcal{I}} & = 0,  \\
  \left[  \mu^{-1} ( \epsilon_0 \epsilon_r \mu )^{-q} \mathbf{n} \times  \nabla \times ( \Delta +  \mu \eta \ * )^q \mathbf{E} \right]_{\mathcal{I}}  & = 0, \\
  \left[ \epsilon_0  \epsilon_r ( \epsilon_0  \epsilon_r \mu )^{-q} \mathbf{n} \cdot   \Delta ( \Delta +  \mu \eta \ * )^{q - 1}  \mathbf{E} \right]_{\mathcal{I}} & = 0, \\
 \left[  ( \epsilon_0  \epsilon_r \mu )^{-q} \mathbf{n} \times  ( \Delta +  \mu \eta \ * )^q \mathbf{E} \right]_{\mathcal{I}} & = 0.
\end{align}
\label{eqns:generalJumps}
\end{subequations}
Discrete versions of \eqref{eqns:generalJumps} are then used as numerical compatibility conditions in defining the fields in the ghost cells on either side of an interface. For instance with $q=1$, Equations \eqref{eqns:generalJumps} are identical to Equations \eqref{eqns:firstDerivativeJumps} and \eqref{eqns:secondDerivativeJumps}, which is appropriate for the second-order discretization as previously discussed. For the fourth-order accurate discretization one requires 2 ghost cells on either side of an interface, and so the equations in \eqref{eqns:generalJumps} would need to be used with both $q=1$ and $q=2$. Discretizing the resulting equations using standard centered differences in a $7$-pt stencil yields a complete definition of the fields in the ghost cells. To be concrete, specific examples using the Drude model are discussed below.

\subsection{Numerical Interface Conditions for the Second-Order Scheme with the Drude Model in 2D} \label{section:secondorderinterface}
As an example, consider an interface $\mathcal{I}$ at $x = 0$ in $\mathbb{R}^2$ separating two dispersive Drude media. Here $E_x$ denotes the normal component of the electric field, and without loss of generality consider $E_y$ to be the tangential component of the field (other polarizations of light are similar). Similarly, let $\psi_x$ and $\psi_y$ be the analogous normal and tangential components of the auxiliary field $\psiv$. For the second-order accurate scheme \eqref{eqns:RC2}, one ghost cell is required on each side of the interface, and therefore two compatibility conditions for the values of the normal and tangential electric field components of $\mathbf{E}$ are needed. The equations in \eqref{eqns:generalJumps} with $q=1$ yield the four conditions
\begin{subequations}
\begin{align}
  \left[   \partial_x E_x + \partial_y E_y  \right]_{\mathcal{I}} &= 0, \label{eqn:2ndorderJump1} \\
  \left[ \mu^{-1} \left( \partial_y E_x - \partial_x E_y  \right) \right]_{\mathcal{I}} &= 0, \label{eqn:2ndorderJump2}\\
  \left[ \mu^{-1} \Delta E_x \right]_{\mathcal{I}} & = 0, \label{eqn:2ndorderJump3} \\
  \left[ ( \epsilon_0 \epsilon_r \mu )^{-1}  ( \Delta E_y - \mu  \omega_p^2 E_y + \mu \omega_p^2 \gamma \psi_y )  \right]_{\mathcal{I}} & =  0 .  \label{eqn:2ndorderJump4}
\end{align}
\label{eqns:2ndOrderJumps}
\end{subequations}
Discrete equations are then derived by replacing derivatives with standard second-order accurate centered differencing. In particular, Equations \eqref{eqns:2ndOrderJumps} become
\begin{subequations}
\begin{align}
  \left[   D_{0,x} E_x + D_{0,y} E_y  \right]_{\mathcal{I}} &= 0, \label{eqn:2ndorderJump1_discrete} \\
  \left[ \mu^{-1} \left( D_{0,y} E_x - D_{0,x} E_y  \right) \right]_{\mathcal{I}} &= 0, \label{eqn:2ndorderJump2_discrete}\\
  \left[ \mu^{-1} \Delta_{2h} E_x \right]_{\mathcal{I}} & = 0, \label{eqn:2ndorderJump3_discrete} \\
  \left[ ( \epsilon_0 \epsilon_r \mu )^{-1}  ( \Delta_{2h} E_y - \mu  \omega_p^2 E_y + \mu \omega_p^2 \gamma \psi_y )  \right]_{\mathcal{I}} & =  0,  \label{eqn:2ndorderJump4_discrete}
\end{align}
\label{eqns:2ndOrderJumps_discrete}
\end{subequations}
where $D_{0,d}=(D_{+,d}+D_{-,d})/2$. The discrete equations \eqref{eqns:2ndOrderJumps_discrete} determine the value of the normal and tangential fields in the ghost cells. Note from the second-order interior scheme \eqref{eqns:RC2} that the auxiliary field $\psiv$ is not differentiated and so it need not be updated in the ghost. Nonetheless, in practice, it is prudent to maintain an updated value for the auxiliary quantity even in the ghost, and for this purpose \eqref{eqn:recursionpsi} is applied.
 
\subsection{Numerical Interface Conditions for the Fourth-Order Scheme with the Drude Model in 2D}
\label{section:fourthorderinterface}
As in section \ref{section:secondorderinterface}, consider the interface $\mathcal{I}$ at $x = 0$ separating two dispersive Drude media. Here however, consider using the fourth-order accurate scheme \eqref{eqns:4th} which requires two ghost cells on either side of the interface. As before, $E_x$ and $E_y$ are the normal and tangential components of the field, $\psi_x$ and $\psi_y$ are the corresponding components of the auxiliary field $\psiv$, and now $\phi_x$ and $\phi_y$ are the normal and tangential components of the auxiliary field $\phiv$. In all ghost cells, the components of $\psiv$ are updated via \eqref{eqn:4thscheme2}, and the components of $\phiv$ are updated via \eqref{eqn:4thscheme3}. The remaining two components of the electric field in the ghost cells then require $8$ conditions, which are obtained from Equations \eqref{eqns:generalJumps} with $q=1$, and $q=2$ as
\begin{subequations}
\begin{align}
  \left[   \partial_x E_x + \partial_y E_y \right]_{\mathcal{I}} & = 0, \label{eqn:4thorderJumps1} \\
  \left[ \mu^{-1} \left( \partial_y E_x - \partial_x E_y  \right) \right]_{\mathcal{I}} & = 0, \label{eqn:4thorderJumps2}\\
  \left[ \mu^{-1} \Delta E_x \right]_{\mathcal{I}} & = 0,  \label{eqn:4thorderJumps3}\\
  \left[ ( \epsilon_0 \epsilon_r \mu )^{-1}  ( \Delta E_y - \mu  \omega_p^2 E_y + \mu \omega_p^2 \gamma \psi_y )  \right]_{\mathcal{I}} & =  0 , \label{eqn:4thorderJumps4}\\
  \left[ \mu^{-1} \left( \partial_x \Delta E_x + \partial_y \Delta E_y \right) \right]_{\mathcal{I}} & = 0, \label{eqn:4thorderJumps5}\\
  \big[  \mu^{-1} ( \epsilon_0 \epsilon_r \mu )^{-1} \big(  \partial_y ( \Delta E_x - \mu  \omega_p^2 E_x + \mu \omega_p^2 \gamma \psi_x ) \hspace{1cm} \label{eqn:4thorderJumps6} \\
   - \ \partial_x ( \Delta E_y - \mu  \omega_p^2 E_y + \mu \omega_p^2 \gamma \psi_y )  \big) \big]_{\mathcal{I}} & = 0, \nonumber \\
  \left[ \epsilon_0  \epsilon_r ( \epsilon_0  \epsilon_r \mu )^{-2}    \Delta ( \Delta E_x - \mu  \omega_p^2 E_x + \mu \omega_p^2 \gamma \psi_x )  \right]_{\mathcal{I}}  & =  0  ,  \label{eqn:4thorderJumps7} \\
  \left[ \Delta^2 E_y  + 2 \mu  \epsilon_0 \omega_p^2 \left( - \Delta E_y +   \gamma \Delta \psi_y \right) + \mu^2  \epsilon_0^2  \omega_p^4 \left(  E_y -  2   \gamma \psi_y  +   \gamma^2  S \right)  \right]_{\mathcal{I}} & =  0 . \label{eqn:4thorderJumps8}
\end{align}
\label{eqns:4thorderJumps}
\end{subequations}
As before, discrete equations are derived from Equations \eqref{eqns:4thorderJumps} by replacing derivative operators with difference operators. As described in~\cite{henshaw_2006}, fourth-order accuracy is obtained when Equations \eqref{eqn:4thorderJumps1} through \eqref{eqn:4thorderJumps4} use fourth-order accurate centered differences (e.g. $\partial_x\approx D_{0,x}(I-\frac{h_x^2}{6}D_{+,x}D_{-,x})$), while Equations \eqref{eqn:4thorderJumps5} through \eqref{eqn:4thorderJumps8} may use second-order accurate centered differences (e.g. $\partial_{xxx}\approx D_{0,x}D_{+,x}D_{-,x}$). In this way the entire discrete version of Equations \eqref{eqns:4thorderJumps} lies in a 7-pt stencil, and thus defines the fields in the ghost cells.

\section{Numerical Example} 
\label{section:six}
We now present a number of example computations using the schemes described in Section~\ref{section:scheme} above. In the first problem a simple traveling plane wave is used to confirm the theoretical predictions discussed in Sections~\ref{section:4thorderstab} and~\ref{section:2ndorderstab}. In particular, the existence of exponentially growing modes is shown for the second-order scheme, while at the same time second-order convergence is also demonstrated. For the same problem no exponential growth at all is observed for the fourth-order scheme, and the expected fourth-order convergence is demonstrated. Subsequently, we will present a number of sample computations of more practical relevance. In particular, second and fourth-order max-norm convergence will be illustrated for the problem of scattering between two disparate materials where the incident and reflecting medium is a dielectric and the refracting medium obeys a dispersive Drude description. Lastly, we investigate the propagation of a surface plasmon polariton mode along the same interface and perform an analogous convergence study.

\subsection{Periodic Dissipative Plane Wave with Application to Stability}
\label{sec:per1d}

To illustrate the theoretical predictions discussed in Section~\ref{section:scheme} above, consider a periodic one-dimensional traveling wave solution for the Drude model of the form 
\begin{align}
E(x,t) = E_0 e^{i k x} e^{s t}.  \label{eqn:dissipative_plane_wave}
\end{align}
Here $k$ is a real valued wave number, $E$ represents any of the electric field components, and $s \in \mathbb{C}$ is to be determined. Substitution of \eqref{eqn:dissipative_plane_wave} into the one-dimensional restriction of Equation \eqref{eqn:governingE} gives the dispersion relation
\begin{align}
s^3 + \gamma s^2 + c^2 k^2 s + \left( \gamma c^2 k^2 + \frac{ \omega_p^2}{\epsilon_r} \right) = 0. \label{eqn:dispersionrelation}
\end{align}
Because the discriminant of Equation \eqref{eqn:dispersionrelation} is negative, there are two complex-valued solutions corresponding to dissipative waves propagating to the left and right, and a third real-valued solution corresponding to a strictly decaying wave. In the discussion to follow, the right-moving wave is chosen, although there is nothing particularly unique to this mode. 

With the goal of demonstrating exponential growth of the approximation for low modes as predicted by Proposition~\ref{prop:RC2Growth}, coefficients are chosen which correspond to a very highly damped material. In particular, $c=1$, $\omega_p=3$, $\epsilon_r=1$, and $\gamma=10$. The physical domain is set to be $x\in[-\pi,\pi]$, periodic boundary conditions are used, and the wavenumber is $k=5$. For reference, these parameters yield a solution to the dispersion relation \eqref{eqn:dispersionrelation} with $s=5.185973- .3765531 i$.  The leftmost panel of Figure~\ref{fig:RC2_conv} shows the time history of the maximum error in the computation using the second-order accurate scheme with the time step chosen to be $99\%$ of the bound given by \eqref{eq:RC2_time_step} for different numbers of grid points in the domain. As expected, all solutions exhibit early time decay, but for the coarser resolutions exponential growth which was seeded by numerical noise on the order of $10^{-16}$ becomes apparent near $t=65$. To more clearly display the nature of that growth, the center panel of Figure~\ref{fig:RC2_conv} shows the numerical solution with $N=101$ points in the domain just prior to, during, and just after the transition from decay to growth. Here the most prominent growing mode is clearly the constant mode, as predicted by the theory. The final panel in Figure~\ref{fig:RC2_conv} shows a convergence study at $t=20$ which clearly indicates convergence of the numerical approximation at the expected second-order rate. Note that later time convergence, for example $t=200$, is much faster than second-order prior to saturation at machine precision, since the theory predicts the spurious growth rate to be $O(\dt^2)$ in physical time $t$. The computations are in good agreement with this prediction given that the numerical seed is machine round-off. Figure~\ref{fig:RC4_conv} shows a similar set of numerical results using the fourth-order accurate scheme with a time-step taken to be $99\%$ of the bound given in \eqref{eqns:RC4Bound}. As expected, there are no spurious exponentially growing modes, and fourth-order max-norm convergence is observed.

\begin{figure}[hbt]
  \begin{center}
  \includegraphics[width=.325\textwidth]{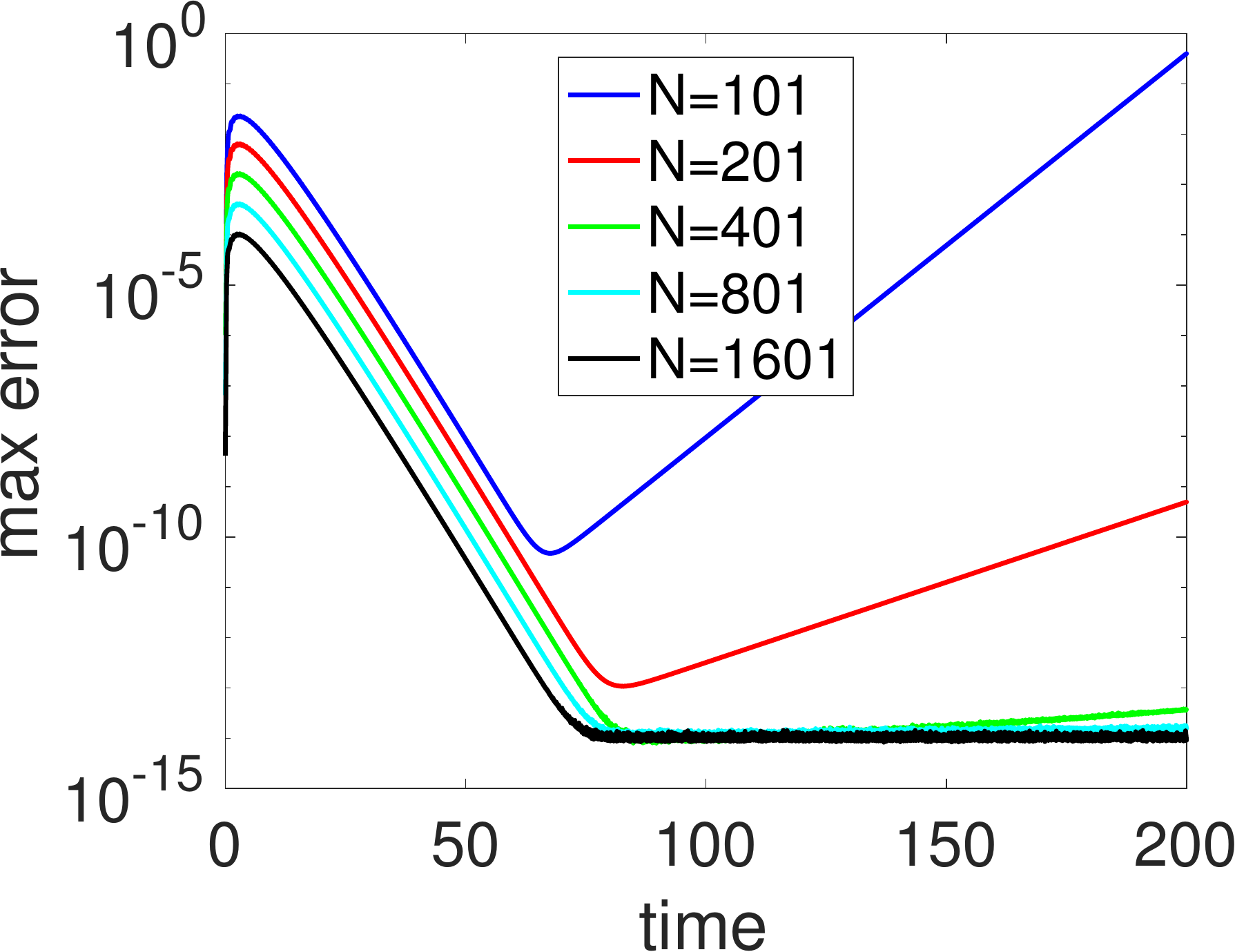} \hfill
  \includegraphics[width=.325\textwidth]{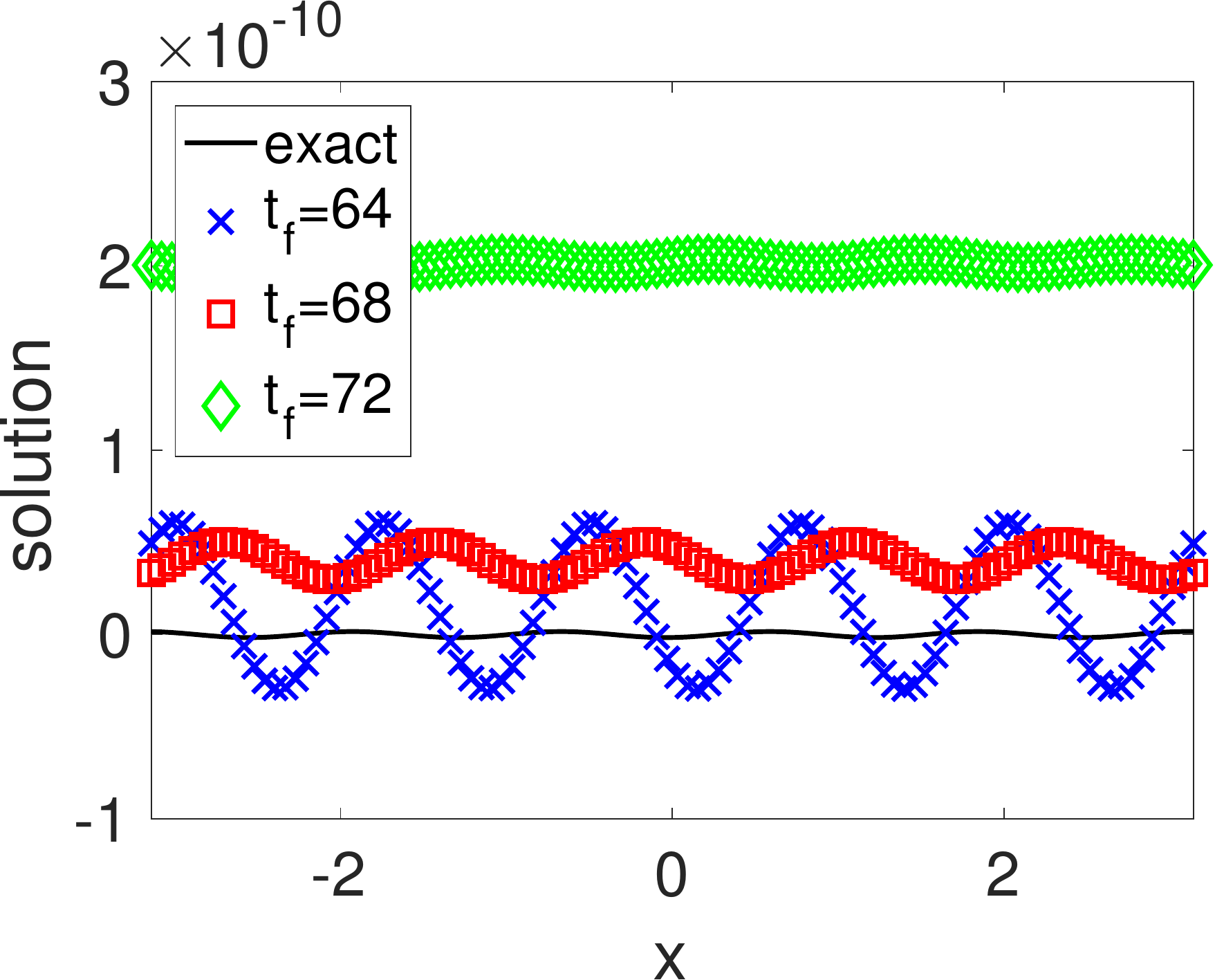} \hfill
  \includegraphics[width=.325\textwidth]{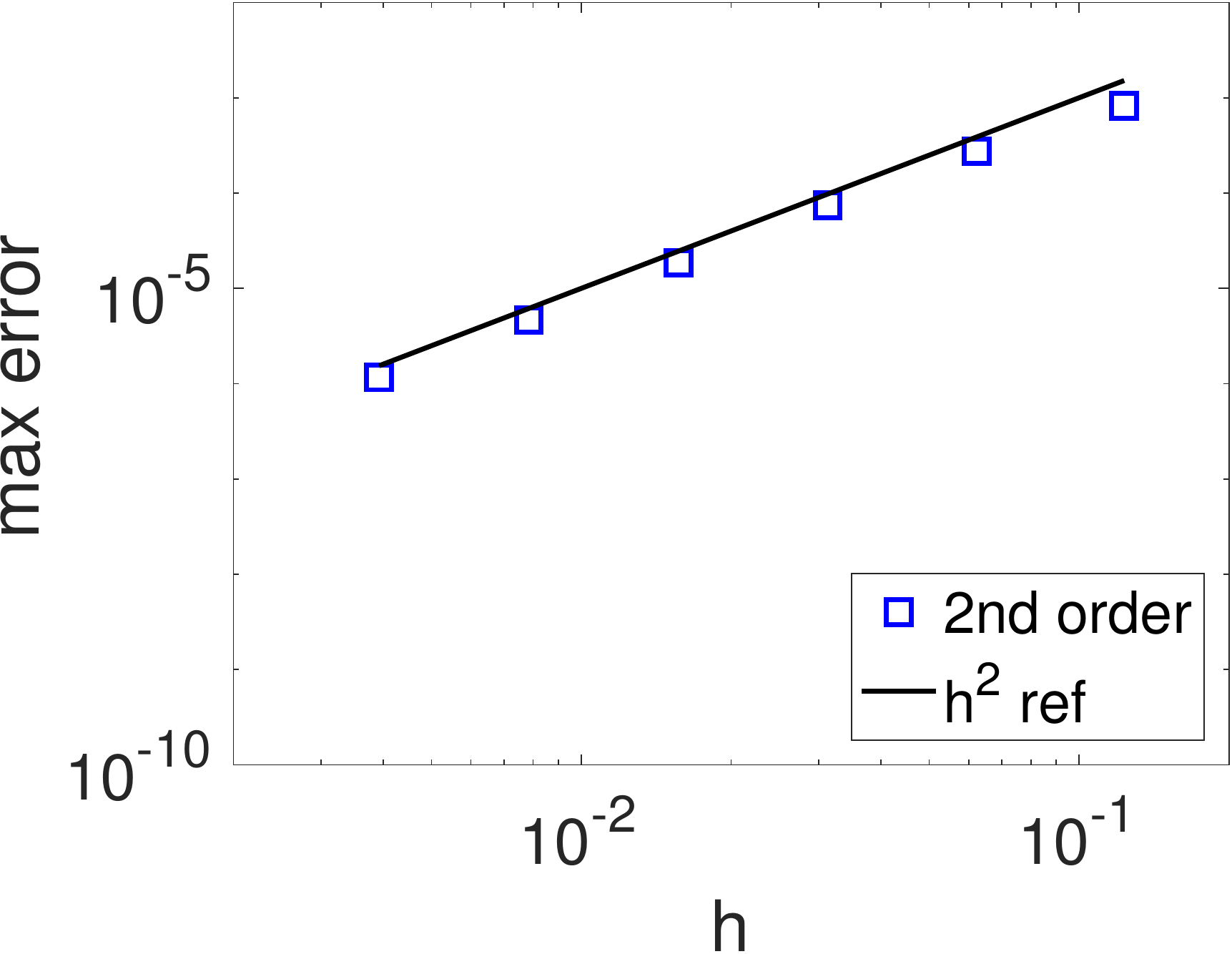}
  \caption{Numerical results using the second-order accurate scheme for a 1D periodic wave with $c=1$, $\omega_p=3$, $\epsilon_r=1$, and $\gamma=10$. At left is the time history of the maximum error for a variety of grid sizes, which simultaneously shows the presence of exponentially growing modes as well as numerical convergence. In the center, the nature of the growing solution is illustrated to be dominated by the constant mode for the computation with $N=101$ points. Finally, at right the expected second-order convergence is demonstrated at $t=20$.}
  \label{fig:RC2_conv}
  \end{center}
\end{figure}

\begin{figure}[hbt]
  \begin{center}
  \includegraphics[width=.325\textwidth]{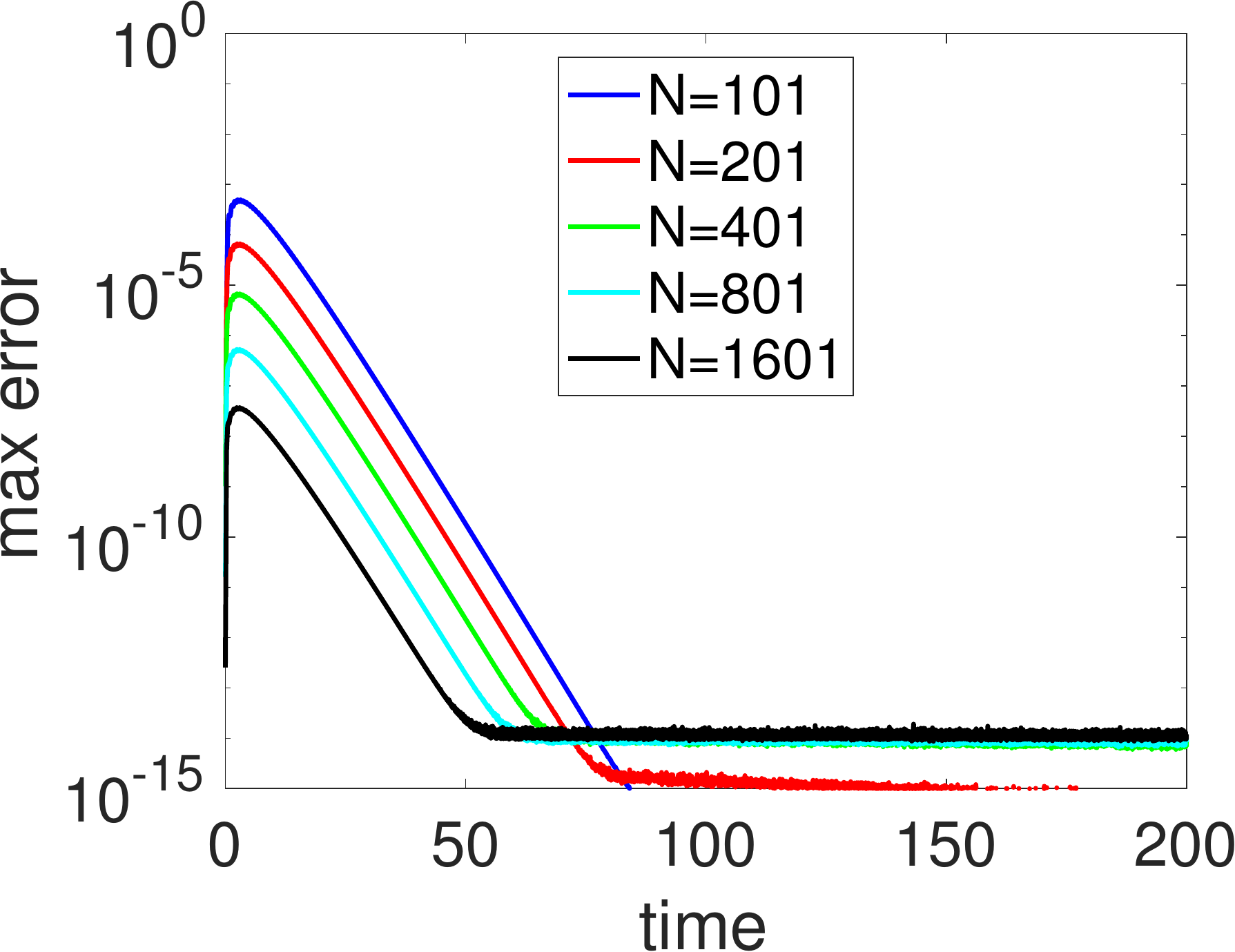} \hfill
  \includegraphics[width=.325\textwidth]{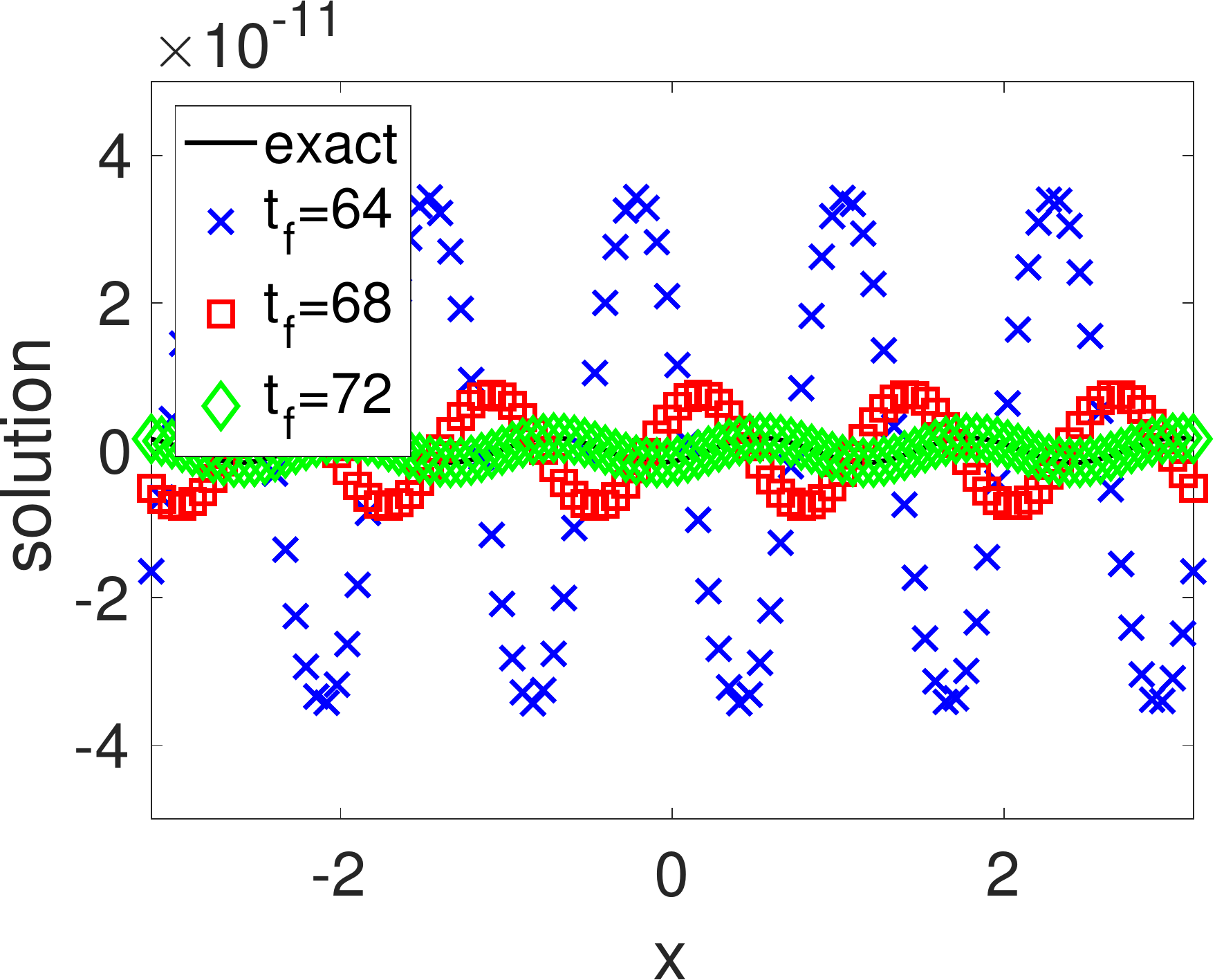} \hfill
  \includegraphics[width=.325\textwidth]{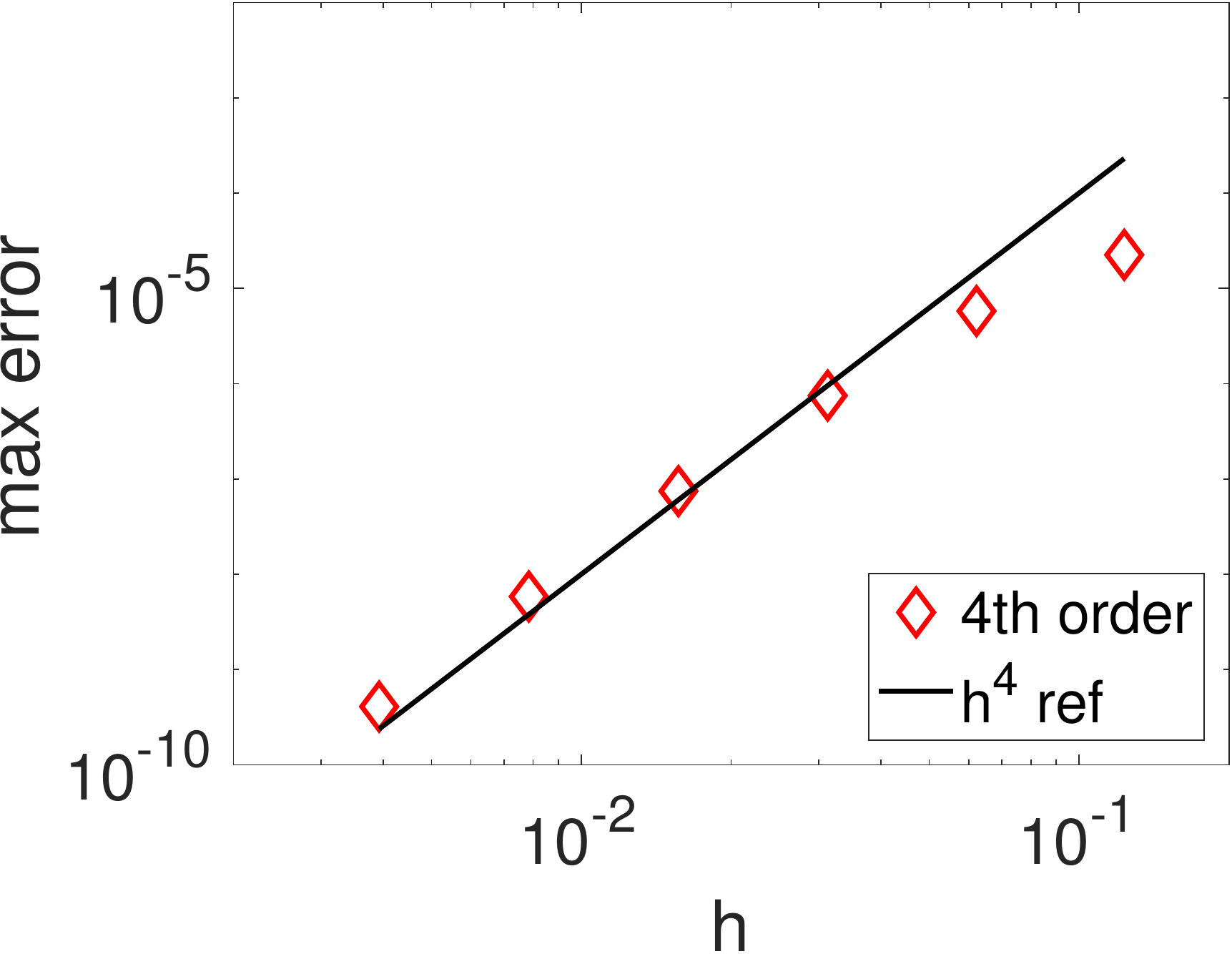} \hfill
  \caption{Numerical results using the fourth-order accurate scheme for a 1D periodic wave with $c=1$, $\omega_p=3$, $\epsilon_r=1$, and $\gamma=10$. The meaning of the panels here mirror those in Figure~\ref{fig:RC2_conv}. At left is the time history of the maximum error for a variety of grid sizes which shows no exponentially growing modes. In the center is the solution at the same times as in the center of Figure~\ref{fig:RC2_conv}, although here there is clearly no spurious growth. Finally, at right the expected 4th order convergence is demonstrated at $t=20$.}
  \label{fig:RC4_conv}
  \end{center}
\end{figure}

\subsection{Example: Scattering on Planar Drude Material Interface in Two Dimensions}
\label{sec:planeScattering}

Consider now a scattering problem in two dimensions at a planar interface between a non-dispersive dielectric and a Drude dispersive medium, located at $x = x_{\rm mid}$. We consider the TE$_z$ mode (transverse electric mode with respect to z), which has components $(E_x, E_y, H_z)$. Note that the numerical method computes only the electric fields $(E_x,E_y)$, and the magnetic field is derived from those computed electric field components. The material in the first domain, $\Omega_1 \equiv \{ (x,y) : x \leq x_{\rm mid} \} $, has solution components $(E_{x,1}, E_{y,1}, H_{z,1})$, and dielectric material parameters $\epsilon_{r,1}, \mu_1, \omega_{p,1} = 0,$ and $ \gamma_1 = 0$. The Drude material in the second domain, $\Omega_2 \equiv \{ (x,y) : x \geq x_{\rm mid} \} $, has solution components $(E_{x,2}, E_{y,2}, H_{z,2})$, and Drude material parameters $\epsilon_{r,2}, \mu_2, \omega_{p,2},$ and $ \gamma_2$. The frequency dependent permittivities as defined in Section \ref{section:governing_equations} are given respectively in each domain by 
\begin{align}
\widehat{\epsilon_1}(\omega) = \epsilon_0 \epsilon_{r,1} , \qquad \qquad \widehat{\epsilon_2}(\omega) = \epsilon_0 \left( \epsilon_{r,1} - \frac{\omega_{p,2}^2}{\omega (i \gamma_2 - \omega) } \right). 
\end{align}
Each of the two components of the electric field are governed by the second-order form of the dispersive Maxwell's equations \eqref{eqn:waveE2}, with Table \ref{table:etas} giving the corresponding values of the integral kernel $\eta(\tau)$. In particular, for $x \in \Omega_1$ the fields obey the equations 
\begin{subequations}
\begin{align}
& \partial_t^2 E_{x,1}  =  \frac{1}{\mu_1 \epsilon_0 \epsilon_{r,1} }  \Delta E_{x,1} , \\
& \partial_t^2 E_{y,1} =  \frac{1}{\mu_1 \epsilon_0 \epsilon_{r,1} }  \Delta E_{y,1} , 
\end{align}
\end{subequations}
while for $x \in \Omega_2$ the governing equations describing the fields are
\begin{subequations}
\begin{align}
& \partial_t^2 E_{x,2} =  \frac{1}{\mu_2 \epsilon_0 \epsilon_{r,2} }  \Delta E_{x,2} - \frac{  \omega_{p,2} ^2 }{\epsilon_{r,2} } E_{x,2} + \frac{ \omega_{p,2}^2 \gamma_2 }{\epsilon_{r,2} } \int_0^{\infty} e^{- \gamma_2 \tau} E_{x,2}(t - \tau) d \tau, \\
& \partial_t^2 E_{y,2} =  \frac{1}{\mu_2 \epsilon_0 \epsilon_{r,2} }  \Delta E_{y,2} - \frac{  \omega_{p,2} ^2 }{\epsilon_{r,2} } E_{y,2} + \frac{ \omega_{p,2} ^2 \gamma_2 }{\epsilon_{r,2} } \int_0^{\infty} e^{- \gamma_2 \tau} E_{y,2}(t - \tau) d \tau.
\end{align}
\end{subequations}
The single component of the magnetic field in each domain is can be described from the electric fields using the Maxwell-Faraday equation \eqref{eqn:macro1} as
\begin{subequations}
\begin{align}
& \partial_t H_{z,1} = \frac{1}{\mu_1} \left( \partial_y E_{x,1} - \partial_x E_{y,1} \right), \\
& \partial_t H_{z,2} = \frac{1}{\mu_2} \left( \partial_y E_{x,2} - \partial_x E_{y,2} \right). 
\end{align}
\end{subequations}
The two domains are coupled by imposition of the physical jump conditions \eqref{eqns:int} along the material interface:
\begin{subequations}
\begin{align}
\left[ \mathbf{n} \times \mathbf{E} \right]_{\mathcal{I}} = 0 \qquad & \Rightarrow \qquad E_{y,1} = E_{y,2}, \\
\left[ \mathbf{n} \cdot \widehat{\mathbf{D}} \right]_{\mathcal{I}} = 0 \qquad & \Rightarrow \qquad \widehat{D_{x,1}}(\omega) = \widehat{\epsilon_1}(\omega) \widehat{E_{x,1}}(\omega) \\
& \hspace{2.44cm} =  \widehat{D_{x,2}}(\omega)  = \widehat{\epsilon_2}(\omega) \widehat{E}_{x,2} (\omega) , \nonumber   \\
\partial_t \left[ \mathbf{n} \times \mathbf{H} \right]_{\mathcal{I}} = 0  \qquad & \Rightarrow \qquad  \partial_t H_{z,1} = \frac{1}{\mu_1} \left( \partial_y E_{x,1} - \partial_x E_{y,1} \right) =  \partial_t H_{z,2}  \\
& \hspace{2.25cm} = \frac{1}{\mu_2} \left( \partial_y E_{x,2} - \partial_x E_{y,2} \right), \nonumber
\end{align}
\end{subequations}
Here, we have applied the Fourier transform to simplify the second condition and have taken a time derivative of the third.

An exact solution consisting of incident, scattered, and transmitted fields is derived following the well-known procedure used in the non-dispersive case of an interface between two dielectrics \cite{jackson_1962}. Let $\theta_i$ be the angle of incidence of the incident plane wave. The solution in the dielectric domain $\Omega_1$ has an incident and reflected component and therefore satisfies
\begin{subequations}
\begin{align}
& E_{x,1} = A_{x,1} \left( e^{i \mathbf{k_i(\omega)} \cdot \mathbf{x} - i \omega t} - R e^{i \mathbf{k_r(\omega)} \cdot \mathbf{x} - i \omega t} \right), \\
& E_{y,1} = A_{y,1} \left( e^{i \mathbf{k_i(\omega)} \cdot \mathbf{x} - i \omega t} + R e^{i \mathbf{k_r(\omega)} \cdot \mathbf{x} - i \omega t} \right),
\end{align}
\end{subequations}
while the Drude domain, $\Omega_2$, is assumed to have only transmitted component and therefore satisfies
\begin{subequations}
\begin{align}
& E_{x,2} = A_{x,2} T  e^{i \mathbf{k_t} (\omega) \cdot \mathbf{x} - i \omega t}, \\
& E_{y,2} = A_{y,2} T e^{i \mathbf{k_t} (\omega) \cdot \mathbf{x} - i \omega t}.
\end{align}
\end{subequations}
In the present case, the matching conditions at the interface then yield
\begin{subequations}
\begin{align}
\mathbf{k_i(\omega)} & = \omega \sqrt{ \widehat{\epsilon_1}(\omega) \mu_1} \left( \cos \theta_i, \sin \theta_i \right)^T = ( k_{i,x}, k_{i,y} )^T, \\
\mathbf{k_r(\omega)} & =  \omega \sqrt{ \widehat{\epsilon_1}(\omega) \mu_1} \left( - \cos \theta_i, \sin \theta_i \right)^T =  ( k_{r, x}, k_{r,y} )^T, \\
\mathbf{k_t(\omega)} & = \omega \sqrt{ \widehat{\epsilon_2}(\omega) \mu_2} \left(\cos \theta_t, \sin \theta_t \right)^T = (  k_{t,x}, k_{t,y} )^T ,  \\
\theta_t & = \sin^{-1} \left( \frac{ \sqrt{ \widehat{ \epsilon_1} (\omega) \mu_1} } { \sqrt{ \widehat{ \epsilon_2} (\omega) \mu_2} } \sin \theta_i \right), \\
R & = e^{2 i k_{i,x}  x_{\text{mid} }} \left( \frac{ \sqrt{ \widehat{\epsilon_1}(\omega) \mu_1 } \cos( \theta_t )  - \sqrt{ \widehat{\epsilon_2}(\omega) \mu_2 } \cos( \theta_i )  }
	{ \sqrt{ \widehat{\epsilon_1}(\omega) \mu_1 } \cos( \theta_t )  + \sqrt{ \widehat{\epsilon_2}(\omega) \mu_2 } \cos( \theta_i )} \right), \\
T & = e^{ i \left( k_{i,x} - k_{t,x} \right) x_{\text{mid}} } \left( \frac{ 2  \sqrt{ \widehat{\epsilon_1 }(\omega) \mu_1 } \cos( \theta_i )  }
	{ \sqrt{ \widehat{\epsilon_1}(\omega) \mu_1 } \cos( \theta_t )  + \sqrt{ \widehat{\epsilon_2}(\omega) \mu_2 } \cos( \theta_i )} \right) , \\
A_{x,1} & = A \sin \theta_i , \\
A_{y,1} & = - A \cos \theta_i , \\
A_{x,2} & = A \sin \theta_t , \\
A_{y,2} & = - A \cos \theta_t . 
\end{align}
\end{subequations}
where $A$ is an arbitrary constant. Note that $\theta_t$ is complex-valued, and one has
\begin{subequations}
\begin{align}
\sin \theta_t & = \frac{ \sqrt{ \widehat{ \epsilon_1} (\omega) \mu_1} } { \sqrt{ \widehat{ \epsilon_2} (\omega) \mu_2} } \sin \theta_i , \\
\cos \theta_t & = \left( 1 - \frac{\widehat{ \epsilon_1} (\omega) \mu_1}{\widehat{ \epsilon_2} (\omega) \mu_2}  \sin^2 \theta_i \right)^{1/2}. 
\end{align}
\end{subequations}

In the left domain $\Omega_1$, we consider vacuum with $\omega_{p,1} = \gamma_1 = 0$ and $\epsilon_{r,1} = \mu_1 = 1$. In the right domain, $\Omega_2$, we consider silver with $\epsilon_{r,2} = 5$, $\mu_2 = 1$, $\omega_{p,2} = 8.9$ eV, and $\gamma_2 = (17)^{-1}$ THz \cite{yang_2015}. The plane wave is incident on the interface at $x_{\text{mid}} = 0$ at an angle of $\theta_i = \pi/5$. The scattered fields and their errors for the second- and fourth-order accurate schemes using two incident frequencies $\omega_1 = 1000$ THz and $\omega_2 = 1300$ THz, are shown respectively in Figures \ref{fig:scattered_fields_1} and \ref{fig:scattered_fields_2}. We remark that these frequencies are higher than the optical range to which the Drude parameters given in \cite{yang_2015} are fit, and are used to produce reasonable field figures which illustrate the properties of the derived Drude model solution. As predicted by the classical Drude theory \cite{maier_2007}, the transmitted wave in silver is strongly evanescent at lower frequencies, while for larger frequencies closer to the plasma frequency $\omega_{p,2}$ of the metal, the wave is transmitted over a longer distance. Note that for both the second- and fourth-order accurate schemes, the numerical errors are smooth in each domain, and show no signs of spurious singular behavior near the material interface. This favorable property is a direct result of the accurate numerical treatment of the interface conditions using compatibility conditions as outline in Section~\ref{sec:interface}. In addition, Figure \ref{fig:rates} presents results of convergence studies for each scheme using the spatial $L_1$, $L_2$, and $L_{\infty}$ norms defined by 
\begin{subequations}
\begin{align}
& \| \mathbf{u}_{\jv}^n  \|_{L_1} = \frac{1}{N_x N_y} \sum_{j_x = 0}^{N_x} \sum_{j_y = 0}^{N_y} | \mathbf{u}_{\jv}^n | ,  \\
& \|  \mathbf{u}_{\jv}^n  \|_{L_2} = \left( \frac{1}{N_x N_y} \sum_{j_x = 0}^{N_x} \sum_{j_y = 0}^{N_y} | \mathbf{u}_{\jv}^n |^2 \right)^{1/2}  , \\
& \| \mathbf{u}_{\jv}^n  \|_{L_{\infty}}  = \underset{ \jv \in [0, N_x] \times [0, N_y] }{\max}  \ | \mathbf{u}_{\jv}^n | .
\end{align}
\label{eqn:normdefs}
\end{subequations}
Here, $N_x$ is the number of grid points in the $x$ direction, $N_y$ is the number of grid points in the $y$ direction, $\jv = (j_x, j_y) \in [0, N_x] \times [0, N_y]$. For the convergence study on the domain $\xv = (x,y) \in [0,x_{\text{max}}] \times [0,y_{\text{max}}]$, where $0 < x_{\text{mid}} < x_{\text{max}}$, the grid spacings $h_1$ and $h_2$ in the $x$ and $y$ direction are refined such that $N_x = x_{\text{max}} / h_1$, $N_y = y_{\text{max}} / h_2$. The final time $T = n \Delta t$ remains fixed, with $\Delta t$ chosen according to the stability bounds \eqref{eq:RC2_time_step} and \eqref{eq:RC4_time_step}-\eqref{eq:RC4_GammaBound}. The results in Figure \ref{fig:rates} illustrate the expected order of convergence for all cases.

\begin{figure}[h] 
\includegraphics[width=.325\textwidth]{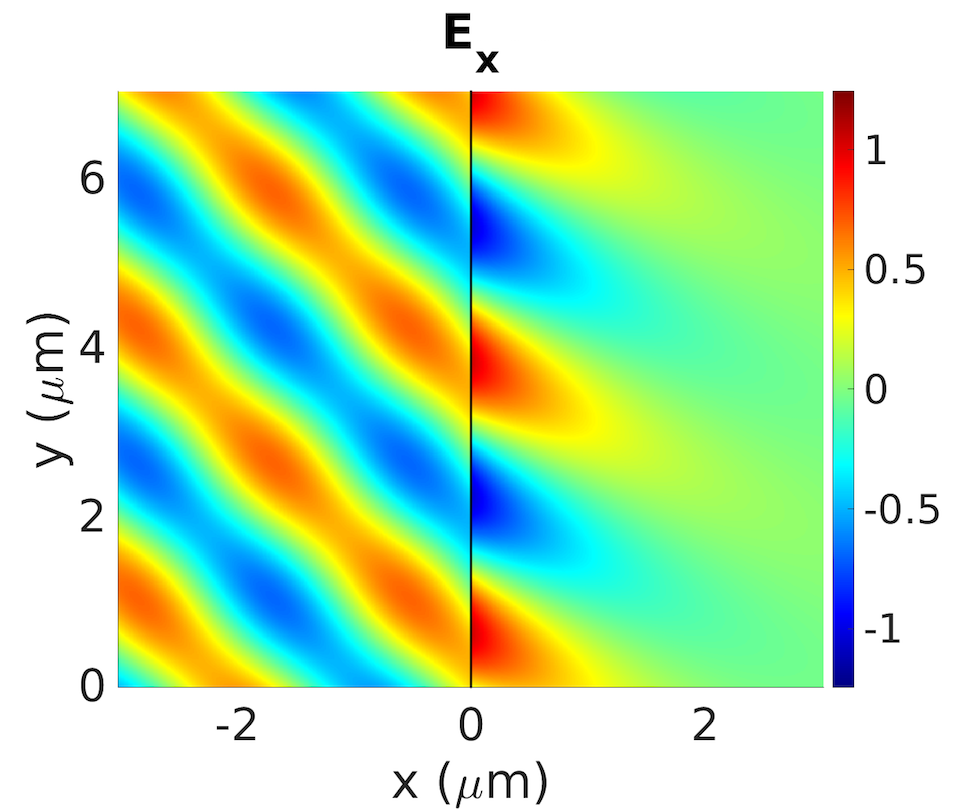} \hfill
\includegraphics[width=.325\textwidth]{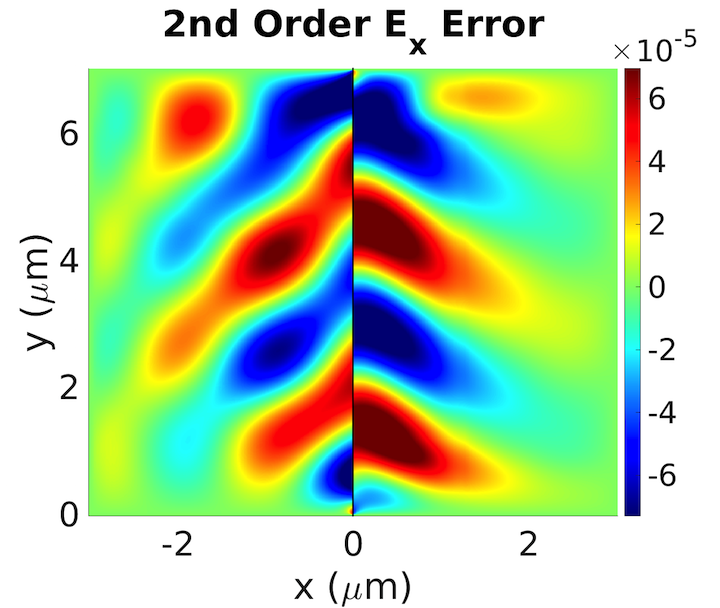} \hfill
\includegraphics[width=.325\textwidth]{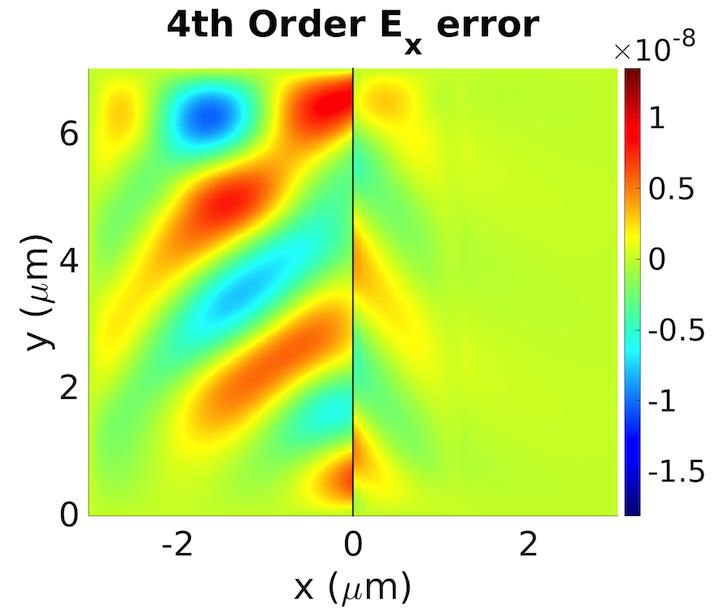}
\\
\includegraphics[width=.325\textwidth]{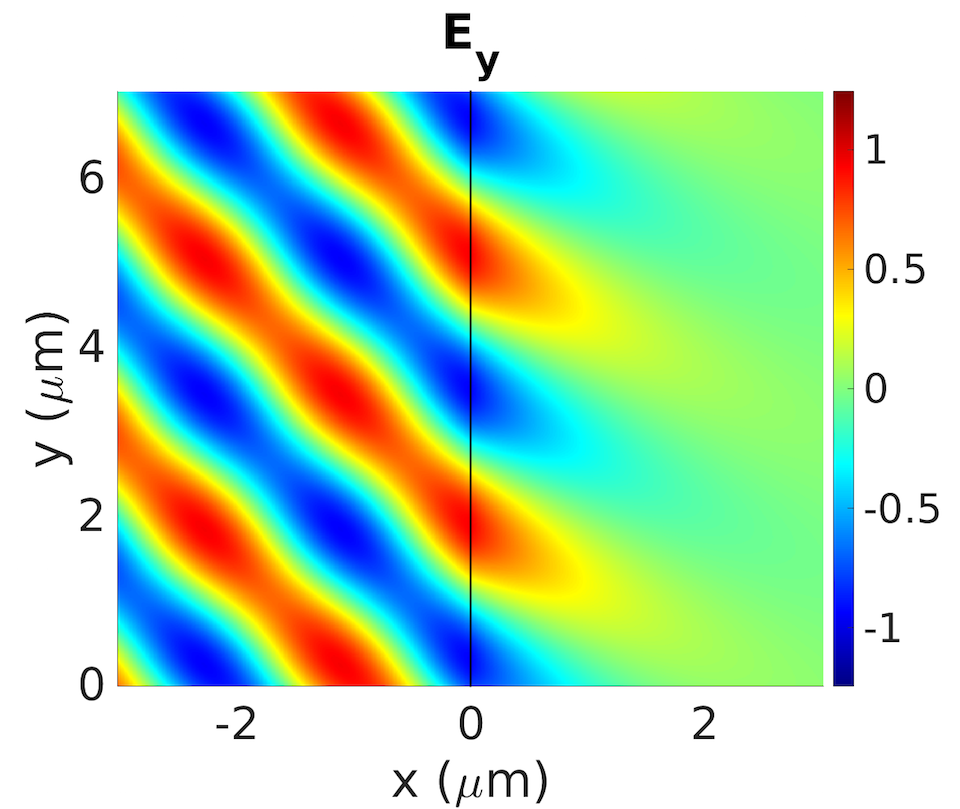} \hfill
\includegraphics[width=.325\textwidth]{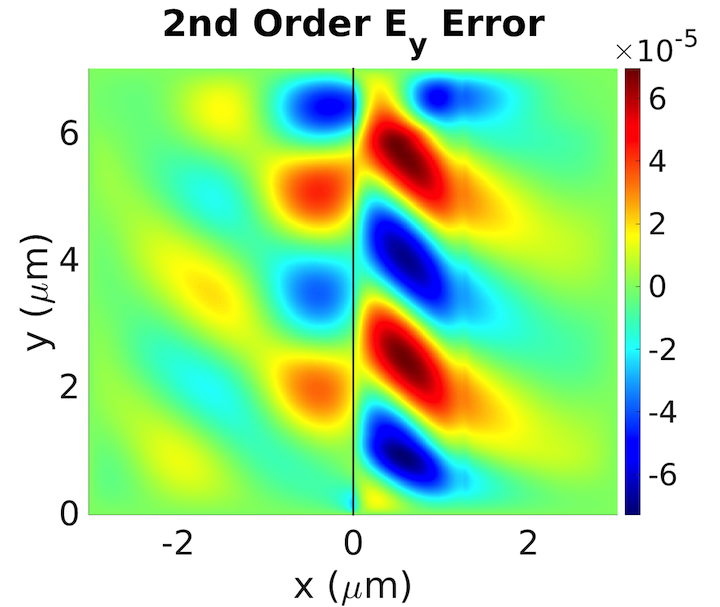} \hfill
\includegraphics[width=.325\textwidth]{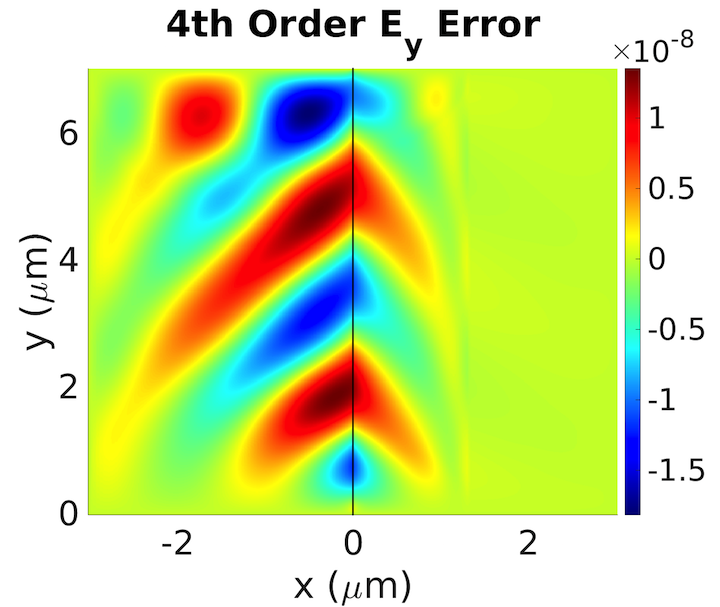}
\caption{Electric fields and their errors for the two-dimensional scattering problem at interface between vacuum medium and Drude silver medium \cite{yang_2015}, computed with the second-order and fourth-order schemes. Here, the wave is incident from the left with frequency $\omega_1 = 1000$ THz at angle $\theta_i = \pi/5$, the displayed time is $t = 1 \times 10^{-14}$ s, and the bounds \eqref{eq:RC2_time_step} and \eqref{eq:RC4_time_step}-\eqref{eq:RC4_GammaBound} were used to choose an appropriate time step which guarantees stability. From left to right, the top row of plots correspond to the fields $E_x$ for the second-order scheme, the error in $E_x$ for the second-order scheme, and the error in $E_x$ for the fourth-order scheme. The bottom row of plots is similar but for the $E_y$ component of the field.}
\label{fig:scattered_fields_1}
\end{figure}

\begin{figure}[h] 
\includegraphics[width=.325\textwidth]{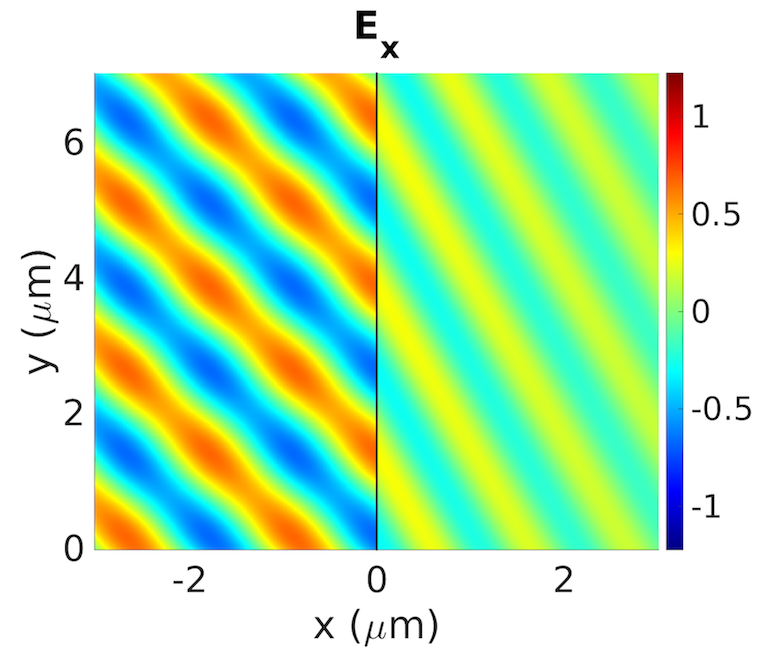} \hfill
\includegraphics[width=.325\textwidth]{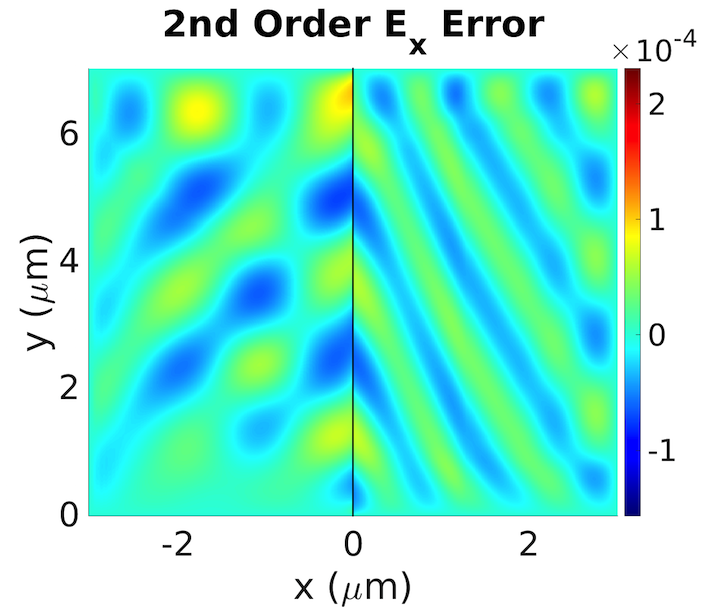} \hfill
\includegraphics[width=.325\textwidth]{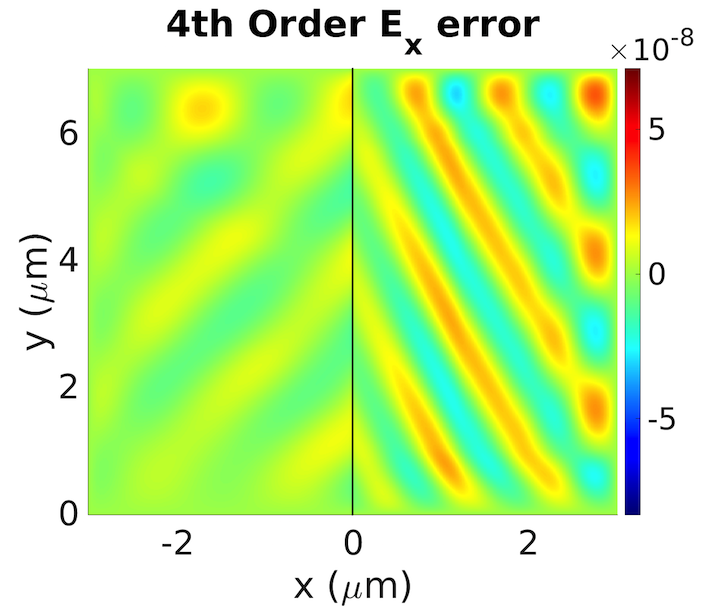}
\\
\includegraphics[width=.325\textwidth]{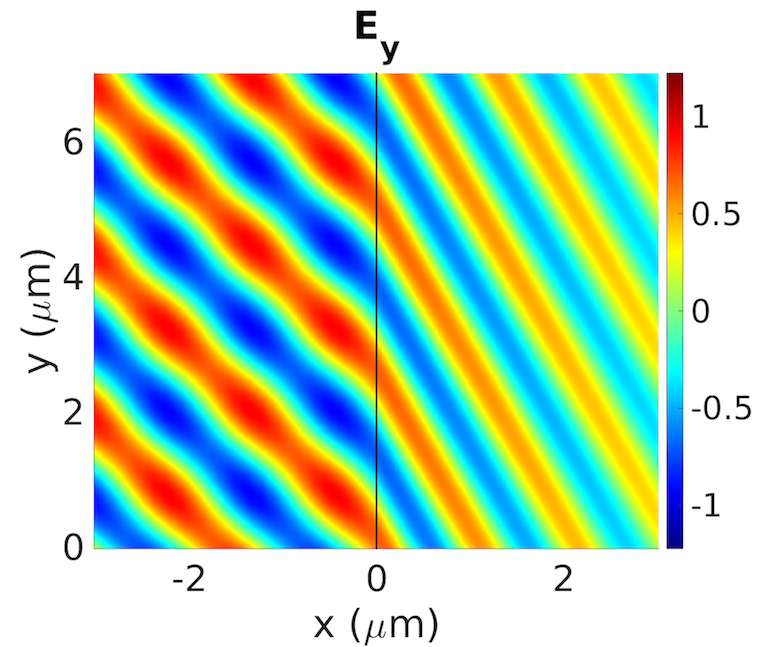} \hfill
\includegraphics[width=.325\textwidth]{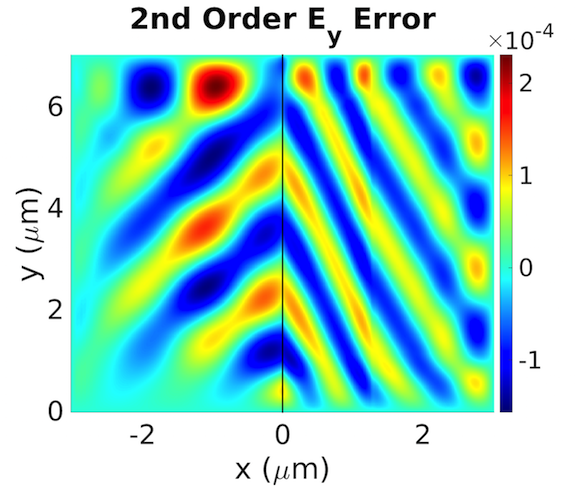} \hfill
\includegraphics[width=.325\textwidth]{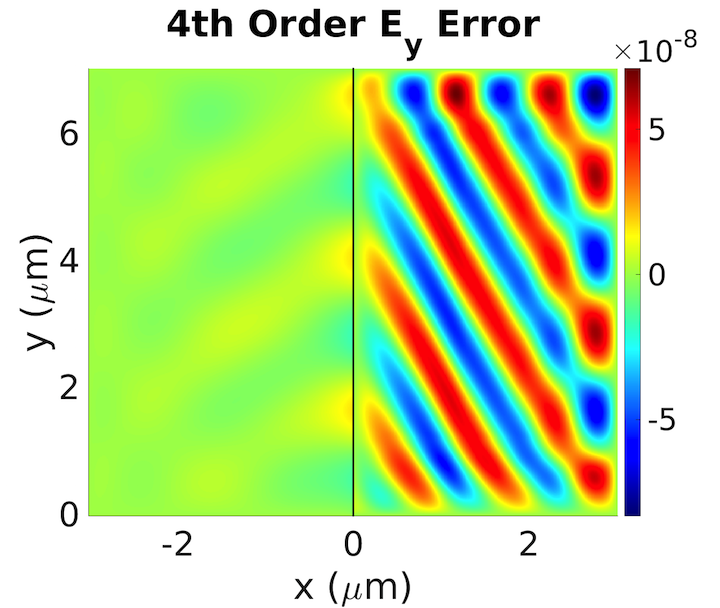}
\caption{Electric fields and their errors for the two-dimensional scattering problem at interface between vacuum medium and Drude silver medium \cite{yang_2015}, computed with the second-order and fourth-order schemes. Here, the wave is incident from the left with frequency $\omega_2 = 1300$ THz at angle $\theta_i = \pi/5$, the displayed time is $t = 1 \times 10^{-14}$ s, and the bounds \eqref{eq:RC2_time_step} and \eqref{eq:RC4_time_step}-\eqref{eq:RC4_GammaBound} were used to choose an appropriate time step which guarantees stability. From left to right, the top row of plots correspond to the fields $E_x$ for the second-order scheme, the error in $E_x$ for the second-order scheme, and the error in $E_x$ for the fourth-order scheme. The bottom row of plots is similar but for the $E_y$ component of the field.}
\label{fig:scattered_fields_2}
\end{figure}

\begin{figure}[h] 
\includegraphics[width=.325\textwidth]{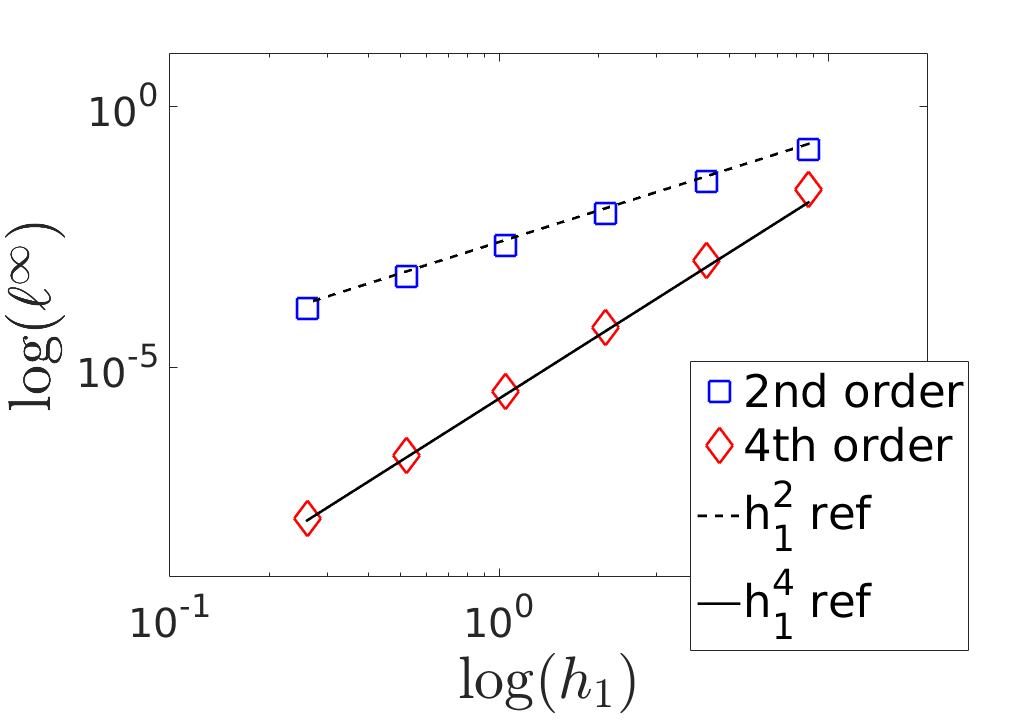} \hfill
\includegraphics[width=.325\textwidth]{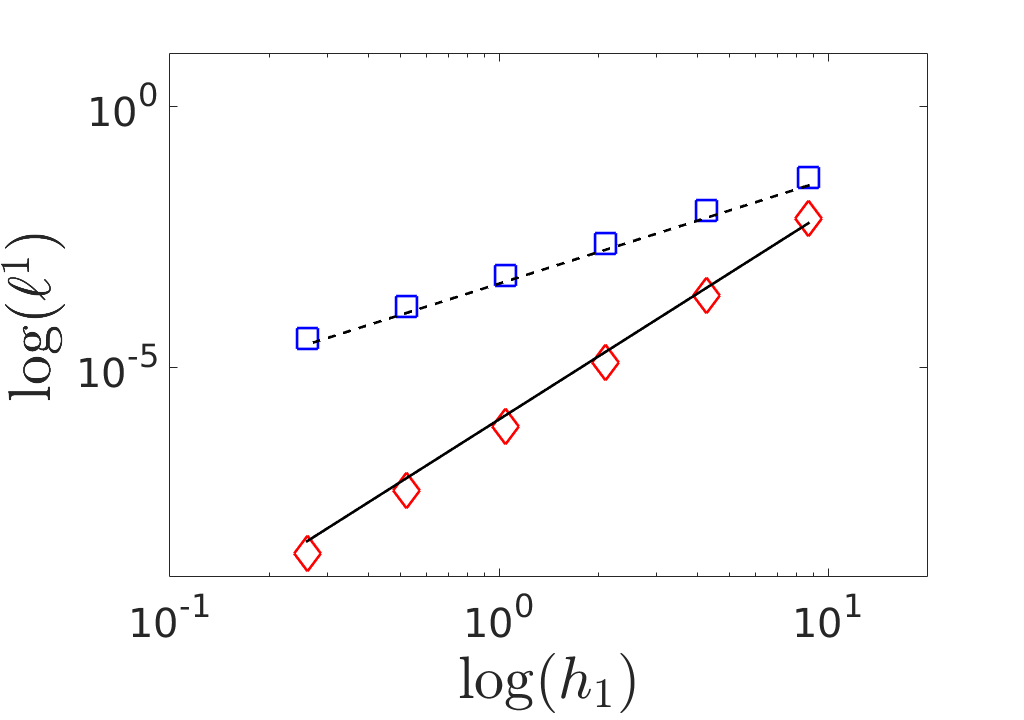} \hfill
\includegraphics[width=.325\textwidth]{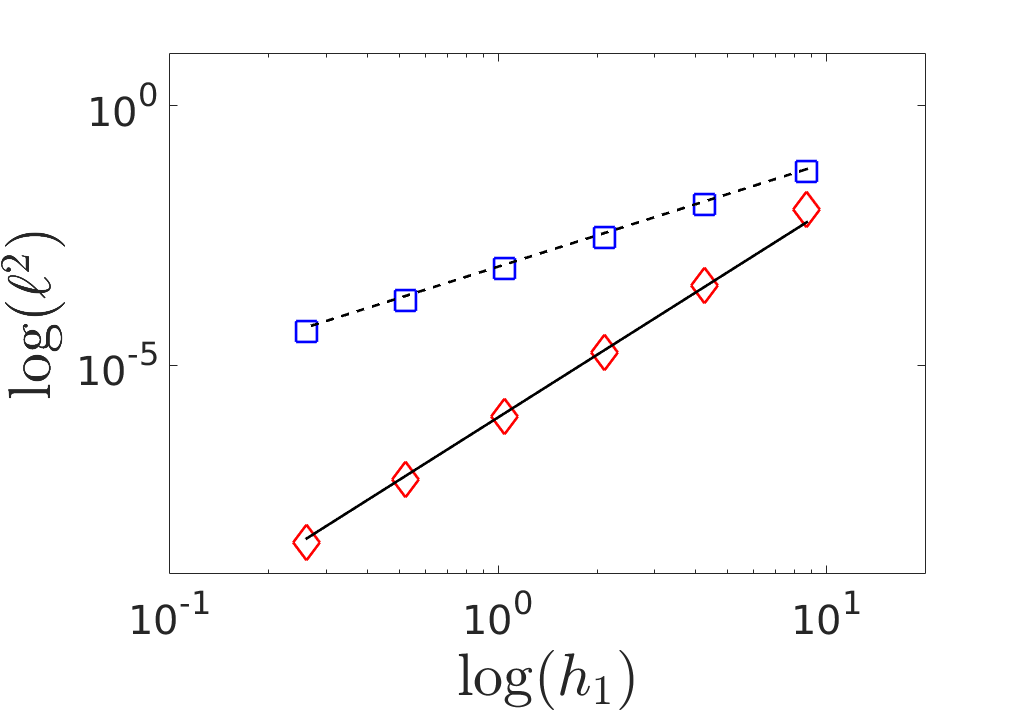} \\
\includegraphics[width=.325\textwidth]{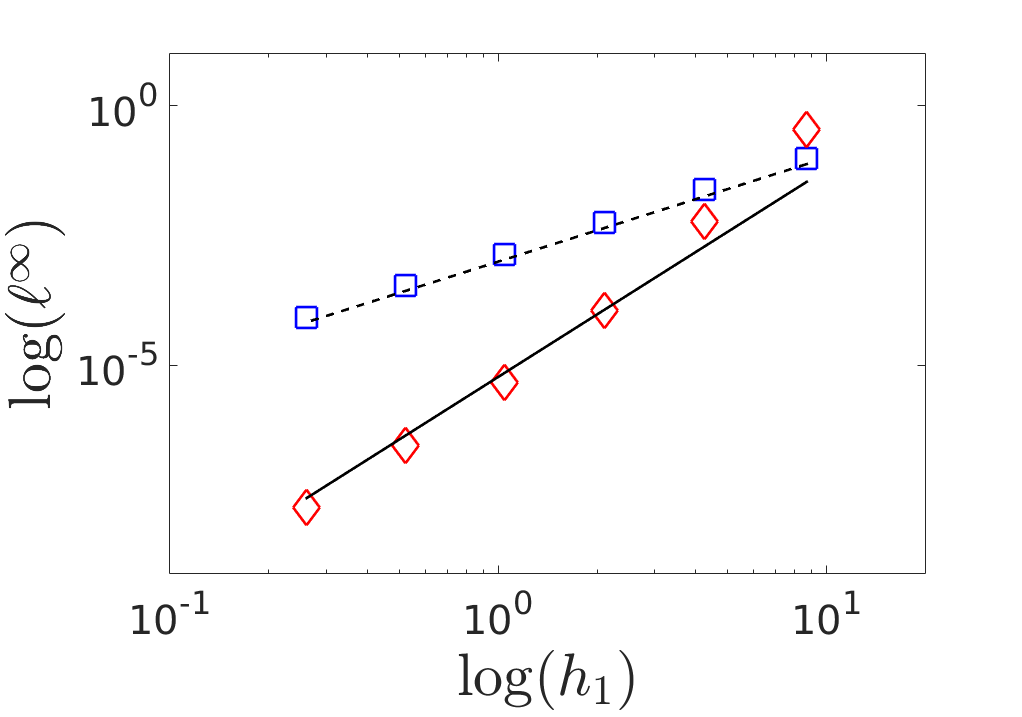} \hfill
\includegraphics[width=.325\textwidth]{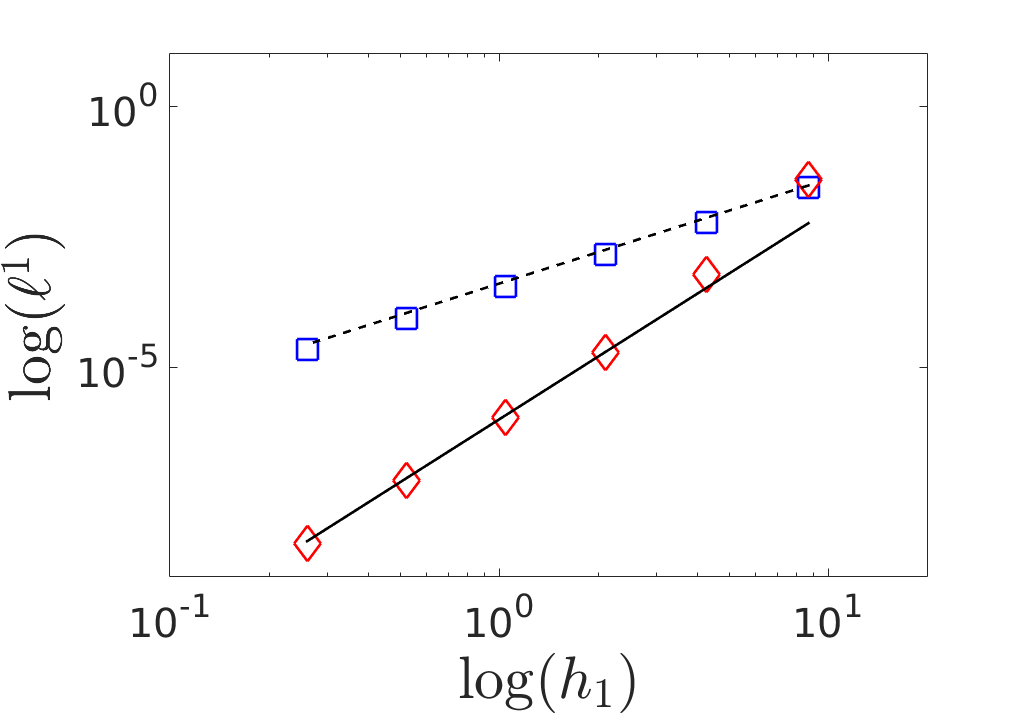} \hfill
\includegraphics[width=.325\textwidth]{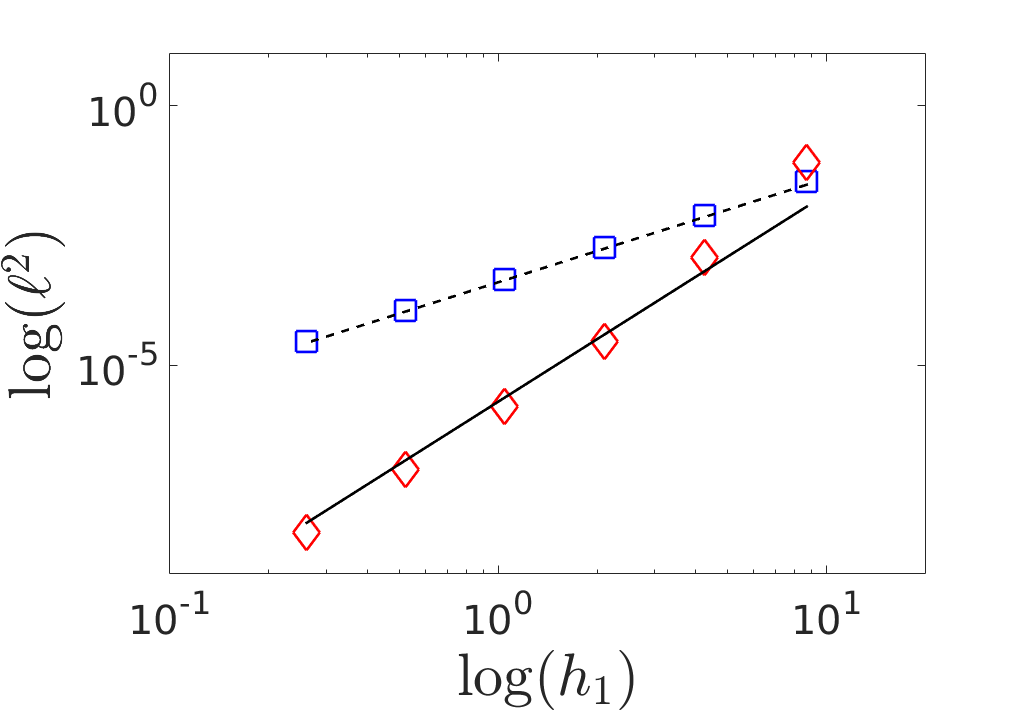}
\caption{Convergence rates for the two-dimensional scattering problem at interface between vacuum and Drude silver medium with $\omega = 1000$ THz (see Figures \ref{fig:scattered_fields_1} and \ref{fig:scattered_fields_2}). The bounds \eqref{eq:RC2_time_step} and \eqref{eq:RC4_time_step}-\eqref{eq:RC4_GammaBound} were used to choose an appropriate time step which guarantees stability. From left to right, the top row shows $L_{\infty}$, $L_1$, and $L_2$ norms for the $E_x$ field. The corresponding $E_y$ rates are shown in the bottom row of plots.}
\label{fig:rates}
\end{figure}

\subsection{Example: Surface Plasmon Polariton at Drude Material Interface in Two Dimensions}
\label{sec:polariton}

As in Section \ref{sec:planeScattering}, we again consider an interface in two spatial dimensions between a non-dispersive dielectric and a Drude dispersive medium located at the plane $x = x_{\rm mid}$, and consider the TE$_z$ mode (transverse electric mode with respect to z), with components $(E_x, E_y, H_z)$. The material in the first domain, $\Omega_1 \equiv \{ (x,y) : x \leq x_{\rm mid} \} $, has dielectric material parameters $\epsilon_{r,1}, \mu_1, \omega_{p,1} = 0,$ and $ \gamma_1 = 0$, while the Drude material in the second domain, $\Omega_2 \equiv \{ (x,y) : x \geq x_{\rm mid} \} $, has Drude material parameters $\epsilon_{r,2}, \mu_2, \omega_{p,2},$ and $ \gamma_2$. The frequency dependent permittivities as defined in Section \ref{section:governing_equations} are given respectively in each domain by 
\begin{align}
\widehat{\epsilon_1}(\omega) = \epsilon_0 \epsilon_{r,1} , \qquad \qquad \widehat{\epsilon_2}(\omega) = \epsilon_0 \left( \epsilon_{r,1} - \frac{\omega_{p,2}^2}{\omega (i \gamma_2 - \omega) } \right). 
\end{align}
Following the derivation in \cite{maier_2007}, we consider the surface plasmon polariton, a mode which is localized to the interface at $x = x_{\rm mid}$ in the $x$-direction ({\it i.e.}, evanescent away from), and propagates in the $y$-direction. Additional details concerning the derivation of the surface plasmon polariton are given in Appendix \ref{app:polariton}. For $x \in \Omega_1$, the electric field components satisfy
\begin{subequations}
\begin{align}
& E_{x,1} = A \ e^{\alpha_1 (x - x_{\text{mid}})  } e^{i \beta y} e^{i \omega t} , \\
& E_{y,1} =  \frac{i \alpha_1}{\beta}  \ A \ e^{\alpha_1 (x - x_{\text{mid}})  } e^{i \beta y} e^{i \omega t} , 
\end{align}
\end{subequations}
while for $x \in \Omega_2$, we have
\begin{subequations}
\begin{align}
& E_{x,2} = \frac{\widehat{\epsilon_1}(\omega)}{\widehat{\epsilon_2}(\omega)} \ A  \  e^{\alpha_2 (x - x_{\text{mid}}) } \  e^{i \beta y} e^{i \omega t} , \\
& E_{y,2} = \frac{i \alpha_1}{\beta} \ A  \ e^{\alpha_2 (x - x_{\text{mid}})  } e^{i \beta y} e^{i \omega t} . 
\end{align}
\end{subequations}
Here, $A$ is an arbitrary constant and the dispersion relations in each direction are given by
\begin{subequations}
\begin{align}
& \beta =  \omega \left( \frac{ \widehat{\epsilon_1}(\omega) \widehat{\epsilon_2}(\omega) \left[ \mu_1 \widehat{\epsilon_2}(\omega) - \mu_2 \widehat{\epsilon_1}(\omega) \right] }{ \widehat{\epsilon_2}(\omega)^2 - \widehat{\epsilon_1}(\omega)^2 } \right)^{1/2} , \\
& \alpha_1 = \left(  \beta^2 - \omega^2 \mu_1 \widehat{\epsilon_1}(\omega) \right)^{1/2} =  \omega \widehat{\epsilon_1}(\omega) \left(  - \frac{ \left[ \mu_2 \widehat{\epsilon_2}(\omega) - \mu_1 \widehat{\epsilon_1}(\omega) \right] }{ \widehat{\epsilon_2}(\omega)^2 - \widehat{\epsilon_1}(\omega)^2 } \right)^{1/2} , \\
& \alpha_2 =  \left(  \beta^2 - \omega^2 \mu_2 \widehat{\epsilon_2}(\omega) \right)^{1/2} = \omega \widehat{\epsilon_2}(\omega) \left(  - \frac{ \left[ \mu_2 \widehat{\epsilon_2}(\omega) - \mu_1 \widehat{\epsilon_1}(\omega) \right] }{ \widehat{\epsilon_2}(\omega)^2 - \widehat{\epsilon_1}(\omega)^2 } \right)^{1/2}  = \frac{\alpha_1 \widehat{\epsilon_2}(\omega)}{\widehat{\epsilon_1}(\omega)} . 
\end{align}
\end{subequations}
Note that each of these expressions is typically complex-valued due to the complex permittivity of the Drude material. For certain frequencies which satisfy $\Re ( \widehat{\epsilon_2}(\omega) ) < 0$, we obtain $\alpha_1 > 0$ and $\alpha_2 < 0$, implying localization in the $x$-direction at $x = x_{\text{mid}}$. In addition, the fields will also decay in the $y$-direction of propagation according to the magnitude of the imaginary component of $\beta$. 

\begin{figure}[h] 
\includegraphics[width=.325\textwidth]{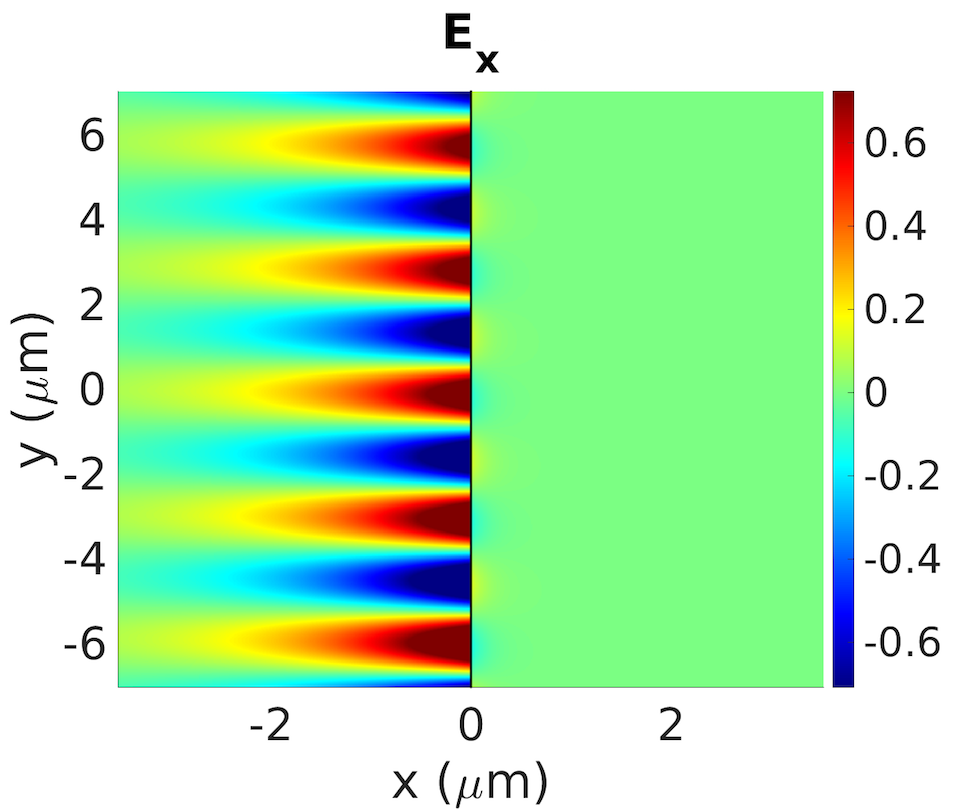} \hfill 
\includegraphics[width=.325\textwidth]{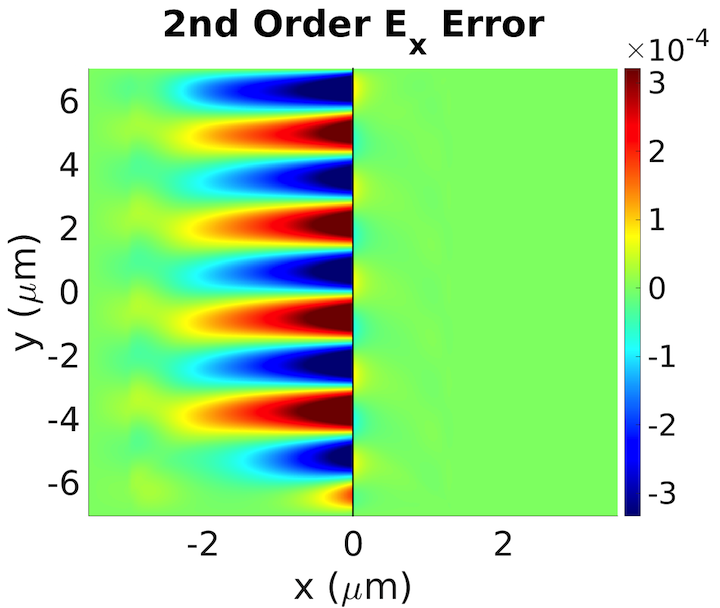} \hfill
\includegraphics[width=.325\textwidth]{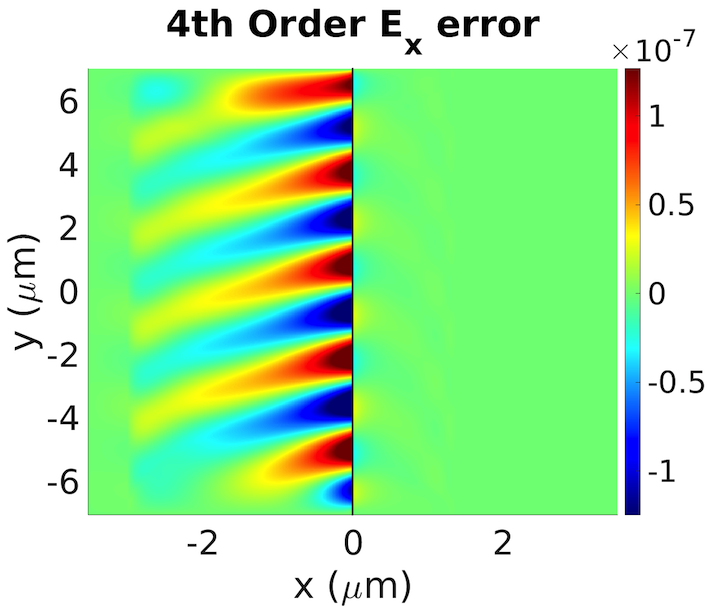}\\
\includegraphics[width=.325\textwidth]{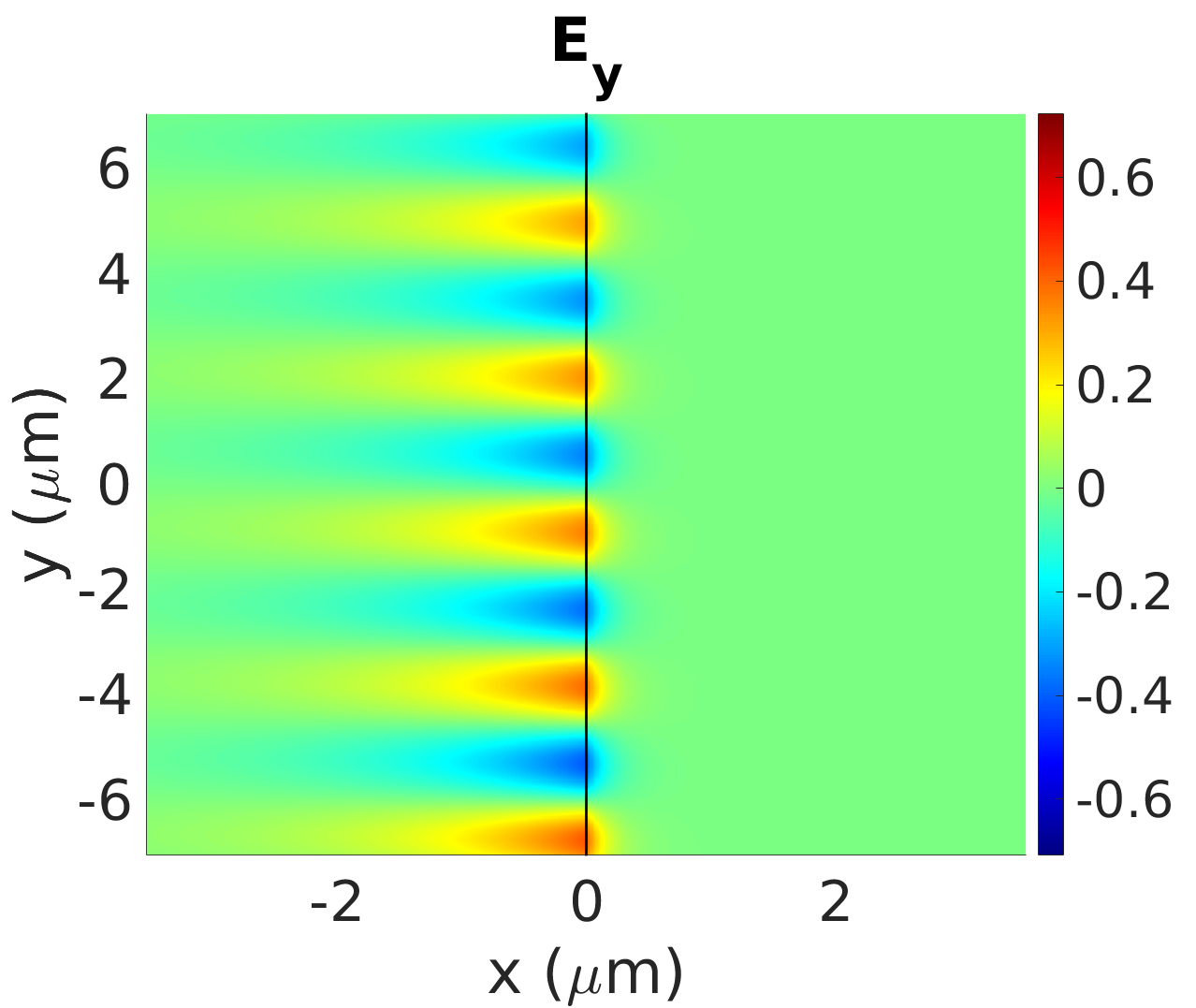} \hfill
\includegraphics[width=.325\textwidth]{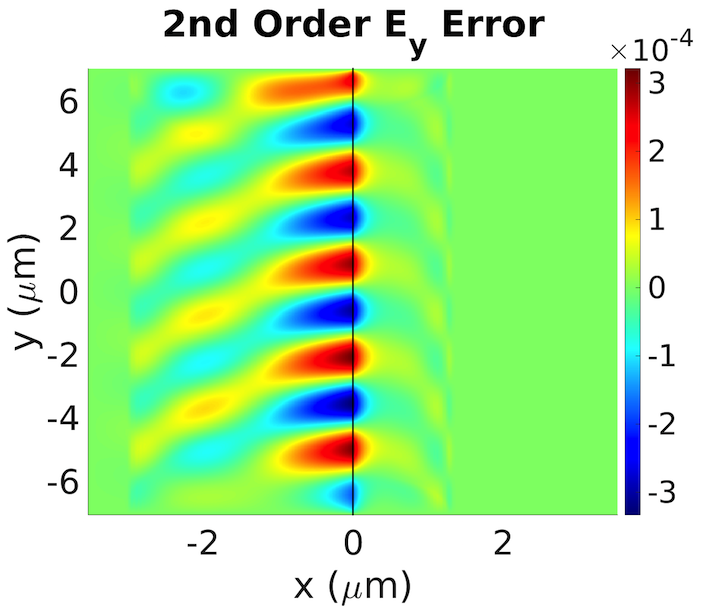} \hfill
\includegraphics[width=.325\textwidth]{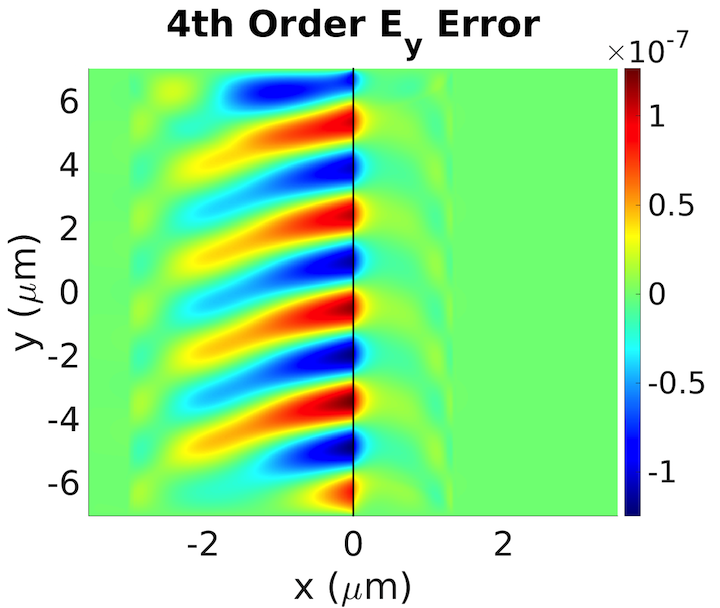}
\caption{Electric fields and their errors for the two-dimensional surface plasmon polariton at interface between vacuum medium and Drude silver medium \cite{olmon_2012}, computed with the second-order and fourth-order schemes. Here, the surface wave propagates in the positive $y$ direction at temporal frequency $\omega = 600$ THz, the displayed time is $t = 1 \times 10^{-14}$ s, and the bounds \eqref{eq:RC2_time_step} and \eqref{eq:RC4_time_step}-\eqref{eq:RC4_GammaBound} were used to choose an appropriate time step which guarantees stability. From left to right, the top row of plots correspond to the fields $E_x$ for the second-order scheme, the error in $E_x$ for the second-order scheme, and the error in $E_x$ for the fourth-order scheme. The bottom row of plots is similar but for the $E_y$ component of the field.}
\label{fig:plasmon_fields_1}
\end{figure}

\begin{figure}[h] 
\includegraphics[width=.325\textwidth]{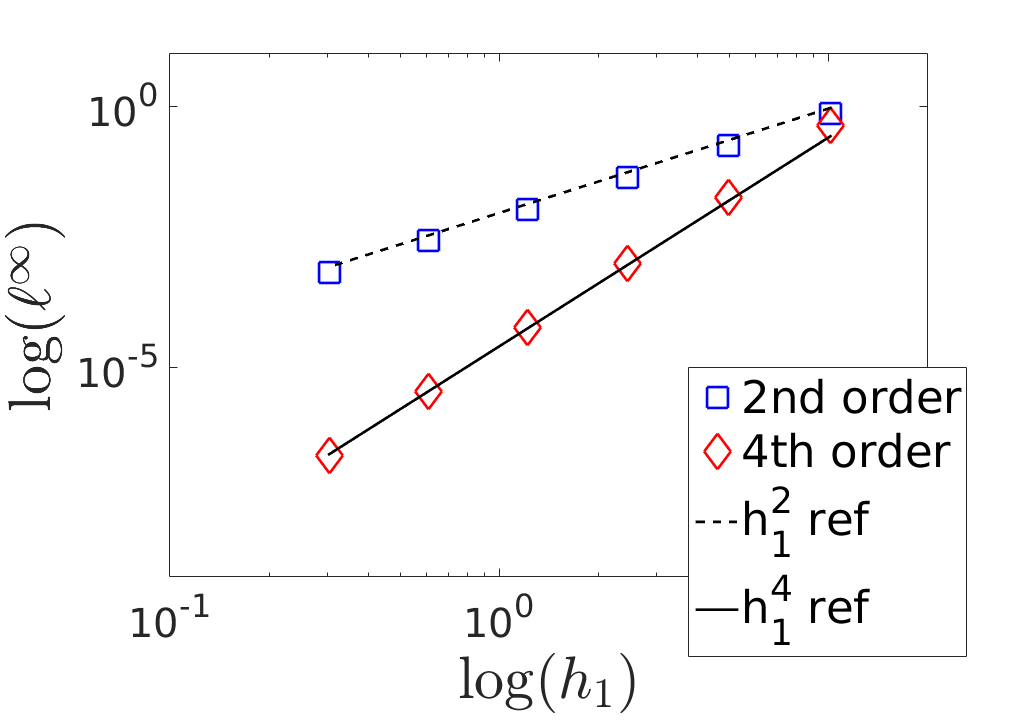} \hfill
\includegraphics[width=.325\textwidth]{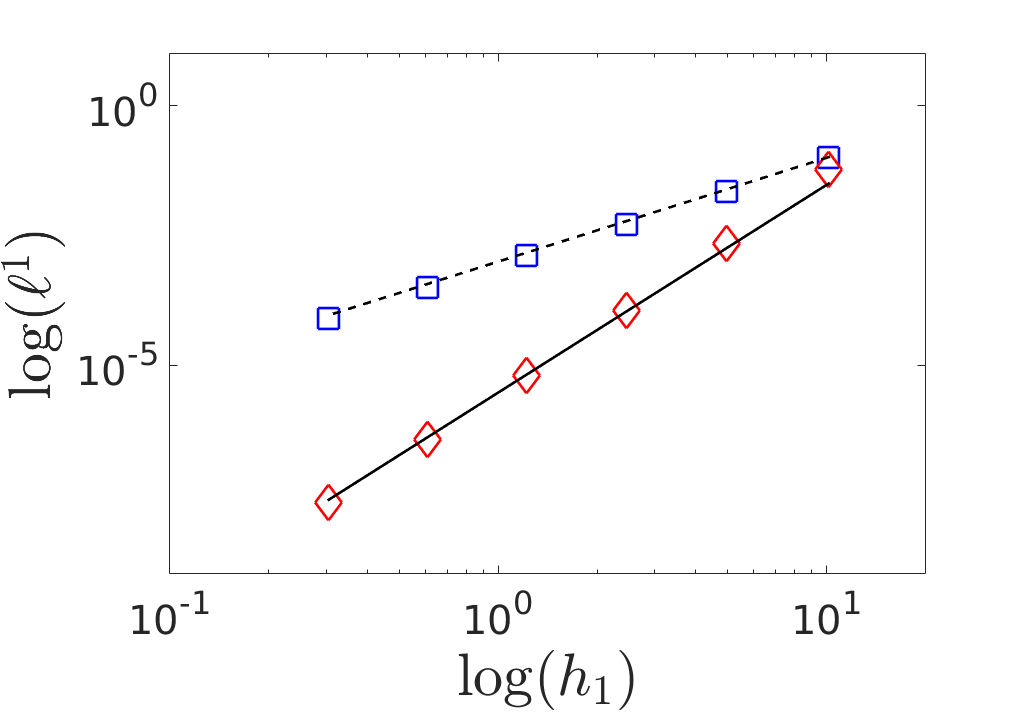} \hfill
\includegraphics[width=.325\textwidth]{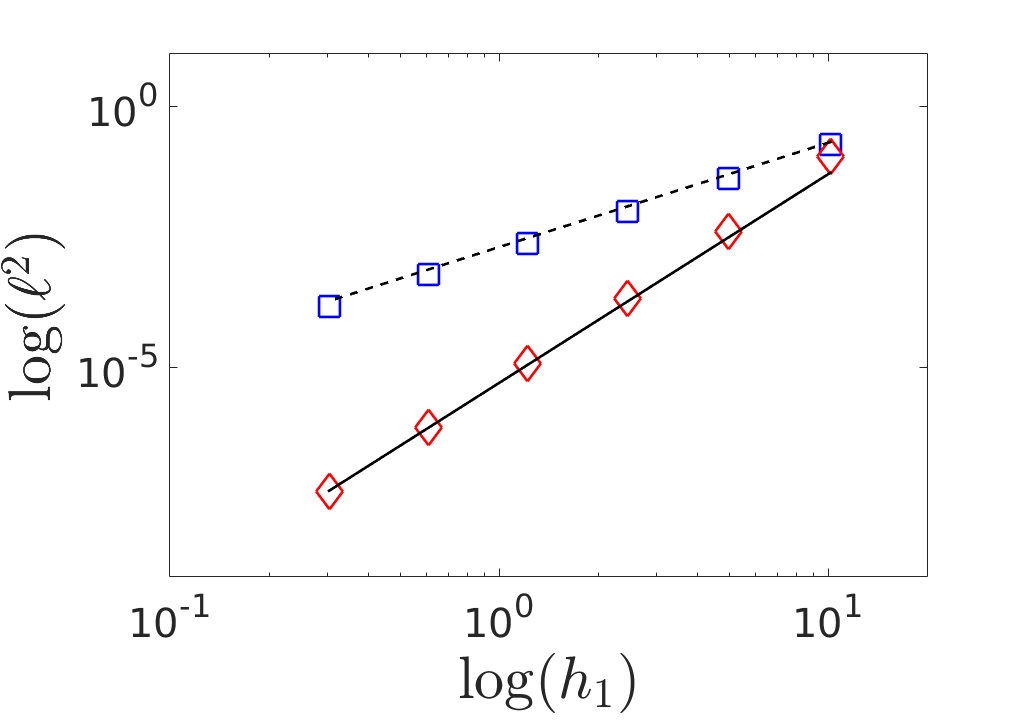} \\
\includegraphics[width=.325\textwidth]{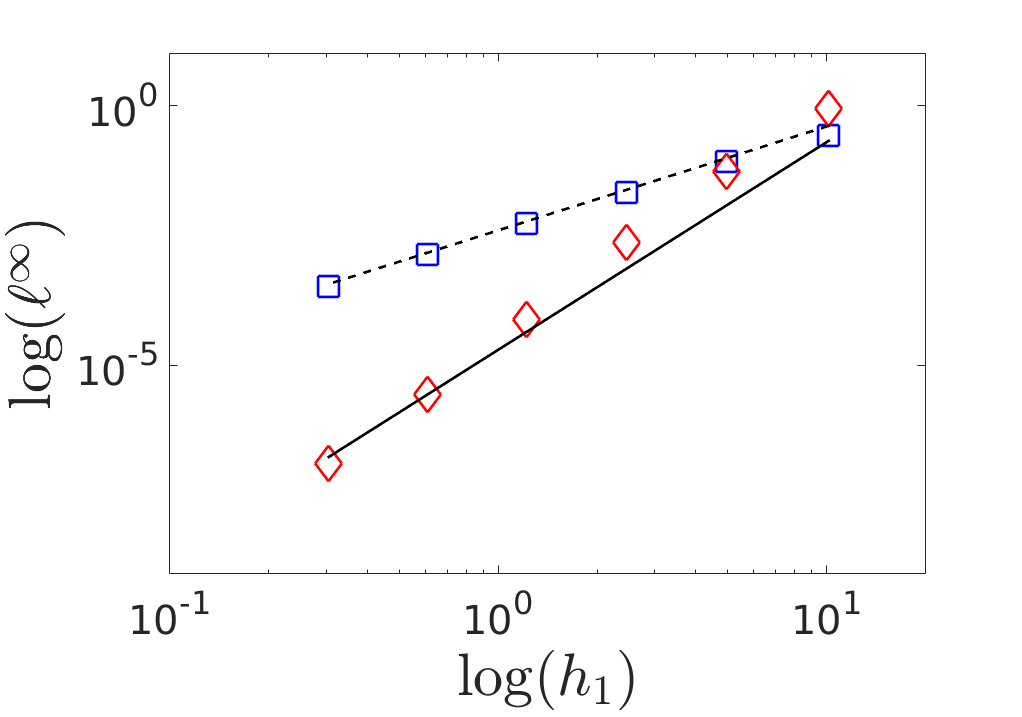} \hfill
\includegraphics[width=.325\textwidth]{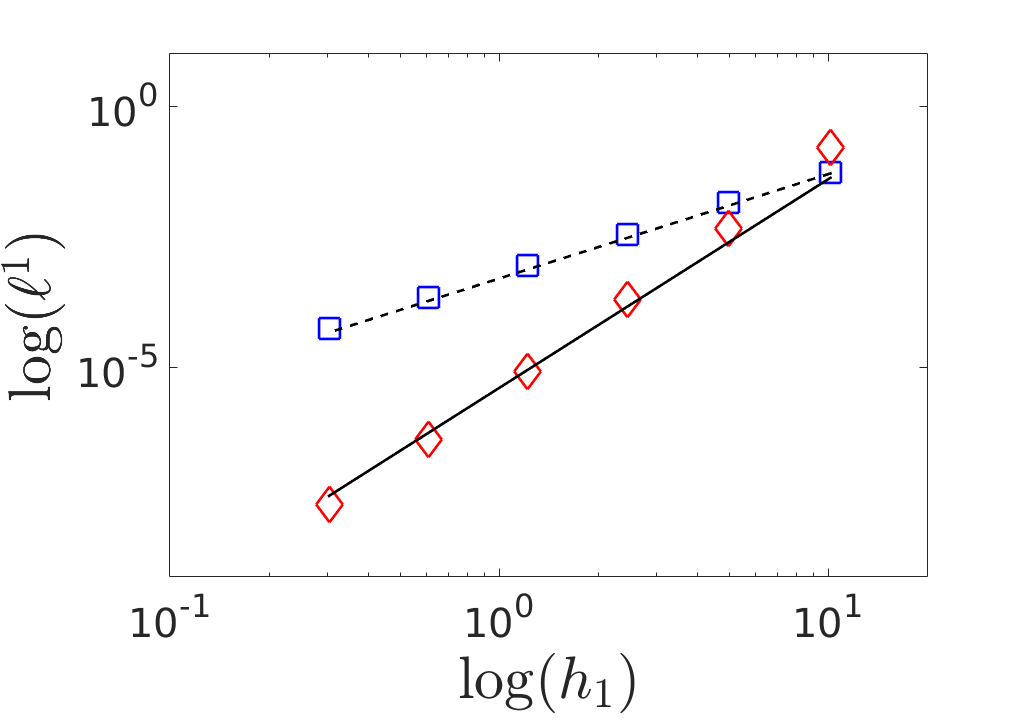} \hfill
\includegraphics[width=.325\textwidth]{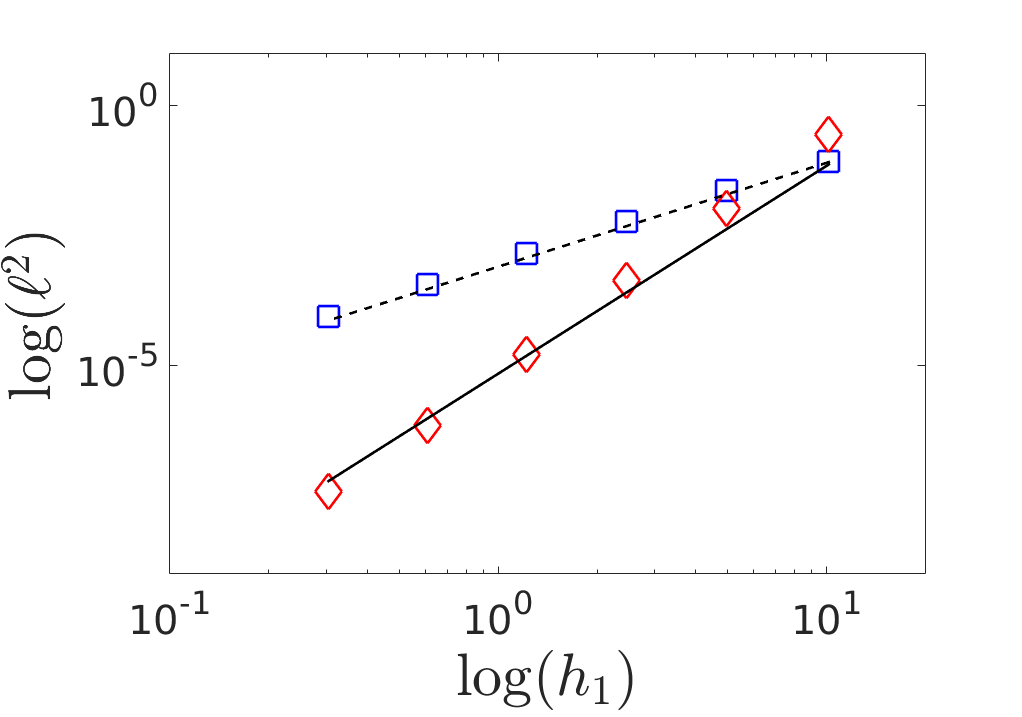}
\caption{Convergence rates for the two-dimensional surface plasmon polariton problem at an interface between vacuum and Drude silver medium with $\omega = 600$ THz (see Figure \ref{fig:plasmon_fields_1}). Here, the bounds \eqref{eq:RC2_time_step} and \eqref{eq:RC4_time_step}-\eqref{eq:RC4_GammaBound} were used to choose an appropriate time step which guarantees stability. From left to right, the top row shows $L_{\infty}$, $L_1$, and $L_2$ norms for the $E_x$ field. The corresponding $E_y$ rates are shown in the bottom row of plots. }
\label{fig:plasmon_rates}
\end{figure}

In the left domain $\Omega_1$, we consider vacuum with $\omega_{p,1} = \gamma_1 = 0$ and $\epsilon_{r,1} = \mu_1 = 1$. In the right domain $\Omega_2$, we again consider silver with $\epsilon_{r,2} = 5$, $\mu_2 = 1$, $\omega_{p,2} = 8.9$ eV, and $\gamma_2 = (17)^{-1}$ THz \cite{yang_2015}. The electric fields and their errors for temporal frequency $600$ THz are shown for the second- and fourth-order accurate schemes in Figure \ref{fig:plasmon_fields_1}. The surface plasmon polariton is localized to the interface at $x_{\text{mid}} = 0$ and propagates in the positive $y$ direction. Note again that the errors are smooth in each domain. The convergence rates with respect to the grid spacing $h_1$ for both the second- and fourth-order schemes are shown in Figure \ref{fig:plasmon_rates} using the $L_1, L_2$, and $L_{\infty}$ norms defined in \eqref{eqn:normdefs}, and again illustrate the expected order of convergence for all cases.

\section{Conclusions}

We have presented a scheme for the solution of the time-dependent Maxwell's equations in piecewise constant linear dispersive media where materials may be discontinuous across interfaces. Maxwell's equations are solved as a second-order wave equation with a time-history convolution term. The modified equation (Taylor) time-stepping approach is used to advance the electric field. The traditional recursive convolution approach is adapted to discretize the convolution term to high-order accuracy in time. This method may be used to generate a scheme which is of any even order of accuracy in time and space. The number of levels required by the full time-stepping scheme increases with the order of accuracy of the scheme, with the second-order and fourth-order schemes respectively requiring three and five levels. 

Second- and fourth-order accurate examples of this scheme, and sufficient conditions for their stability, were presented for the Drude dispersion model. In particular, depending on the boundary conditions implemented, the second-order scheme may admit weak exponential growth but remains conditionally stable based on a time-step restriction. Since this result is dependent on the treatment of the recursive convolution, the possibility of weak exponential growth may be removed by replacing the second-order recursive convolution update step with the fourth-order update. During the time-step, the fields are first advanced independently in each material domain. Coupling between the two domains was achieved through a discretization of the physical interface jump conditions in addition to extra numerical compatibility conditions which are derived to be consistent with the PDE and interface conditions. These coupling conditions are imposed after each interior update using fictitious (ghost) cells on the either side of the interface.

The second- and fourth-order versions of the scheme were implemented in cartesian coordinates for one and two spatial dimensions. Numerical results using the second-order accurate scheme for a 1D periodic wave were presented to illustrate the presence of the exponentially growing modes as well as numerical convergence. These results were also replicated with the fourth-order scheme to illustrate the absence of the exponentially growing mode along with the higher rate of convergence. Numerical results in two spatial dimensions were presented for two exact analytical solutions at a discontinuous material interface between a dielectric and a Drude metallic medium. These solutions were used to verify the accuracy of the schemes, as well as the numerical compatibility conditions.

The schemes presented here are amenable to extension to curvilinear grids and more complex material geometries via overlapping grids. This extension is left for future work. An additional avenue for future work is a similar investigation of the auxiliary differential equation (ADE) treatment of material dispersion, and in particular high-order accurate extensions in the presence of material jumps.

\section{Appendix}

\subsection{Scaling of the Governing Equations}

In this section, we summarize the rescaling of physical parameters required to obtain reasonable numerical parameters for computation. Consider the governing wave equation with a time-history term \eqref{eqn:waveE1} for the Drude model function $\eta(\tau)$ given in table \ref{table:etas}:
\begin{align}
&  \partial_t^2 \mathbf{E} = \frac{1}{\epsilon_0 \epsilon_{r} \mu } \Delta \mathbf{E} - \frac{ \omega_{p}^2 }{\epsilon_r} \mathbf{E} + \frac{   \omega_{p}^2  \gamma }{\epsilon_r} \int_0^{\infty} e^{- \gamma \tau }  \mathbf{E}(t - \tau) d \tau . \label{eqn:wavescale}
\end{align}
Let $\mu  \equiv \mu_0 \mu_r$,
\begin{align}
\widetilde{t} \equiv C_t \ t , \qquad  \widetilde{\mathbf{x}} \equiv \frac{ \mathbf{x} }{ C_t \sqrt{ \epsilon_0 \mu_0}}  , \qquad \widetilde{\mathbf{E}}( \widetilde{ \mathbf{x}}, \widetilde{t} ) \equiv \mathbf{E}(\mathbf{x},t), \qquad \widetilde{\omega}_p = \frac{\omega_p}{C_t} , \qquad {\text{and}} \qquad \widetilde{\gamma} =  \frac{ \gamma }{ C_t} .  
\end{align}
Equation \eqref{eqn:wavescale} becomes
\begin{align}
&  \partial_{\widetilde{t}}^2 \widetilde{\mathbf{E}} =  \frac{1} { \epsilon_r \mu_r } \Delta \widetilde{\mathbf{E}} - \frac{ \widetilde{\omega}_p^2 }{ \epsilon_r  } \widetilde{\mathbf{E}} + \frac{  \widetilde{\omega}_p^2 \widetilde{\gamma} }{\epsilon_r} \int_0^{\infty} e^{- \widetilde{\gamma} \tau }  \widetilde{\mathbf{E}}( \widetilde{t}  - \tau) d \tau . \label{eqn:wavescale2}
\end{align}
Here $c_0 = 1/\sqrt{\epsilon_0 \mu_0} = 2.99792458 \times10^8$ m/s is the speed of light in vacuum. Parameters $\omega_p$ and $\gamma$ are typically on the order of $1$ PHz and $10$ THz respectively  \cite{blaber_2009, olmon_2012, yang_2015}, which implies that a choice of $C_t \sim 10^{16}$ is a practical choice. Under this scaling, one obtains typical numerical parameters $\widetilde{\omega}_p \sim 10^{-1}$ and $\widetilde{\gamma} \sim 10^{-3}$. 

\subsection{Derivation of the Surface Plasmon Polariton Solution}
\label{app:polariton}

In this section, we derive the exact solution to Maxwell's equations which we use in section \ref{sec:polariton}. Consider an interface in two spatial dimensions $\mathbf{x} = (x,y)$ between a non-dispersive dielectric and a Drude dispersive medium located at the plane $x = x_{\rm mid}$. The material in the first domain, $\Omega_1 \equiv \{ \mathbf{x} = (x,y) : x \leq x_{\rm mid} \} $, has frequency domain material parameters $\epsilon_{r,1}, \mu_1, \omega_{p,1} = 0,$ and $ \gamma_1 = 0$, while the Drude material in the second domain, $\Omega_2 \equiv \{  \mathbf{x} = (x,y) : x \geq x_{\rm mid} \} $, has frequency domain material parameters $\epsilon_{r,2}, \mu_2, \omega_{p,2},$ and $ \gamma_2$. Following \cite{maier_2007}, we seek the TM (transverse-magnetic) surface plasmon polariton $(E_x, E_y, H_z)$, a mode localized to ({\it i.e.}, evanescent away from) the interface at $x = x_{\rm mid}$ in the $x$ direction and propagating in the $y$ direction. Note that due to the dissipation of the Drude material, the modes will also decay in the direction of propagation. In our second-order wave equation formulation given by \eqref{eqn:waveE1}, we solve the following equations respectively in domains $\Omega_1$ and $\Omega_2$:
\begin{subequations}
\begin{align}
& \partial_t^2 \mathbf{E}_1 = \frac{1}{\epsilon_0 \epsilon_{r,1} \mu_1} \Delta \mathbf{E}_1, \qquad \mathbf{x} \in \Omega_1, \label{eqn:wave_dom1} \\
&  \partial_t^2 \mathbf{E}_2 = \frac{1}{\epsilon_0 \epsilon_{r,2} \mu_2 } \Delta \mathbf{E}_2 - \frac{ \omega_{p,2}^2 }{\epsilon_r} \mathbf{E}_2 +
 \frac{   \omega_{p,2}^2  \gamma_2 }{\epsilon_r} \int_0^{\infty} e^{- \gamma_2 \tau }  \mathbf{E}_2(t - \tau) d \tau , \qquad \mathbf{x} \in \Omega_2. \label{eqn:wave_dom2}
\end{align}
\end{subequations}
These equations may alternately be written in the frequency domain as
\begin{subequations}
\begin{align}
& - \omega^2 \mu_1 \widehat{\epsilon_1}(\omega) \widehat{\mathbf{E}}_1 =  \Delta \widehat{\mathbf{E}}_1, \qquad \mathbf{x} \in \Omega_1, \label{eqn:wave_dom1_2} \\
& - \omega^2 \mu_2 \widehat{\epsilon_2}(\omega) \widehat{\mathbf{E}}_2 =  \Delta \widehat{\mathbf{E}}_2, \qquad \mathbf{x} \in \Omega_2, \label{eqn:wave_dom2_2} 
\end{align}
\end{subequations}
where
\begin{align}
\widehat{\epsilon_1}(\omega) = \epsilon_0 \epsilon_{r,1} , \qquad \qquad \widehat{\epsilon_2}(\omega) = \epsilon_0 \left( \epsilon_{r,1} - \frac{\omega_{p,2}^2}{\omega (i \gamma_2 - \omega) } \right). 
\end{align}
In addition, the divergence condition \eqref{eqn:macro2} in each domain gives
\begin{subequations}
\begin{align}
& \nabla \cdot \mathbf{E}_1  = \partial_x E_{x,1} + \partial_y E_{y,1} = 0, \qquad \mathbf{x} \in \Omega_1 , \label{eqn:tez1gauss} \\
& \nabla \cdot \mathbf{E}_2  = \partial_x E_{x,2} + \partial_y E_{y,2} = 0, \qquad \mathbf{x} \in \Omega_2 . \label{eqn:tez2gauss}
\end{align}
 \label{eqn:tezgauss}
\end{subequations}
Finally, the interface conditions \eqref{eqn:jumps2} at the plane $x = x_{ \text{mid}}$ yield
\begin{align}
\left[ \mathbf{n} \cdot \mathbf{D} \right]_{\mathcal{I}} = 0, \qquad \left[ H_z \right]_{\mathcal{I}} = 0, \qquad \left[ \mathbf{n} \times \mathbf{E} \right]_{\mathcal{I}} = 0. 
\end{align}

We assume the ansatz
\begin{subequations}
\begin{align}
E_{x,1} = A_{x,1} \ e^{\alpha_1 x} e^{i \beta y} e^{i \omega t} , \\
E_{y,1} = A_{y,1} \ e^{\alpha_1 x} e^{i \beta y} e^{i \omega t} , 
\end{align}
 \label{eqn:plasmon1}
\end{subequations}
for the first domain $\mathbf{x} \in \Omega_1$ and 
\begin{subequations}
\begin{align}
E_{x,2} = A_{x,2} \ e^{\alpha_2 x} e^{i \beta y} e^{i \omega t} , \\
E_{y,2}  = A_{y,2} \  e^{\alpha_2 x} e^{i \beta y} e^{i \omega t} ,
\end{align}
 \label{eqn:plasmon2}
\end{subequations}
for the second domain $\mathbf{x} \in \Omega_2$. Equations \eqref{eqn:tez1gauss} and \eqref{eqn:tez2gauss} respectively give
\begin{align}
A_{y,1} = \frac{ i \alpha_1 }{\beta} A_{x,1}, \qquad A_{y,2} = \frac{ i \alpha_2 }{\beta} A_{x,2}.  \label{eqn:AyfromAx}
\end{align}
Similarly, Equations \eqref{eqn:wave_dom1_2} and \eqref{eqn:wave_dom2_2} become
\begin{subequations}
\begin{align}
&& - \omega^2 \mu_1  \widehat{\epsilon_1}(\omega) & = \left( \alpha_1^2 - \beta^2 \right) ,  && \\
&& - \omega^2 \mu_2 \widehat{\epsilon_2}(\omega) & = \left( \alpha_2^2 - \beta^2 \right).&&
\end{align}
\label{eqn:goal1}
\end{subequations}
Next, the interface conditions, $\left[ \mathbf{n} \times \mathbf{E} \right]_{\mathcal{I}} = \left[ E_y \right]_{\mathcal{I}} = 0$ along with \eqref{eqn:AyfromAx} give
\begin{align}
A_{y,1}  = \frac{ i \alpha_1}{\beta} A_{x,1} = A_{y,2} =   \frac{ i \alpha_2}{\beta}  A_{x,2} \qquad \Longrightarrow \qquad \alpha_1 A_{x,1} = \alpha_2 A_{x,2}. \label{eqn:Axs1}
\end{align}
In the frequency-domain, the interface condition $\left[ \mathbf{n} \cdot \mathbf{D} \right]_{\mathcal{I}} =  \left[ D_x \right]_{\mathcal{I}} = 0 $ becomes
\begin{align}
\left[ \mathbf{n} \cdot \widehat{\epsilon}(\omega) \widehat{\mathbf{E}} \right]_{\mathcal{I}} = \left[ \widehat{\epsilon}(\omega) \widehat{E}_x \right]_{\mathcal{I}} =  0,
\end{align}
from whence
\begin{align}
\widehat{\epsilon_1}(\omega) A_{x,1} =   \widehat{\epsilon_2}(\omega) A_{x,2} . \label{eqn:Axs2}
\end{align}
Equations \eqref{eqn:Axs1} and \eqref{eqn:Axs2} now give
\begin{align}
\frac{ \widehat{\epsilon_2}(\omega) }{ \widehat{\epsilon_1}(\omega) } = \frac{\alpha_2}{\alpha_1} . \label{eqn:goal2}
\end{align}
Combining \eqref{eqn:goal1} and \eqref{eqn:goal2} gives the dispersion relation
\begin{align}
\beta =  \omega \left( \frac{ \widehat{\epsilon_1}(\omega) \widehat{\epsilon_2}(\omega) \left[ \mu_1 \widehat{\epsilon_2}(\omega) - \mu_2 \widehat{\epsilon_1}(\omega) \right] }{ \widehat{\epsilon_2}(\omega)^2 - \widehat{\epsilon_1}(\omega)^2 } \right)^{1/2} . 
\end{align}
In addition, we obtain
\begin{subequations}
\begin{align}
& \alpha_1 = \left(  \beta^2 - \omega^2 \mu_1 \widehat{\epsilon_1}(\omega) \right)^{1/2} =  \omega \widehat{\epsilon_1}(\omega) \left(  - \frac{ \left[ \mu_2 \widehat{\epsilon_2}(\omega) - \mu_1 \widehat{\epsilon_1}(\omega) \right] }{ \widehat{\epsilon_2}(\omega)^2 - \widehat{\epsilon_1}(\omega)^2 } \right)^{1/2} , \\
& \alpha_2 =  \left(  \beta^2 - \omega^2 \mu_2 \widehat{\epsilon_2}(\omega) \right)^{1/2} = \omega \widehat{\epsilon_2}(\omega) \left(  - \frac{ \left[ \mu_2 \widehat{\epsilon_2}(\omega) - \mu_1 \widehat{\epsilon_1}(\omega) \right] }{ \widehat{\epsilon_2}(\omega)^2 - \widehat{\epsilon_1}(\omega)^2 } \right)^{1/2}  = \frac{\alpha_1 \widehat{\epsilon_2}(\omega)}{\widehat{\epsilon_1}(\omega)} . 
\end{align}
\end{subequations}
Finally, the solutions given by \eqref{eqn:plasmon1} and \eqref{eqn:plasmon2} along with \eqref{eqn:AyfromAx} are simply
\begin{subequations}
\begin{align}
& E_{x,1} = A_{x,1} \ e^{\alpha_1 x} e^{i \beta y} e^{i \omega t} , \\
& E_{y,1} =  \frac{i \alpha_1}{\beta}  \ A_{x,1} \ e^{\alpha_1 x} e^{i \beta y} e^{i \omega t} , \\
& E_{x,2} = \frac{\widehat{\epsilon_1}(\omega)}{\widehat{\epsilon_2}(\omega)} \ A_{x,1}  \  e^{\alpha_2 x} \  e^{i \beta y} e^{i \omega t} , \\
& E_{y,2} = \frac{i \alpha_1}{\beta} \ A_{x,1} \ e^{\alpha_2 x} e^{i \beta y} e^{i \omega t} . 
\end{align}
\end{subequations}

\newpage
\bibliography{DMX_RC}

\begin{thebibliography}{10}

\bibitem{banks_2009}
{\sc H.~Banks, V.~Bokil, and N.~Gibson}, {\em Analysis of stability and
  dispersion in a finite element method for {D}ebye and {L}orentz dispersive
  media}, Numerical Methods for Partial Differential Equations, 25 (2009).

\bibitem{blaber_2009}
{\sc M.~G. Blaber, M.~D. Arnold, and M.~J. Ford}, {\em Search for the ideal
  plasmonic nanoshell: The effects of surface scattering and alternatives to
  gold and silver}, The Journal of Physical Chemistry C, 113 (2009),
  pp.~3041--3045, \url{https://doi.org/10.1021/jp810808h},
  \url{http://dx.doi.org/10.1021/jp810808h},
  \url{https://arxiv.org/abs/http://dx.doi.org/10.1021/jp810808h}.

\bibitem{chun_2013}
{\sc K.~Chun, H.~Kim, H.~Kim, and Y.~Chung}, {\em {PLRC} and {ADE}
  implementations of {D}rude-critical point dispersive model for the {FDTD}
  method}, Progress in Electromagnetics Research, 135 (2013), pp.~373--390.

\bibitem{deinega_2007}
{\sc A.~Deinega and I.~Valuev}, {\em Subpixel smoothing for conductive and
  dispersive media in the finite-difference time-domain method}, Optics
  Letters, 32 (2007).

\bibitem{ditkowski_2001}
{\sc A.~Ditkowski, K.~Dridi, and J.~Hesthaven}, {\em Convergent {C}artesian
  grid methods for {M}axwell's equations in complex geometries}, Journal of
  Computational Physics, 170 (2001), pp.~39 -- 80,
  \url{https://doi.org/http://dx.doi.org/10.1006/jcph.2001.6719},
  \url{http://www.sciencedirect.com/science/article/pii/S0021999101967191}.

\bibitem{fornberg96}
{\sc B.~Fornberg}, {\em A Practical Guide to Pseudospectral Methods}, Cambridge
  University Press, Cambridge, 1996.

\bibitem{gandhi_1993}
{\sc O.~P. Gandhi, B.-Q. Gao, and J.-Y. Chen}, {\em A frequency-dependent
  finite-difference time-domain formulation for general dispersive media}, IEEE
  Transactions on Microwave Theory and Techniques, 41 (1993).

\bibitem{gedney_2012}
{\sc S.~D. Gedney, J.~C. Young, T.~C. Kramer, and J.~A. Roden}, {\em A
  discontinuous {G}alerkin finite element time-domain method modeling of
  dispersive media}, IEEE Transactions on Antennas and Propagation, 60 (2012),
  pp.~1969--1977.

\bibitem{henshaw_2006}
{\sc W.~Henshaw}, {\em A high-order accurate parallel solver for {M}axwell's
  equations on overlapping grids}, SIAM J. Sci. Comput., 28 (2006),
  pp.~1730--1765.

\bibitem{jackson_1962}
{\sc J.~D. Jackson}, {\em Classical Electrodynamics}, John Wiley \& Sons Ltd.,
  1962.

\bibitem{jiao_2001}
{\sc D.~Jiao and J.-M. Jin}, {\em Time-domain finite-element modeling of
  dispersive media}, IEEE Microwave and Wireless Components Letters, 11 (2001),
  pp.~220--222.

\bibitem{jin_1993}
{\sc J.~Jin}, {\em The Finite Element Method in Electromagnetics}, John Wiley
  and Sons, New York, 1993.

\bibitem{joseph_1991}
{\sc R.~M. Joseph, S.~C. Hagness, and A.~Taflove}, {\em Direct time integration
  of {M}axwell's equations in linear dispersive media with absorption for
  scattering and propagation of femtosecond electromagnetic pulses}, Optics
  Letters, 16 (1991).

\bibitem{kashiwa_1990_2}
{\sc T.~Kashiwa and I.~Fukai}, {\em A treatment by the {FD-TD} method of the
  dispersive characteristics associated with electronic polarization},
  Microwave and Optical Technology Letters, 3 (1990).

\bibitem{kashiwa_1990_3}
{\sc T.~Kashiwa, Y.~Ohtomo, and I.~Fukai}, {\em A finite-difference time-domain
  formulation for transient propagation in dispersive media associated with
  {C}ole-{C}ole's circular arc law}, Microwave and Optical Technology Letters,
  3 (1990).

\bibitem{kashiwa_1990}
{\sc T.~Kashiwa, N.~Yoshida, and I.~Fukai}, {\em A treatment by the
  finite-difference time-domain method of the dispersive characteristics
  associated with orientation polarization}, IEICE Transactions, E73 (1990).

\bibitem{kelley_1996}
{\sc D.~Kelley and R.~Luebbers}, {\em Piecewise linear recursive convolution
  for dispersive media using {FDTD}}, IEEE Transactions on Antennas and
  Propagation, 44 (1996), pp.~792--797.

\bibitem{kreiss_2006}
{\sc H.-O. Kreiss and N.~A. Petersson}, {\em A second order accurate embedded
  boundary method for the wave equation with {D}irichlet data}, SIAM J. Sci.
  Comput., 27 (2006).

\bibitem{li_2006}
{\sc J.~Li}, {\em Error analysis of finite element methods for 3-{D}
  {M}axwell's equations in dispersive media}, J. Comput. Appl. Math., 188
  (2006), pp.~107--120, \url{https://doi.org/10.1016/j.cam.2005.03.060},
  \url{http://dx.doi.org/10.1016/j.cam.2005.03.060}.

\bibitem{li_2008}
{\sc J.~Li and Y.~Chen}, {\em Finite element study of time-dependent
  {M}axwell's equations in dispersive media}, Numerical Methods for Partial
  Differential Equations, 24 (2008), pp.~1203--1221,
  \url{https://doi.org/10.1002/num.20314},
  \url{http://dx.doi.org/10.1002/num.20314}.

\bibitem{liu_2012}
{\sc J.~Liu, M.~Brio, and J.~V. Moloney}, {\em Subpixel smoothing
  finite-difference time-domain method for material interface between
  dielectric and dispersive media}, Optics Letters, 37 (2012).

\bibitem{lu_2004}
{\sc T.~Lu, P.~Zhang, and W.~Cai}, {\em Discontinuous {G}alerkin methods for
  dispersive and lossy {M}axwell's equations and {PML} boundary conditions},
  Journal of Computational Physics, 200 (2004), pp.~549 -- 580,
  \url{https://doi.org/http://dx.doi.org/10.1016/j.jcp.2004.02.022},
  \url{http://www.sciencedirect.com/science/article/pii/S0021999104001792}.

\bibitem{luebbers_1990}
{\sc R.~Luebbers, F.~P. Hunsberger, K.~S. Kunz, R.~B. Standler, and
  M.~Schneider}, {\em A frequency-dependent finite-difference time-domain
  formulation for dispersive materials}, IEEE Transactions on Electromagnetic
  Compatibility, 32 (1990), pp.~222--227.

\bibitem{luebbers_1992}
{\sc R.~J. Luebbers and F.~Hunsberger}, {\em {FDTD} for {N}th-order dispersive
  media}, IEEE Transactions on Antennas and Propagation, 40 (1992).

\bibitem{luebbers_1991}
{\sc R.~J. Luebbers, F.~Hunsberger, and K.~S. Kunz}, {\em A frequency-dependent
  finite-difference time-domain formulation for transient propagation in
  plasma}, IEEE Transactions on Antennas and Propagation, 39 (1991).

\bibitem{maier_2007}
{\sc S.~A. Maier}, {\em Plasmonics: Fundamentals and Applications}, Springer
  US, 2007.

\bibitem{mohammadi_2005}
{\sc A.~Mohammadi, H.~Nadgaran, and M.~Agio}, {\em Contour-path effective
  permittivities for the two-dimensional finite-difference time-domain method},
  Optics Express, 13 (2005).

\bibitem{nedelec_1980}
{\sc J.~Nedelec}, {\em Mixed finite elements in $\mathbb{R}^3$}, Numer. Math.,
  35 (1980).

\bibitem{nguyen_2014_2}
{\sc D.~D. Nguyen and S.~Zhao}, {\em High order {FDTD} methods for transverse
  magnetic modes with dispersive interfaces}, Applied Mathematics and
  Computation,  (2014), pp.~699--707.

\bibitem{nguyen_2014}
{\sc D.~D. Nguyen and S.~Zhao}, {\em Time-domain matched interface and boundary
  ({MIB}) modeling of {D}ebye dispersive media with curved interfaces}, Journal
  of Computational Physics,  (2014), pp.~298--325.

\bibitem{nguyen_2015}
{\sc D.~D. Nguyen and S.~Zhao}, {\em A new high order dispersive {FDTD} method
  for {D}rude material with complex interfaces}, Journal of Computational and
  Applied Mathematics, 285 (2015).

\bibitem{nguyen_2016}
{\sc D.~D. Nguyen and S.~Zhao}, {\em A second order dispersive {FDTD} algorithm
  for transverse electric {M}axwell's equations with complex interfaces},
  Computers and Mathematics with Applications,  (2016), pp.~1010--1035.

\bibitem{olmon_2012}
{\sc R.~L. Olmon, B.~Slovick, T.~W. Johnson, D.~Shelton, S.-H. Oh, G.~D.
  Boreman, and M.~B. Raschke}, {\em Optical dielectric function of gold}, Phys.
  Rev. B, 86 (2012).

\bibitem{pitarke_2007}
{\sc J.~M. Pitarke, V.~M. Silkin, E.~V. Chulkov, and P.~M. Echenique}, {\em
  Theory of surface plasmons and surface-plasmon polaritons}, Reports on
  Progress in Physics, 70 (2007), p.~1,
  \url{http://stacks.iop.org/0034-4885/70/i=1/a=R01}.

\bibitem{popovic_2003}
{\sc D.~Popovic and M.~Okoniewski}, {\em Effective permittivity at the
  interface of dispersive dielectrics in {FDTD}}, IEEE Microwave and Wireless
  Components Letters, 13 (2003).

\bibitem{prokopeva_2011_2}
{\sc L.~J. Prokopeva, J.~D. Borneman, and A.~V. Kildishev}, {\em Optical
  dispersion models for time-domain modeling of metal-dielectric
  nanostructures}, IEEE Transactions on Magnetics, 47 (2011).

\bibitem{prokopeva_2011}
{\sc L.~J. Prokopeva, J.~Trieschmann, T.~A. Klar, and A.~V. Kildishev}, {\em
  Numerical modeling of active plasmonic metamaterials}, Proc. of SPIE, 8172
  (2011).

\bibitem{prokopidis_2004}
{\sc K.~Prokopidis, E.~Kosmidou, and T.~Tsiboukis}, {\em An {FDTD} algorithm
  for wave propagation in dispersive media using higher-order schemes}, Journal
  of Electromagnetic Waves and Applications, 18 (2004), pp.~1171--1194.

\bibitem{prokopidis_2006}
{\sc K.~Prokopidis and T.~Tsiboukis}, {\em Higher-order spatial {FDTD} schemes
  for {EM} propagation in dispersive media}, in Electromagnetic fields in
  mechatronics, electrical and electronic engineering: proceedings of ISEF'05,
  vol.~27, Studies in Applied Electromagnetics and Mechanics, 2006,
  pp.~240--246.

\bibitem{rodrigue_2001}
{\sc G.~Rodrigue and D.~White}, {\em A vector finite element time-domian method
  for solving {M}axwell's equations on unstructured hexahedral grids}, SIAM J.
  Sci. Comput., 23 (2001).

\bibitem{sullivan_1992}
{\sc D.~M. Sullivan}, {\em Frequency-dependent {FDTD} methods using
  {Z}-transforms}, IEEE Transactions on Antennas and Propagation, 40 (1992).

\bibitem{taflove_2005}
{\sc A.~Taflove and S.~C. Hagness}, {\em Computational electrodynamics: the
  finite-difference time-domain method}, Artech House, Norwood, MA, 3~ed.,
  2005.

\bibitem{thomas1999}
{\sc J.~W. Thomas}, {\em Numerical Partial Differential Equations: Conservation
  Laws and Elliptic Equations}, Springer-Verlag, New York, 1999.

\bibitem{yang_2015}
{\sc H.~U. Yang, J.~D'Archangel, M.~L. Sundheimer, G.~D.~B. Eric~Tucker, and
  M.~B. Raschke}, {\em Optical dielectric function of silver}, Phys. Rev. B, 91
  (2015).

\bibitem{yee_1966}
{\sc K.~Yee}, {\em Numerical solution of initial value problems of {M}axwells
  equations}, IEEE Transactions on Antennas and Propagation, 14 (1966).

\bibitem{young_1996}
{\sc J.~L. Young}, {\em A higher order {FDTD} method for {EM} propagation in a
  collisionless cold plasma}, IEEE Transactions on Antennas and Propagation, 44
  (1996).

\bibitem{young_1997}
{\sc J.~L. Young, D.~Gaitonde, and J.~Shang}, {\em Toward the construction of a
  fourth-order difference scheme for transient {EM} wave simulation: Staggered
  grid approach}, IEEE Transactions on Antennas and Propagation, 45 (1997),
  pp.~1573--1580.

\bibitem{young_1995}
{\sc J.~L. Young, A.~Kittichartphayak, Y.~M. Kwok, and D.~Sullivan}, {\em On
  the dispersion errors related to {(FD)$^2$TD} type schemes}, IEEE
  transactions on microwave theory and techniques, 43 (1995), pp.~1902--1910.

\bibitem{zhang_2016}
{\sc Y.~Zhang, D.~D. Nguyen, K.~Du, J.~Xu, and S.~Zhao}, {\em Time-domain
  numerical solutions of {M}axwell interface problems with discontinuous
  electromagnetic waves}, Advances in Applied Mathematics and Mechanics, 8
  (2016), pp.~353--385.

\bibitem{zhao_2004}
{\sc S.~Zhao and G.~Wei}, {\em High-order {FDTD} methods via derivative
  matching for {M}axwells equations with material interfaces}, Journal of
  Computational Physics, 200 (2004).

\end{thebibliography}
\bibliographystyle{siamplain}

\end{document}